\documentclass[12pt]{amsart}

\usepackage[dvipsnames]{xcolor}
\usepackage{tikz-cd}
\usepackage{graphicx}
\usepackage{amsmath}
\usepackage{color}
\usepackage{caption}
\usepackage{subcaption}
\usepackage{amsfonts,amssymb,amscd}
\usepackage{mathrsfs}
\usepackage{bm}
\usepackage{verbatim} 

\usepackage[toc,page]{appendix}

\calclayout

\newtheoremstyle{remboldstyle}
  {}{}{\itshape}{}{\bfseries}{.}{.5em}{{\thmname{#1 }}{\thmnumber{#2}}{\thmnote{ (#3)}}}
\theoremstyle{remboldstyle}

\usepackage[outdir=./]{epstopdf}

\newtheorem{thm}{Theorem}[section]
\newtheorem{prop}[thm]{Proposition}
\newtheorem{lem}[thm]{Lemma}
\newtheorem{cor}[thm]{Corollary}

\newtheorem{thmx}{Theorem}

\newcommand{\etalchar}[1]{$^{#1}$}

\theoremstyle{definition}
\newtheorem{definition}[thm]{Definition}

\newtheorem{rem}[thm]{Remark}
\newtheorem{notation}[thm]{Notation}

\numberwithin{equation}{section}

\catcode`\@=11
\newdimen\cdsep
\def\cdstrut{\vrule height .6\cdsep width 0pt depth .4\cdsep}
\def\@cdstrut{{\advance\cdsep by 2em\cdstrut}}

\def\arrow#1#2{
  \ifx d#1
    \llap{$\scriptstyle#2$}\left\downarrow\cdstrut\right.\@cdstrut\fi
  \ifx u#1
    \llap{$\scriptstyle#2$}\left\uparrow\cdstrut\right.\@cdstrut\fi
  \ifx r#1
    \mathop{\hbox to \cdsep{\rightarrowfill}}\limits^{#2}\fi
  \ifx l#1
    \mathop{\hbox to \cdsep{\leftarrowfill}}\limits^{#2}\fi
}
\catcode`\@=12

\newcommand{\myvec}[1]{{\boldsymbol{#1}}}   

\def\re{{\rm{Re}}}
\def\im{{\rm{Im}}}
\def\uhp{{\mathbb H}}
\def\zbar{{\overline{z}}}

\def\Bbb{\mathbb}
\def\reals{\Bbb R}
\def\complex{\Bbb C}
\def\disk{\Bbb D}
\def\circle{\Bbb T}
\def\naturals{\Bbb N}

\newcommand{\Chat}{\widehat{\mathbb{C}}}

\DeclareMathOperator{\diam}{diam} 

\DeclareMathOperator{\dist}{dist} 

\makeatletter
\@namedef{subjclassname@2020}{\textup{2020} Mathematics Subject Classification}
\makeatother

\begin{document}


\title[On the Shapes of Rational Lemniscates]{On the Shapes of Rational Lemniscates}



\author {Christopher J. Bishop}
\address{C.J. Bishop\\
         Mathematics Department\\
         Stony Brook University \\
         Stony Brook, NY 11794-3651}
\email {bishop@math.stonybrook.edu}
\author {Alexandre Eremenko}
\address{Alexandre Eremenko \\ 
Mathematics Department, Purdue University, West Lafayette, IN 47907}
\email {eremenko@purdue.edu}
\author{Kirill Lazebnik}
\address{Kirill Lazebnik\\
Mathematics Department \\
University of Texas at Dallas \\
Richardson, TX, 75080}
\email{Kirill.Lazebnik@utdallas.edu}
\thanks{\noindent The first author is supported in part by NSF grant DMS-2303987, and the third author is supported in part by NSF grant DMS-2452130.}





\begin{abstract} A rational lemniscate is a level set of 
$|r|$ where $r: \Chat\rightarrow\Chat$ is rational. 
We prove that any planar Euler graph  
can be approximated, in a strong sense, by a  homeomorphic 
rational lemniscate.  This  generalizes Hilbert's 
lemniscate theorem; he proved  that any Jordan 
curve  can be approximated (in the same strong sense) by a 
polynomial lemniscate that is also a Jordan curve.
As consequences, we  obtain 
a sharp  quantitative  version of the classical 
Runge's theorem on rational approximation, and we give a new result 
on the approximation of planar continua by Julia sets 
of rational maps.
\end{abstract}

\date{\today}

\maketitle


\section{Introduction}

\subsection{Rational lemniscates and  Euler graphs}

 \begin{definition} A \emph{rational lemniscate} is a set of the form 
\begin{equation*}\label{lemn_defn} 
	L_r(c) :=\{z \in \Chat: |r(z)|=c\},
\end{equation*}
where $0< c<\infty$, $r$ is a rational function and $\Chat$ is
the Riemann sphere; in other words, a rational lemniscate
is a level set of $|r|$. The constant $c$ is often omitted from
the notation, since by rescaling $r$ we can always take $c=1$, 
and for brevity we will write $L_r = L_r(1)$. 
\end{definition}

\begin{definition}\label{embeddedgraph} 
A  \emph{lemniscate graph} is a set 
$G\subset \Chat$ so that there is a finite set 
$V\subset G$ (called the \emph{vertices} of $G$), so that:
\begin{enumerate}
\item $G\setminus V$ has finitely many components 
(these are called the \emph{edges} of $G$), 
each of which is either a (closed) Jordan curve, 
or else an (open) simple arc $\gamma$ satisfying
$\overline{\gamma}\setminus\gamma\subset V$. 
\item The degree of each vertex is even and at least four, 
where the \emph{degree} of a vertex $v$ is defined
as the number of edges $\gamma$ satisfying
$v\in\overline{\gamma}\setminus\gamma$, and we count 
an edge $\gamma$ twice if $\{v\}=\overline{\gamma}\setminus\gamma$. 
\end{enumerate}
\end{definition}

See Figure \ref{top_lem_coloring}.
It is not hard to prove that every rational lemniscate is 
a lemniscate graph (see Proposition \ref{toplemimpliesgraph}) 
and our main result will show that every lemniscate graph is 
homeomorphic to, and can be approximated by, a rational 
lemniscate.  Before stating the precise result, we need a 
few more definitions.

A lemniscate graph need not be a graph in the usual sense,
since it can have 
Jordan curve components with no vertices (see Figure \ref{top_lem_coloring}).
However, if we add a vertex (of degree two)  to each such curve component,
we create an Euler graph in the usual sense
(an Euler graph is one where every vertex has even degree; 
an Eulerian graph is a connected Euler graph). 
Thus as closed subsets of the plane (forgetting the vertex/edge
structure), lemniscate graphs and Euler graphs are the same.
In particular, the faces  of the graph (that is, the  connected components 
of the complement of the graph) are the same. 

\begin{figure}[htb]
\centerline{
\includegraphics[height=1.8in]{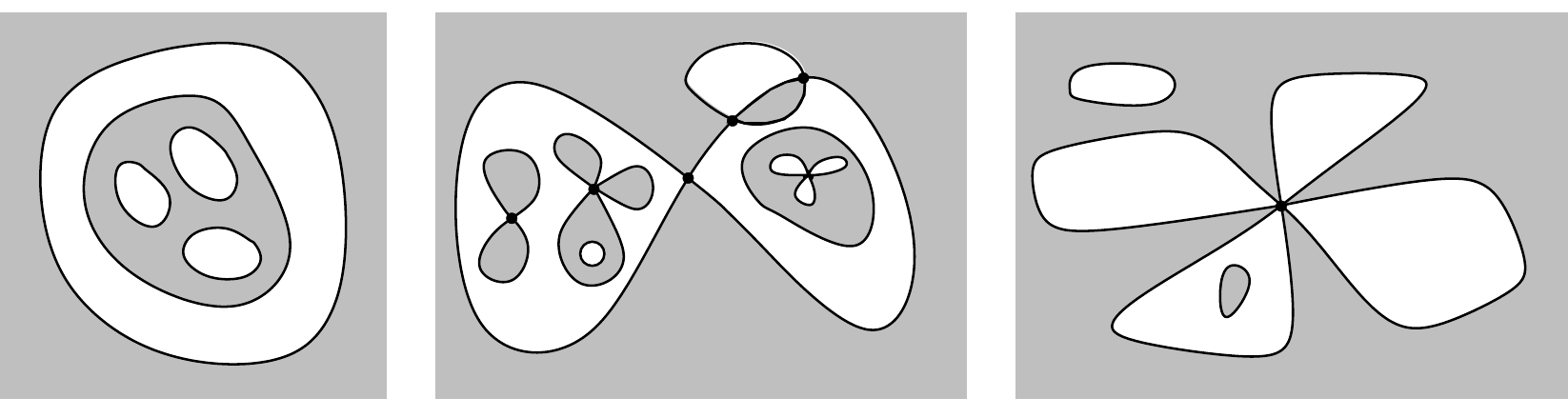}
}
\caption{Pictured are three examples of lemniscate graphs, together with 2-colorings of their faces. For each example, the unbounded face is colored grey. The leftmost example consists of five disjoint Jordan curves and has no vertices. The center example has six connected components and six vertices (two of degree six and four of degree four). The right hand example has three connected components and one vertex of degree eight.}
\label{top_lem_coloring}
\end{figure}

\begin{definition}\label{coloring_defn} 
Let $G$ be a lemniscate graph. 
	A  {\it $2$-coloring} of the faces of $G$  assigns one of two 
colors to each face (we will use white and grey) so 
that any two faces sharing a common edge have different colors. 
\end{definition}

The fact that the faces of a planar Euler 
graph can be $2$-colored is well known in graph theory.
Moreover, there are exactly two such colorings, obtained 
from each other by swapping the colors. 
Similarly, the faces of a rational lemniscate  have  a natural 
2-coloring  where  the components of  $\{z\in\Chat: |r(z)|<1\}$
are colored white, and the components of 
$\{z \in \Chat: |r(z) | >1\}$ are 
colored grey (we can swap colors by replacing $r$ by $1/r$). 
Clearly the  
poles of $r$ must lie in the grey components. It turns out
that this is the only restriction  for the 
following type of approximation to hold.

\begin{definition} \label{epsilon homeo defn}
Let $\varepsilon>0$. Two sets $E$, $F\subset\Chat$ 
are said to be $\varepsilon$\emph{-homeomorphic} 
if there exists a homeomorphism 
$\phi: \Chat\rightarrow\Chat$ satisfying
$\phi(E)=F$ and $\sup_{z\in\Chat}d(\phi(z),z)<\varepsilon$,
where $d(\cdot, \cdot)$ denotes the spherical metric on $\Chat$.
\end{definition}

\begin{thmx}\label{main_thm_2} 
Let $G$ be a lemniscate graph, let  $\varepsilon>0$,
fix a $2$-coloring of the faces of $G$, and suppose
that $P\subset\Chat$ contains exactly one point in each grey face of $G$.
Then there exists a rational mapping 
$r: \Chat\rightarrow\Chat$ so that $G$ and  $L_r$ 
are $\varepsilon$-homeomorphic and $r^{-1}(\infty)=P$.
\end{thmx}

\begin{rem}
We might have  called two sets $E$ and $F$  $\delta$-homeomorphic  
if there is homeomorphism $\phi: E \to F$ so that 
$\sup_{z \in E} |\phi(z)-z|< \delta$, but this  
is weaker than Definition \ref{epsilon homeo defn}.
However, we shall prove in Section \ref{extending homeos sec} that 
if $G$ is a lemniscate graph and 
$ \phi: G \to G'$ is a  $\delta$-homeomorphism 
in this weaker sense, then it can be extended to 
an $\varepsilon$-homeomorphism of $\Chat$, 
assuming that $\delta$ is sufficiently  small depending 
on $G$ and $\varepsilon$ (this is false for more general sets).
Thus to prove Theorem \ref{main_thm_2},  it will  suffice to 
verify that for every $\delta>0$ there is a rational 
function $r$, so that $G$ and $L_r$ are $\delta$-homeomorphic in 
the weaker sense. 
\end{rem} 

\vskip.2in
\subsection{Hilbert's Lemniscate Theorem}

Hilbert  \cite{Hilbert1897} proved  that any closed Jordan 
curve is $\varepsilon$-homeomorphic to a  polynomial 
lemniscate. This is not how the result is usually stated, 
but it is an equivalent formulation, and it  makes it  easy
to see that Theorem  \ref{main_thm_2} generalizes Hilbert's result. 
More precisely, 
since a Jordan curve is a  lemniscate graph (one edge, no vertices)
with exactly two faces, we can choose the unbounded face to 
be colored grey, and  place the pole in Theorem \ref{main_thm_2}
at infinity. Thus the approximating rational lemniscate 
is actually a polynomial lemniscate, giving  Hilbert's theorem. 
More generally, 
any finite collection of disjoint Jordan curves 
that are not nested (no curve separates another one from infinity)
is  $\varepsilon$-homeomorphic to a polynomial lemniscate, since we can 
color the bounded faces white and the unbounded face grey. This case  
is due to Walsh and Russell \cite{MR1501732}, and  
generalizing their  result to arbitrary  families of 
disjoint  Jordan curves (i.e., allowing nesting) 
was the original motivation for the current paper.
In addition to recovering the theorems of Hilbert and Walsh-Russell,  
Theorem \ref{main_thm_2} also  gives the following  new
result about  polynomial lemniscates.

\begin{cor}
If $G$ is a  lemniscate graph that is the boundary of 
its unbounded face, then $G$ is  $\varepsilon$-homeomorphic to 
a polynomial lemniscate for every $\varepsilon>0$.
If a lemniscate graph  is not the boundary of its unbounded face,
then it is not  the image of a  polynomial lemniscate
under any homeomorphism of the plane.
\end{cor}

Both  Hilbert's lemniscate theorem and Theorem \ref{main_thm_2} 
say that every topological version of some object 
(a Jordan curve or lemniscate graph) is $\varepsilon$-homeomorphic to 
an algebraic version of the same object.

\newpage
\subsection{Quantitative Approximation by Rational Functions}\label{approx_by_rat}


\begin{thmx}\label{power_series_cor}
Let $K\subset\Chat$ be compact, 
let $P$ contain exactly one point from each component of 
$\Chat\setminus K$, and suppose $f$ is holomorphic in a 
neighborhood $U$ of $K$. Then there exist constants $ A,B
\in(1,\infty)$ and a sequence of rational mappings 
$R_n$ of degree $\leq n$ satisfying $R_n^{-1}(\infty) \subset P$ and
such that 
\begin{equation}\label{power_series_est} 
\sup_{z\in K}|f(z)-R_n(z)|\leq\frac{A}{B^n}  \, \, \textrm{ for all } n.
\end{equation}
\end{thmx}

Replacing the geometric rate of convergence  in (\ref{power_series_est})
with $o(1)$ is exactly Runge's classical approximation theorem. 
When the  neighborhood  $U$ in Theorem \ref{power_series_cor} is a disc, we 
can take $R_n$ to be the degree $n$  truncation of the  Taylor series for $f$ in $U$. 
We then see that Theorem \ref{power_series_cor} 
generalizes the well-known fact that the  Taylor series of an 
analytic function $f$ converges geometrically fast to $f$ on compact 
subsets of the disc of convergence.  
For more general open sets $U$,  we will first  choose a 
rational map $r$ whose 
lemniscate $L_r$ separates $K$ from $\partial U$, and then choose 
the maps $R_n$ so that their derivatives, $R_n^{(k)}$,  agree with 
the derivatives $f^{(k)}$ up to some order (depending on $n$)
at the zeros of $r$. 
See Section \ref{sec_cor_proof} for details.
To prove Theorem \ref{power_series_cor}, we will only need 
Theorem \ref{main_thm_2} in the case when the lemniscate
graph has no vertices, i.e., it  consists only of disjoint 
Jordan curves. We shall see that this case is much easier than 
the  general case of graphs with vertices. 

We will also show that the geometric decay in Theorem 
\ref{power_series_cor} is sharp for most functions: 
if there is a sequence of rational approximations 
that converge to $f$ faster than  geometrically 
along some subsequence of 
degrees, then $f$ extends holomorphically to $\Chat \setminus P$. 
See Section \ref{sharp_approx_sec} for the precise statement 
and proof.

In the case that $\Chat\setminus K$ is connected,
$\infty\in\Chat\setminus K$,
and $P=\{\infty\}$, Theorem \ref{power_series_cor} 
is exactly Theorem I of \cite{MR1501732}. 
All other cases of Theorem \ref{power_series_cor} 
(namely, whenever $\Chat\setminus K$ is not connected) 
are new, to the best of our knowledge.

\subsection{Approximation by Julia Sets}
If a rational map $r$ has exactly two attracting 
cycles, and if the Fatou set $\mathcal{F}(r)$  is 
equal to the union of the two corresponding 
attracting basins,  $\mathcal{A}_1$ and  $\mathcal{A}_2$,
then the sphere decomposes as (here $\sqcup$ denotes
disjoint union) 
\[ \Chat=\mathcal{A}_1\sqcup\mathcal{A}_2\sqcup\mathcal{J}(r), \]
where $\mathcal{J}(r)$ is the Julia set of $r$, and  
$\partial\mathcal{A}_1=\partial\mathcal{A}_2=\mathcal{J}(r)$.
See, for instance, Corollary 4.12 of \cite{MilnorCDBook}.
Note that each attracting basin is an open set, but   need not be 
connected.
Our next result implies that any disjoint pair of open sets 
sharing a common boundary
can be approximated by a pair of  attracting basins for some 
rational map.

\begin{thmx}\label{julia_set_cor} 
Let $\varepsilon>0$ and $A_1$, $A_2\subset\Chat$ be open, non-empty, disjoint sets
with common boundary $J$ satisfying $\Chat=A_1\sqcup A_2\sqcup J$.
Then there is a hyperbolic rational map $r$ with two attracting basins,
$\mathcal{A}_1$ and  $\mathcal{A}_2$, sharing a common boundary 
$\mathcal{J}(r)$, satisfying
$\Chat=\mathcal{A}_1\sqcup\mathcal{A}_2\sqcup\mathcal{J}(r)$, so that
\begin{equation}
d_H(A_i, \mathcal{A}_i)<\varepsilon 
\textrm{ for } i=1,2 \textrm{ and } d_H(J, \mathcal{J}(r))<\varepsilon,
\end{equation}
where  $d_H$  denotes Hausdorff distance.
Moreover, if $A_1$ has finitely many components,  $P$ contains 
one point from each component of $A_1$, and $p \in P$, 
then we may choose $r$ so  that $r(p)=p$, $r^{-1}(p)=P$
and  $\mathcal{A}_1$ is the basin of attraction for $p$.
\end{thmx}

In the case when $A_1$ is connected, contains $\infty$ 
and $P=\{\infty\}$, this result is   essentially the same as
Theorem 1.2 of \cite{MR3955554} by Lindsey and Younsi.  
They  show that a compact set $K$  not separating the plane can be 
approximated by the filled Julia set of a polynomial.
Moreover, they show that such a polynomial approximation 
is possible only if the interior of $K$ does not separate the plane;
see Theorem 1.4 of \cite{MR3955554}. 
Thus when $A_1$ is not connected in Theorem \ref{julia_set_cor}, 
it is necessary to consider rational maps having finite poles.

Lindsey and Younsi give two proofs of their Theorem 1.2. The 
first argument 
does not seem to  extend to the case that $A_1$ is disconnected,
but their second proof  is a short application of 
Hilbert's Lemniscate Theorem, and 
replacing Hilbert's result with our
Theorem  \ref{main_thm_2} yields a quick proof of
Theorem \ref{julia_set_cor}. 
See Section \ref{first_cor_proof} for details. 
Like  Theorem \ref{power_series_cor}, the proof of
Theorem \ref{julia_set_cor}  only requires
Theorem \ref{main_thm_2} in the easier case when the lemniscate
graph has no vertices.

\subsection{Related Work.}

Hilbert's interest in lemniscates seems to have arisen from a study of
polynomial approximation \cite{Hilbert1897}, where he used his
lemniscate theorem to establish the special case of Theorem 
\ref{power_series_cor} in which $K$ has simply-connected complement.
Later, Walsh and Russell \cite{MR1501732} also studied lemniscates 
in the course of their proof of Theorem \ref{power_series_cor}
in the case that $K\subset\mathbb{C}$ has a connected complement. 

Hilbert's theorem has been generalized to higher complex 
dimensions by 
Bloom,  Levenberg and Lyubarskii
\cite{MR2473634}, 
and by Nivoche
\cite{MR2518003}. 
Rashkovskii and Zakharyuta \cite{MR2838245} 
prove that every bounded polynomially convex poly-circular 
region in $\complex^n$ can be approximated by regions 
$\{|p_k(z)|<1, 1\leq k\leq n\}$, where
$p_k$ are homogeneous polynomials of the same degree with a
unique single common zero at 0.
Nagy and Totik consider placing a lemniscate between tangent 
curves in 
\cite{MR2177185}.  
The  rate of convergence in Hilbert's theorem  has been studied by Andrievskii 
\cite{MR1775150}, 
\cite{MR3978145}, 
and  Kosukhin \cite{MR2246962}. 
Results on generalized lemniscates for resolvents of operators  are 
surveyed by Putinar in \cite{MR2129650}.
Lemniscate approximations via Runge's theorem are given in 
\cite{MR428074} and \cite{MR4416767}, and applied  to various problems 
in functional analysis. 

There has also been recent interest in rational lemniscates and 
Hilbert's theorem stemming from work of Ebenfelt, Khavinson, 
and Shapiro \cite{MR2868587}, in which they propose coordinates
on  the space of Jordan curve polynomial lemniscates. By Hilbert's
Lemniscate Theorem, this lemniscate space is dense in the larger
space of smooth Jordan curves, and this larger space of smooth
Jordan curves is the central object of study in the approaches
of Kirillov \cite{MR0902292}, Sharon and Mumford \cite{Sharon-Mumford}
and others to computer vision and pattern recognition.
The results in \cite{MR2868587} led to a study of the conformal properties of
lemniscates by Fortier Bourque and  Younsi \cite{MR3296178}, by 
Younsi \cite{MR3447662}, and by Frolova,  Khavinson and Vasil'ev \cite{MR3784168}.

 A well-known question of Erd\H{o}s, Herzog, and Piranian \cite{MR101311}
asks what is the maximum length of 
a lemniscate of a monic polynomial: $z^n-1$ is conjectured to be the 
extreme case. This problem is still open, but 
related results are given by Borwein in 
\cite{MR1223265}, 
by the second author and Hayman in 
\cite{MR1704189}, and by Nazarov and Fryntov in
\cite{MR2500509}.  
Other recent work on lemniscates includes a study of the expected 
length of random rational lemniscates as considered by Lerario and Lundberg in
\cite{MR3356754}, \cite{MR3570241}. Random polynomial lemniscates
are considered by Lundberg and Ramachandran 
\cite{MR3742436}, and by Epstein, Hanin and Lundberg \cite{MR4245584}. 
There has been  work on the possible topologies of rational 
lemniscates, the tree structure of nested components, and the 
comparison of functions whose level sets are 
topologically equivalent, e.g.,  
\cite{MR1133876},
\cite{MR3922299},
\cite{MR3558373},
\cite{MR3590700},
\cite{MR825840}.
For a survey of recent  results on lemniscates and level sets, see \cite{MR4261771}.

The geometry of higher dimensional, real  polynomial lemniscates is 
a very active field,  e.g. the
``polynomial ham sandwich theorem'' of Stone and Tukey \cite{MR7036} 
gave  rise to the ``polynomial method''  in discrete geometry
and combinatorics, as described by Guth in \cite{MR3202645}. 
Recently, real variable polynomial lemniscates have been used 
to separate data points in the context of machine learning,  e.g., 
\cite{Korda2022} and  \cite{MR3964612}.

We also mention again the recent work of Lindsey and Younsi \cite{MR3955554}
explaining the connection between lemniscates and the approximation of 
continua by polynomial  Julia sets, a problem also studied  
by  Lindsey in \cite{MR3377290}, and  by the first author and
Pilgrim in \cite{MR3420484}.
Bok-Thaler
\cite{MR4375923} and Marti-Pete, Rempe and Waterman \cite{2022arXiv220411781M}
consider analogous problems in the setting of  transcendental 
(i.e., non-polynomial) entire functions. 

A slightly different, but related, problem is the study of pullbacks
$r^{-1}(\Gamma)$ where $\Gamma$ is a Jordan curve passing through 
the critical values of $r$ (in our case, $\Gamma$ is always a Euclidean 
circle, and need not contain all the critical values of $r$).
Such pullbacks are called \emph{nets}, and have been studied by
the second author and Gabrielov \cite{MR1888795},  Thurston 
\cite{38274}, Koch and Lei \cite{MR4205641},
Tomasini \cite{MR3422731}, 
and the third author \cite{lazebnik2023analytic}.

\subsection{Proof Sketch.}
In Section \ref{smoothing sec}, we show that any lemniscate 
graph is $\varepsilon$-homeomorphic to a graph with analytic edges 
making equal angles at each vertex. Thus Theorem \ref{main_thm_2} is 
reduced to this case, and this simplifies various arguments. 
In Section \ref{extending homeos sec} we show 
that a homeomorphism moving points 
of a lemniscate graph $G$ less than $\delta = \delta(\epsilon, G)$
can be extended to 
an $\epsilon$-homeomorphism of  the sphere; this fact 
further simplifies the proof of Theorem \ref{main_thm_2}.

In Section \ref{Alexs_proof}, 
we prove Theorem \ref{main_thm_2} in the case that
$G$ has no vertices (it is a union of disjoint Jordan curves). This case 
is much easier, but introduces several of the key ideas. Briefly, 
we consider Green's functions on the grey faces with poles at the  points 
of  $P$ (one point per grey face),  
and note that $G$ can be approximated by 
the level lines of  a  function $u$ that is the sum of these Green's functions.
This function $u$  can be  written as a convolution of the logarithmic kernel 
with the sum of harmonic measures for the grey faces, and negative point masses 
at the poles in $P$. The harmonic measures 
are then approximated by sums of point masses, and this leads to a rational 
function  $r$  so that $\log|r|$ approximates $u$ away from $G$; 
thus $L_r$ approximates
the given level line of $u$, and hence it also approximates $G$. 

The case when  $G$ has vertices is more difficult, and
requires some new ideas. Given a lemniscate graph $G$  with vertices, 
we claim that there is a graph $H$, without vertices, and a
corresponding function $u$ as above, 
so that a certain  level set of $u$   is $\varepsilon$-homeomorphic to $G$. 
The proof of this  claim is   postponed 
to Sections \ref{topological_section_app} and  \ref{harmonic_app_appendix}.
Assuming the claim  holds,  we prove  Theorem \ref{main_thm_2} 
in Sections \ref{graphs_with_verts_sec} 
and \ref{proof_of_main_thm}: in Section  \ref{graphs_with_verts_sec}
we  place the 
poles of the rational functions as close to $P$ as we wish, 
and in  Section \ref{proof_of_main_thm} we give 
a fixed point argument to  position them exactly on $P$.
Sections \ref{sec_cor_proof} and  \ref{first_cor_proof} are devoted
to deducing Theorems  \ref{power_series_cor} and
\ref{julia_set_cor}, and the sharpness of Theorem \ref{power_series_cor}
is proven in Section \ref{sharp_approx_sec}.

As noted above, the proof of Theorem \ref{main_thm_2} is much 
shorter when $G$ has no vertices, and this special case is 
sufficient for the proofs of Theorems 
\ref{power_series_cor} and \ref{julia_set_cor}.
The reader who is only interested in this case can skip 
the proof of Theorem \ref{eps homeo graphs}, 
and the proofs in Sections \ref{graphs_with_verts_sec}, 
\ref{proof_of_main_thm}, 
\ref{topological_section_app} and  \ref{harmonic_app_appendix}; 
this will omit about half of the paper.
We also note that only the first half of Section  \ref{harmonic_app_appendix} 
is needed when all vertices of $G$ have degree four.

\subsection{Notation.}
We  will generally use boldfaced symbols for
vectors, e.g., $\myvec{x} = (x_1, \dots, x_n)$ or
$\myvec{\theta} =(\theta_1, \dots, \theta_n)$.
An open  disk of radius $r$ centered at $z$ will be
denoted by $D(z,r)$.
When $A$ and $B$ are both quantities that depend on a common parameter, then
we use the usual notation $A=O(B)$ to mean that the ratio $A/B$ is bounded independent
of the parameter.
We write $A \simeq B$ if both  $A=O(B)$ and $B = O(A)$.  The notation $A=o(B)$ 
means $A/B \to 0$ as the parameter tends to infinity. 
We will use $G$  to denote general lemniscate graphs and $H$ 
to denote lemniscate graphs without vertices (finite unions of 
disjoint closed Jordan curves). A $G$ with a subscript, such 
as $G_W$, will denote the Green's function for the domain $W$, 
and we define such functions to be zero off of $W$.
In general, we use $A:=B$ to mean that $A$ is being defined 
in terms of $B$,  and $A=B$ to show equality between two (already)
defined quantities.

\subsection{Acknowledgements.}
Thanks to Fedja Nazarov for helpful discussions.

\section{From topological graphs to analytic graphs} \label{smoothing sec}

In the introduction, 
we defined  lemniscate graphs so that the edges are just Jordan arcs; no smoothness
is assumed.
 However, we will show every topological graph is 
$\varepsilon$-homeomorphic (for every $\varepsilon >0$) 
to an analytic version of itself, that is, a graph with analytic edges. 
This allows us to reduce the proof of Theorem \ref{main_thm_2}
to the case  when the lemniscate graph has smooth edges, and this 
simplifies some of the arguments.
We can think of this as a ``stepping stone'' to the main result 
of this paper, that every   lemniscate graph  
is $\varepsilon$-homeomorphic 
to an algebraic version of itself, i.e., a rational lemniscate. 

By an analytic edge $e$, we mean the image of the line segment 
$[0,1]$ under a locally injective holomorphic map defined on 
some neighborhood of $[0,1]$ (the map is conformal on a 
neighborhood of $[0,1]$ if the endpoints of $e$ are distinct). 
In particular, such an edge is a subset of a slightly longer 
analytic curve (the image of a longer segment
$[-\varepsilon, 1+\varepsilon]$ under
the same map), and hence it has a well defined direction at each 
endpoint, and it has uniformly bounded curvature. 

\begin{thm}\label{eps homeo curves} 
Suppose that $H$ is a   closed Jordan curve. For 
any $\varepsilon >0$, $H$ is $\varepsilon$-homeomorphic 
to an analytic Jordan curve, and 
the homeomorphism may be taken to be the identity outside 
an $\varepsilon$-neighborhood of $H$. 
\end{thm}

\begin{thm}\label{eps homeo graphs} 
Suppose that $G$ is a connected   lemniscate graph. Then for 
any $\varepsilon >0$, $G$ is $\varepsilon$-homeomorphic 
to a lemniscate graph whose edges are analytic, 
and so that  the edges meeting at any vertex form equal angles.
The homeomorphism may be taken to be the identity outside 
an $\varepsilon$-neighborhood of $G$. 
\end{thm}

Clearly, the first result is a special case of the second,
but in the proof it is convenient to first deal with the case 
when there are no vertices (and Theorem \ref{eps homeo graphs} 
is not needed for the proofs of Theorems  \ref{power_series_cor} and
\ref{julia_set_cor}).

For disconnected graphs, we can apply one of these results to 
each of the connected components. If  $\varepsilon$ is 
less than half the minimal distance between components, then
the composition of these maps gives an $\varepsilon$-homeomorphism 
of the entire graph to an analytic version of itself (each map in 
the composition only moves points that all the other maps fix). 
Our proof gives a new edge $e'$    that is 
disjoint  from  the edge $e$ it replaces (except for the endpoints), 
and can  place it inside either of the two faces of $G$ bounded by $e$.

\begin{proof} [Proof of Theorem  \ref{eps homeo curves}]
Let $\phi: \mathbb{D}\rightarrow 
\Omega$ denote a Riemann map onto one of the complementary 
	components of $H$. The map $\phi$ 
extends continuously to a map $\phi: \mathbb{T} \rightarrow H$,
and this extension further extends to a homeomorphism 
$\phi: \Chat\rightarrow \Chat$ by the Jordan-Sch\"onflies theorem.
Let $\rho_\delta$ be an increasing homeomorphism from 
	$I_\delta = [1-2\delta, 1+2 \delta]$ 
to itself so that $\rho_\delta(1) = 1-\delta$. 
	Since $\rho_\delta$ is increasing, it fixes 
	both endpoints of $I_\delta$, and we can 
	extend $\rho_\delta$ to 
the complex plane by $\rho_\delta(z) = \rho_\delta(|z|) z/|z|$ 
if $ 1- 2 \delta < |z| < 1+2 \delta$, and as the identity otherwise.
We claim that the homeomorphism
\[ \psi_\delta:=\phi\circ  \rho_\delta \circ \phi^{-1}: 
	\Chat \rightarrow \Chat  
\]
satisfies the conclusions of the theorem  if $\delta >0$ is small enough.
Indeed, we have $\psi_\delta(H)=\phi((1-\delta) \mathbb{T})$ 
is an analytic Jordan curve, and
\begin{equation}\label{firsteq} 
	\dist(\psi_\delta(z), z)=
	\dist(\phi (\rho_\delta(\phi^{-1}(z))), z) \textrm{ for all } z\in\Chat,
\end{equation}
where we use $\dist$ to denote spherical distance. 
Since $\rho_\delta$ tends to the identity as $\delta \searrow 0$ 
and 
  $\phi$ is uniformly continuous on $\Chat$ (since $\phi$ 
is continuous on the compact set $\Chat$), we conclude 
that $\psi_\delta$ is an $\varepsilon$-homeomorphism 
for $\delta$ close enough to $0$.
\end{proof}

To prove Theorem \ref{eps homeo graphs}, we first need  
to modify the graph near each vertex.

\begin{lem}\label{WLOG_thm} 
Let $G'$ be a lemniscate graph and $\varepsilon>0$. Then there 
exists an $\varepsilon$-homeomorphism of $G'$ onto a lemniscate
graph $G$ having the property that in a neighborhood of each 
vertex of degree $2n$, the graph $G$ is a union of $2n$ 
straight line segments each terminating at $v$, and making equal angles.
The homeomorphism may be taken to be the identity outside 
an $\varepsilon$-neighborhood of $G'$. 
\end{lem}

\begin{proof} 
Let $v$ be a vertex of $G'$ of degree $2n$, and $D$ a Jordan domain 
containing $v$ so that $G'$ intersects $\partial D$ at exactly $2n$ points.
We will define $G$ by replacing $G'$ in $D$ with $2n$ Jordan arcs which
start at the points $G'\cap\partial D$, and in a small ball around $v$
consist of straight lines terminating at $v$ and making equal angles.
Doing this for each vertex $v$ of $G'$ defines the graph $G$.
Suppose each $D$ has diameter $<\varepsilon$.  Define a homeomorphism
as the identity outside of each $D$, and inside each $D$   by 
first by mapping the Jordan arcs $G'\cap D$
onto the corresponding Jordan arcs in $G\cap D$,
and then extending by the Jordan-Sch\"onflies theorem throughout $D$
(e.g., Corollary 12.15  in \cite{MR4321146}). 
The result is an $\varepsilon$-homeomorphism of $\Chat $ mapping 
$G'$ to $G$.
\end{proof}

\begin{proof} [Proof of Theorem \ref{eps homeo graphs}]
We may assume that $G$ has the form given in 
Lemma \ref{WLOG_thm}, but that $G$ is not a Jordan curve, since
that case is covered by Theorem \ref{eps homeo curves}. 
Thus  $G$ has vertices and every vertex has even degree 
at least four.
 It suffices to replace each edge $e$ of $G$ by a 
new edge that is analytic (in the sense discussed earlier) 
and tangent to $e$ at their   common endpoints. 
There are two cases, depending on whether the
endpoints of $e$ are two distinct points ($e$ is an arc) 
or a single point ($e$ is a loop). 

For any edge $e$ in a lemniscate graph $G$, we claim that 
there is a Jordan 
curve $\gamma \subset G$ containing $e$. 
We can prove this by forming a directed  path in $G$ that starts with 
the edge $e$ (choose either orientation). Since every vertex has
even degree the path  can be continued until it returns to the 
initial vertex of $e$. For any vertex $v$ on this path,
we erase the sub-path between the first and last visits to $v$.
This gives a Jordan curve  $\gamma$ containing $e$.
If  $e$ has distinct endpoints, then  $\gamma$ strictly contains $e$. 
If both ends of $e$ are the same point, then we can take $\gamma =e$,
but for the proof below,  it 
will be more convenient to consider a different curve. If we remove $e$ from $G$, 
we still have an Eulerian graph, so there is a Jordan curve $\gamma'$  in $G$ 
containing $v$ but not  $e$. Then $\gamma = e \cup \gamma'$ is 
a ``figure 8'' contained in $G$. 

First assume the endpoints of $e$ are  two different points. 
Fix an  endpoint $v$ of $e$. The curve  $\gamma$  contains two line 
segments, $s_1$ and $s_2$, each with endpoint $v$. 
If these two segments  lie on the 
same line, then their union is a line segment with $v$ in its interior. 
In this case, we do nothing. If the two segments are not on the 
same line, then we modify $\gamma$  by extending
the segment $s_1$  a short distance past $v$, 
and then connect the end of this  extended segment 
by a circular arc (centered at $v$)  to  an interior point of $s_2$.
See  Figure \ref{ModifyGamma}. 
This gives a new Jordan curve that contains a segment with 
midpoint $v$, and that equals $\gamma$ outside a small
neighborhood of  $v$.  Do this replacement, if necessary, at
both endpoints of $e$, and  call the new curve $\gamma$ also.

\begin{figure}[htb]
\centerline{
\includegraphics[height=1.0in]{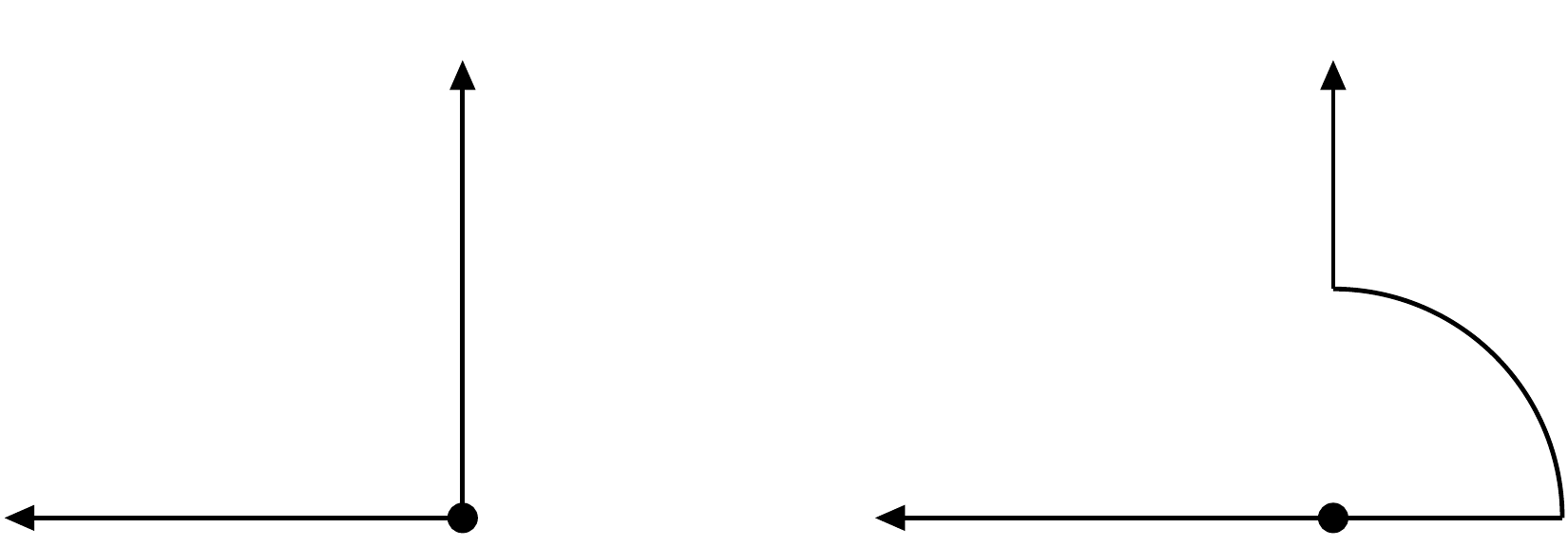}
}
\caption{
We modify $\gamma$ near each vertex $v$ so it contains 
a line segment centered at $v$.
}
\label{ModifyGamma}
\end{figure}

Let $\Omega$ denote a component of $\Chat\setminus\gamma$ 
(it does not matter which component we pick). 
Let $\phi: \mathbb{D}\rightarrow \Omega$ denote a Riemann map. 
Since $\partial \Omega$ is locally connected, 
$\phi$ extends to a homeomorphism $\phi: \mathbb{T}\rightarrow \gamma$.
We may choose $\phi$ so that the points $-1$ and $1$ map to the two endpoints 
of $e$, and so that $e= \phi(\Gamma_0)$, where $\Gamma_0$ is the 
upper half-circle, $\Gamma_0 = \circle \cap \uhp$.
Because $\gamma$ contains line segments  centered at each of 
the endpoints of $e$,  there is some $\eta>0$ so that 
$\phi$ extends analytically to a  $\eta$-neighborhood  around each of 
$-1$ and $1$ by Schwarz reflection.
Extend $\phi$  from $\disk \cup D(-1, \eta) \cup D(1, \eta)$ 
to a homeomorphism $\phi:\Chat\rightarrow\Chat$
by the Jordan-Sch\"onflies Theorem.

For $|\delta|$ small, consider the rational  map
$R_\delta(z) = (z +\delta/z)/(1+\delta)$.
On the unit circle $1/z = \overline{z}$,  so 
$$
R(x+iy) = (x+iy + \delta(x-iy))/(1+\delta) = 
x+i y /\mu ,
$$
 where $\mu = (1+\delta)/(1-\delta)$. Thus $R$ 
maps the unit circle onto the ellipse 
$$
E_\delta =\{(x,y)\in \reals^2  :
x^2 + \mu^2 y^2 =1\}.
$$
When $\delta \in (0,1)$, we have $  \mu>1$, which implies 
this ellipse  is contained in the 
open unit disk, except at the points
$-1,1$ where it is tangent to the unit circle. 
Thus $\phi$ is analytic on a neighborhood of $E_\delta$ and 
$\phi(E_\delta)$ is an analytic closed curve that is tangent to $e$ at 
both of its endpoints.  We define the  replacement edge  
to be $e' = \phi (\Gamma_\delta)$, 
where $\Gamma_\delta = E_\delta \cap \uhp$ is the upper half of $E_\delta$.

Let $G'$ be the lemniscate graph obtained from $G$ by replacing the edge $e$
by $e'$. We want to show these two graphs are $\varepsilon$-homeomorphic by 
a map that is the identity off a neighborhood of $e$ (recall that 
edges are open arcs and do not contain their endpoints). 
As above, let $\Gamma_t = E_t \cap \uhp  =  R_t(\Gamma_0)$, and 
let $U_\delta$ be the union of $\Gamma_t$  for 
$t \in  I = [-2\delta, 2 \delta]$.  See Figure \ref{Ellipses}. 
For $z \in U_\delta$, define $t(z)= s$ if $z \in \Gamma_s$.
Let $\rho: I  \to I$ be a homeomorphism that fixes each endpoint
and satisfies $\rho(0) = \delta$.
Then $ h_\delta (z) =  R_{\rho(t(z))}( R_{t(z)}^{-1} (z))$ is a homeomorphism of 
$U_\delta$ that is the identity on $\partial U_\delta$ and maps $\Gamma_0$ 
to $\Gamma_\delta$. It moves points in $U_\delta$ by at most $O(\delta)$. 
Set $W_\delta := \phi(U_\delta)$ and define
$ \psi_\delta(z) =   \phi(h_\delta(\phi^{-1}(z)))$  if 
$z\in W_\delta$
and  $\psi(z) =z$  elsewhere. For $\delta$ small enough, $\psi_\delta$ is a 
$\varepsilon$-homeomorphism of the sphere taking $G$ to $G'$, and it 
is the identity on every edge of $G$ except $e$. 

\begin{figure}[htb]
\centerline{
\includegraphics[height=2.0in]{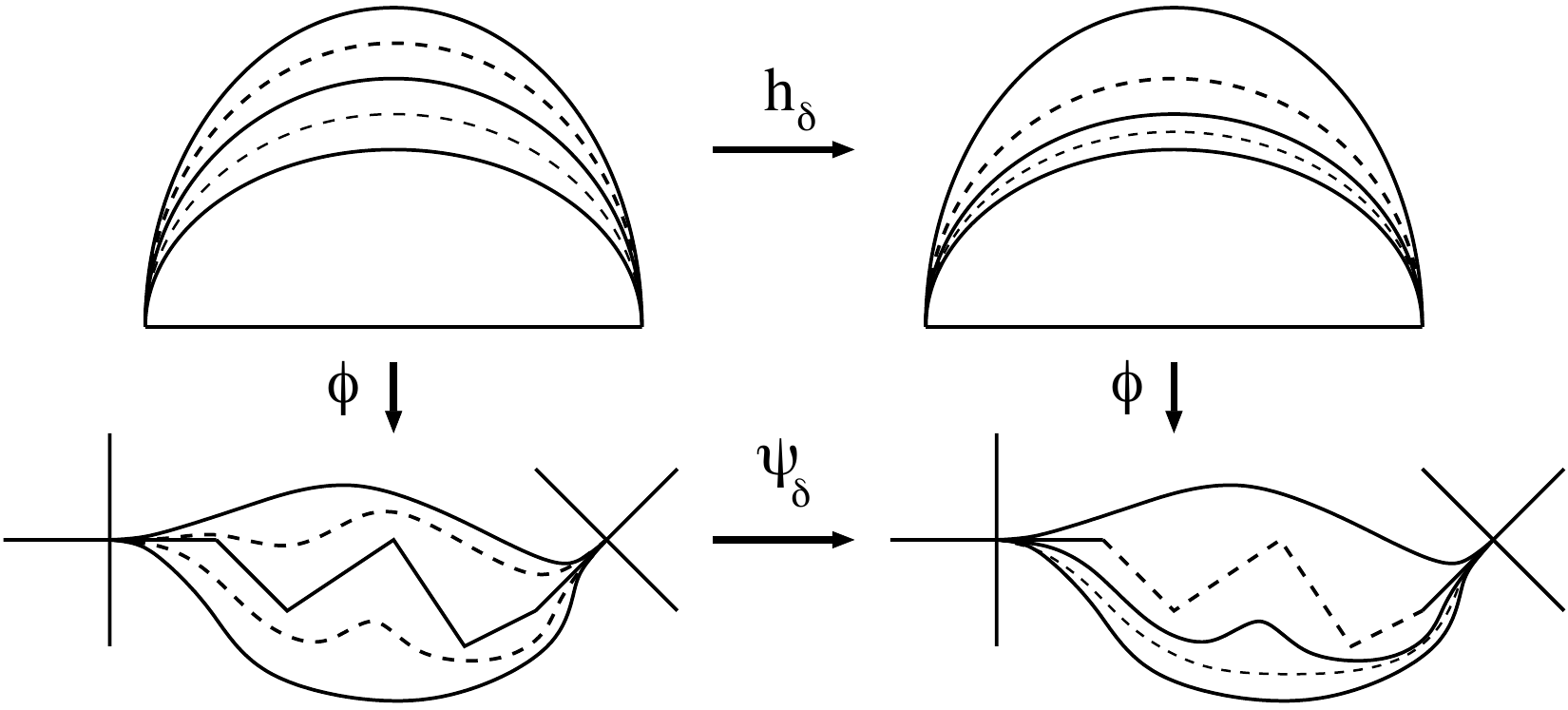}
}
\caption{
We define a neighborhood $U$  of the upper half circle as a union of 
elliptical arcs, and define a homeomorphism of $U$ 
to itself by shifting the arcs. This is the identity outside  $U$.
When conjugated by $\phi$  this becomes an $\varepsilon$-homeomorphism 
of $G$ to $G'$  taking $e$ to $e'$, that is the identity on $G \setminus e$.
}
\label{Ellipses}
\end{figure}

Next, suppose the endpoints of $e$ 
are the same point $v$. This case can be reduced to the 
previous one.   
See Figures \ref{SquareRoot} and  \ref{OneEndpoint}. 
As noted earlier, in this case $e$ is part
of a ``figure 8'' curve $\gamma = e \cup \gamma' \subset G$. 
The curve $\gamma$ has three complementary components: one with 
boundary $e$, one with boundary $\gamma'$, and a third  
component $\Omega$ with boundary $\gamma$.
Choose points $a$ and $b$ in the first two, and 
let $\tau(z) = (z-a)/(z-b)$; this is a linear fractional transformation 
sending $a$ to $0$ and $b$ to $\infty$. 
Define $f(z)= \tau^{-1}(\tau(z)^2)$ (we could write the 
formula explicitly, but we don't need it). This is rational map of
degree $2$, and we can define a branch of $f^{-1}$ on  
$\Omega$ by taking a branch of $z^{1/2}$ on $\tau(\Omega)$
(which is simply connected and omits both  $0$ and $\infty$). 
Then $f^{-1}$  has two different values at $v$ and 
it maps  $\Omega$ to a Jordan region $\Omega'$, and 
$v$ corresponds to two points on the boundary of $\Omega'$.
See Figure \ref{SquareRoot}. 

\begin{figure}[htb]
\centerline{
\includegraphics[height=1.1in]{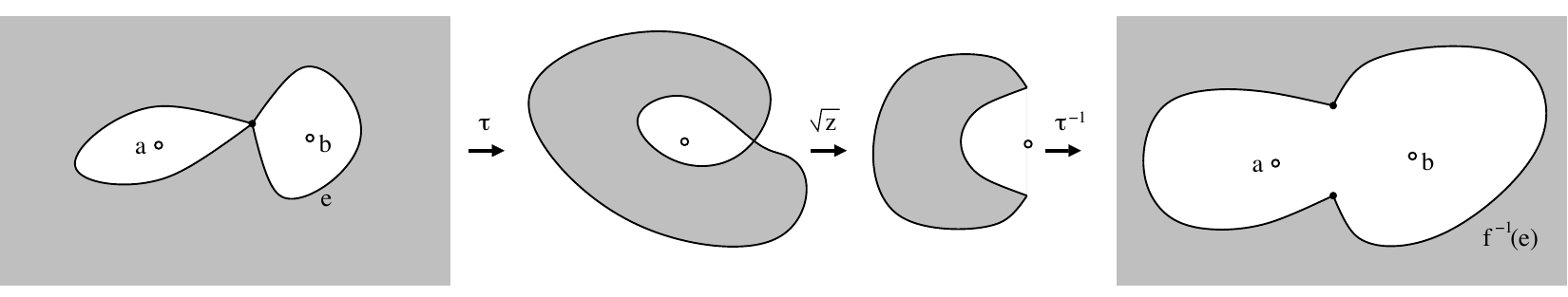}
}
\caption{
Using linear fractional transformations and a square root, 
we can map the non-Jordan face of $\gamma$ to a Jordan domain. 
The inverse of this map is a rational map. 
}
\label{SquareRoot}
\end{figure}

We repeat the earlier  construction for Jordan faces given 
above to find an analytic curve  $\sigma$ that approximates
$f^{-1}(e)$, and that is tangent 
to $f^{-1}(e)$ at its endpoints. Then apply the  
rational map $f$ to  $\sigma$  and obtain an 
analytic curve $e'$ approximating  $e$. 
See Figure \ref{OneEndpoint}.

\begin{figure}[htb]
\centerline{
\includegraphics[height=2.5in]{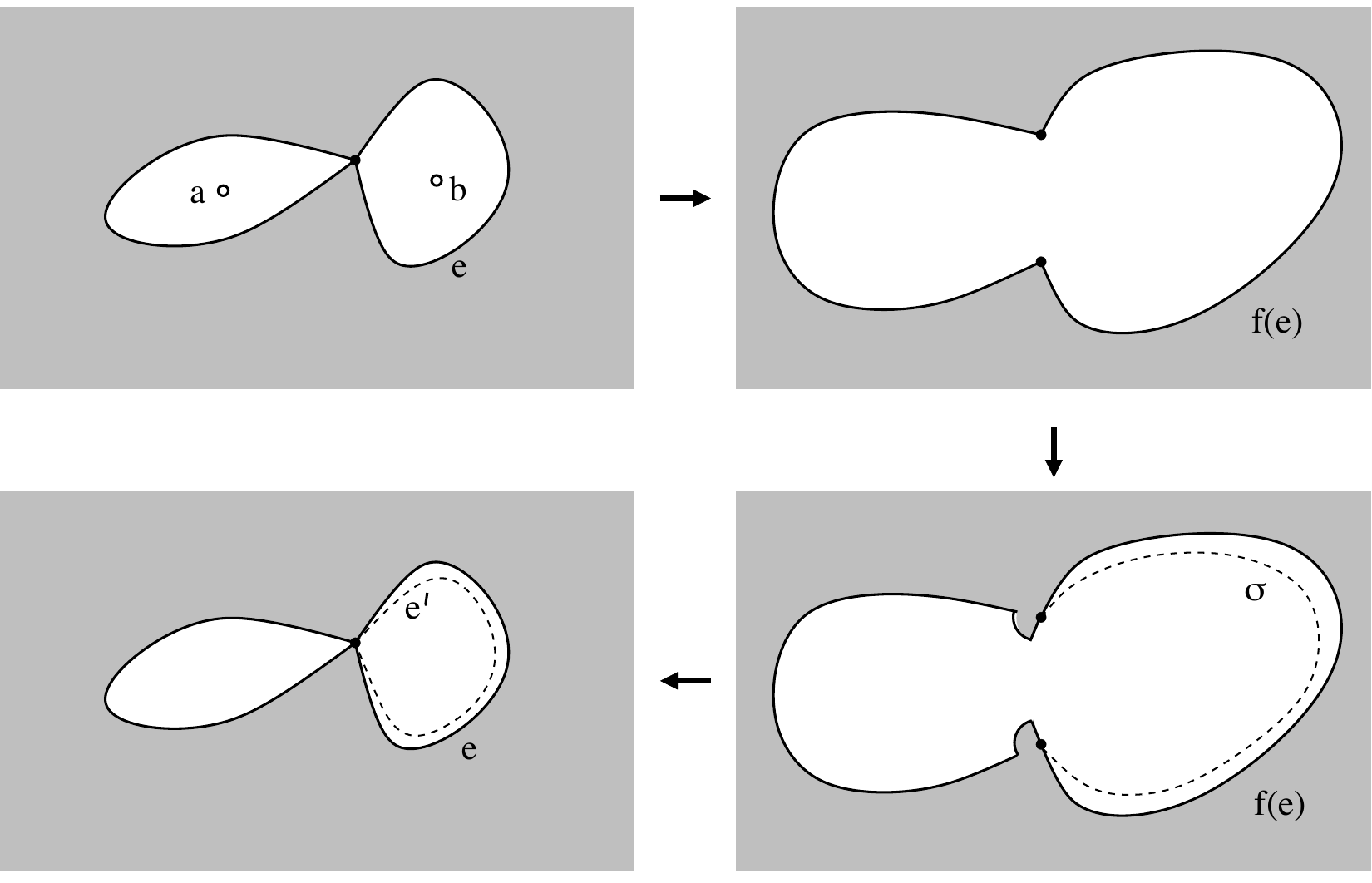}
}
\caption{
If both endpoints of $e$ are the same, we can 
find a ``figure 8'' curve $\gamma$ containing $e$
and use a branch of $ f^{-1}$ to map the 
non-Jordan face  of $\gamma$ to a Jordan 
domain. Then the earlier construction  gives
an analytic approximation to $f(e)$ and applying the 
the  rational map $f$ gives the an analytic approximation to $e$.
}
\label{OneEndpoint}
\end{figure}

Thus we may  proceed edge by edge, replacing edges which have not previously
been replaced, and leaving  previously altered  edges fixed. Each graph we 
obtain is $\varepsilon$-homeomorphic to the previous one, and since 
at most  one of these  homeomorphisms  moves any given point of $\Chat$, the 
composition is an $\varepsilon$-homeomorphism of $\Chat$, as desired.
\end{proof} 


The following result was promised in the introduction.

\begin{prop}\label{toplemimpliesgraph}
Let $L$ be a rational lemniscate.
	Then $L$ is a lemniscate graph.
\end{prop}

\begin{proof}
We will show $L$ satisfies Definition \ref{embeddedgraph}.
We set $V:=\{z\in L_r: r'(z)=0\}$.
Recall the local normal form for holomorphic maps (see for
instance Theorem 1.59 in \cite{MR4422101}); namely that for
any $z_0\in\mathbb{C}$, up to holomorphic changes of coordinates
in the domain and co-domain, $r$ is locally given by $z\mapsto z^n$
near $z_0$, and $n=1$ if and only if $r'(z)\not=0$. Thus each
component of $L_r\setminus V$ is a connected $1$-manifold. The
only connected $1$-manifolds are (closed) Jordan curves or (open)
simple arcs (see for instance \cite{MR1487640}).
This shows that
$L$ satisfies condition (1) in Definition \ref{embeddedgraph}.
The fact that $L$ satisfies condition (2) in Definition \ref{embeddedgraph}
also follows from the local normal form for holomorphic mappings near critical points.
\end{proof}

\section{Extending $\varepsilon$-homeomorphisms}\label{extending homeos sec}

Next, we record a fact that we will use several times later
in the paper.

\begin{lem} \label{small move graphs}
Let $G$ be a lemniscate graph and $\varepsilon>0$. 
Then there exists $\delta>0$ so that if $h:G \rightarrow G'$ 
is any homeomorphism satisfying $\sup_{z\in G}|h(z)-z|<\delta$,
then $h$ admits an $\varepsilon$-homeomorphic
extension $h: \Chat\rightarrow\Chat$.
\end{lem}

This fact is probably well  known, but we  are not aware of a reference,
so we give a proof for completeness.
The main difficulty is not that the extension exists, but that
it can be chosen to move points very little.

\begin{definition}
We will call a domain $\Omega$ a \emph{lemniscate domain} 
iff $\Omega$ is simply connected and $\partial\Omega$ is a lemniscate graph.
\end{definition}

Lemniscate domains include Jordan domains, but also include
non-Jordan domains such as the unbounded complementary component
of a figure-8 curve in the plane.
Lemniscate graphs are locally connected, and so Carath\'eodory's
Theorem implies that any Riemann map $\phi: \mathbb{D} \rightarrow \Omega$ 
onto a lemniscate domain $\Omega$ extends continuously to $\mathbb{T}$.

\begin{lem}\label{Pomm_lem} 
Let $\Omega$ be a lemniscate domain, $z_0\in\Omega$ and set
$\gamma:=\partial\Omega$. Assume $(\delta_n)_{n=1}^\infty$ 
is a positive sequence converging to $0$, and 
$h_n: \gamma \rightarrow h_n(\gamma)$ is a sequence of 
$\delta_n$-homeomorphisms. Let $\Omega_n$ be the component 
of $\Chat\setminus h_n(\gamma)$ containing $z_0$. Denote by 
$\phi: \mathbb{D} \rightarrow \Omega$, 
$\phi_n: \mathbb{D}\rightarrow \Omega_n$ the Riemann maps 
normalized to map $0$ to $z_0$ and have positive derivative 
at $0$. Then $\phi_n\rightarrow\phi$ uniformly on 
$\overline{\mathbb{D}}$.
\end{lem}

\begin{proof} 
It is straightforward to check that the domains $\Omega_n$ 
converge to $\Omega$ in the \emph{Carath\'eodory kernel} sense
(see Section 1.4 of \cite{MR1217706}), and so Carath\'eodory's
convergence theorem implies that $\phi_n\rightarrow \phi$ 
uniformly on compact subsets of $\mathbb{D}$; in particular 
$\phi_n\rightarrow \phi$ pointwise on $\mathbb{D}$.

Thus, by Corollary 2.4 of \cite{MR1217706}, in order to conclude
that $\phi_n\rightarrow \phi$ uniformly on $\overline{\mathbb{D}}$,
it suffices to check that the sets $\Chat\setminus\Omega_n$ 
are \emph{uniformly locally connected} (see Section 2.2 of 
\cite{MR1217706}); that is, it suffices to check that for all
$\varepsilon>0$, there exists $\delta>0$ (independent of $n$)
so that any two points $a$, $b \in \Chat\setminus\Omega_n$ 
satisfying $|a-b|<\delta$ can be joined by a continuum 
$K\subset \Chat\setminus\Omega_n$ of diameter $<\varepsilon$.

Fix $\varepsilon>0$. Since $\gamma=\partial\Omega$ is a lemniscate
graph, $\gamma$ is locally connected and hence there exists 
$\delta_\gamma>0$ so that any two points $a$, $b\in\gamma$ 
satisfying $|a-b|<\delta_\gamma$ can be joined by a continuum 
$K\subset\gamma$ of diameter $<\varepsilon/2$. Set
$\delta=\delta_\gamma/2$. Let $z, w \in h_n(\gamma)$. For 
$n$ large enough so that $\delta_n<\delta_\gamma/4$, we have
that $|z-w|<\delta$ implies that $a=h_n^{-1}(z)$, 
$b=h_n^{-1}(w) \in \gamma$ satisfy $|a-b|<\delta_\gamma$.
Thus (by our choice of $\delta_\gamma$) there is a continuum
$K\subset\gamma$ of diameter $<\varepsilon/2$ joining $a$, 
$b$, and so $h_n(K)\subset h_n(\gamma)$ is a continuum of 
diameter $<\varepsilon/2+2\delta_n$ joining $z$, $w$. Since 
$\varepsilon/2+2\delta_n<\varepsilon$ for large $n$, we have
demonstrated that the $(h_n(\gamma))_{n=1}^\infty$ are uniformly
locally connected. The $(h_n(\gamma))_{n=1}^\infty$ being 
uniformly locally connected implies, in turn, that the 
$\Chat\setminus\Omega_n$ are uniformly locally connected 
(see Theorem 2.1 of \cite{MR1217706}).
\end{proof}

\begin{cor}\label{Pomm_cor} 
Let $\Omega$ be a lemniscate domain, $z_0\in\Omega$, set 
$\gamma:=\partial\Omega$ and let $\varepsilon>0$. Then 
there exists a $\delta>0$ so that if 
$h: \gamma \rightarrow h(\gamma)$ is a $\delta$-homeomorphism,
and $\Omega_h$ is the component of $\Chat\setminus h(\gamma)$ 
containing $z_0$, then the Riemann maps 
$\phi_h: \mathbb{D} \rightarrow \Omega_h$, 
$\phi: \mathbb{D} \rightarrow \Omega$ (normalized as
in Lemma \ref{Pomm_lem}) satisfy
$\sup_{z\in\overline{\mathbb{D}}}|\phi(z)-\phi_h(z)|<\varepsilon$.
\end{cor}

\begin{proof} 
If there existed $\Omega$, $z_0$, $\varepsilon$ for which 
the corollary failed, then extracting 
sequences $h_n$, $\phi_n$ from the counterexamples arising 
from $\delta_n:=1/n$ would contradict Lemma \ref{Pomm_lem}.
\end{proof}

\begin{lem}\label{sc_case} 
Let $\Omega$ be a lemniscate domain with $\gamma:=\partial\Omega$,
$\varepsilon>0$ and $U$ a neighborhood of $\gamma$. Then 
there exists $\delta>0$ so that any $\delta$-homeomorphism 
$h: \gamma \rightarrow h(\gamma)$ admits an 
$\varepsilon$-homeomorphic extension 
$h: \overline{\Omega}\rightarrow\overline{\Omega_h}$
satisfying $h(z)=z$ for $z\not\in U$,
where $\Omega_h$ denotes the component of
$\Chat\setminus h(\gamma)$ containing $z_0$.
\end{lem}

\begin{proof} 
Fix $z_0\in\Omega$ and $\varepsilon>0$. By the Riemann 
Mapping Theorem, there exists a conformal map 
$\phi: \mathbb{D} \rightarrow \Omega$ satisfying 
$\phi(0)=z_0$ and $\phi'(0)>0$. For a homeomorphism 
$h: \gamma\rightarrow h(\gamma)$, let 
$\phi_h: \mathbb{D}\rightarrow \Omega_h$ denote a
conformal map also normalized so that $\phi_h(0)=z_0$, $\phi_h'(0)>0$.

Consider the following composition (for now defined only in a
	formal sense, not as a mapping):
\begin{equation}\label{formal_comp}
f:=\phi_h^{-1}\circ h\circ\phi. 
\end{equation}
The map $\phi_h$ is not injective on $\mathbb{T}$ whenever 
$\gamma$ (and hence $h(\gamma)$) has at least one vertex; 
thus $\phi_h^{-1}$ may be multi-valued on $h(\gamma)$. We 
claim that the expression (\ref{formal_comp}) nevertheless 
gives a well-defined homeomorphism 
$f: \mathbb{T} \rightarrow \mathbb{T}$ for $\delta$ small 
enough. Indeed, for $\delta$ small, we have that $v$ is a 
vertex of $\gamma$ having $n$ accesses from $z_0$ (within 
$\Omega$) if and only if $h(v)$ is a vertex of $h(\gamma)$ 
having $n$ accesses from $z_0$ (within $\Omega_h$). Moreover,
the degrees of the vertices occur in the same order 
counterclockwise around $\gamma$, $h(\gamma)$ (as seen from 
$z_0$). Thus, for $\delta$ small, the mapping $f$ is a 
homeomorphism off of the $\phi$-images of the vertices of 
$\gamma$, and extends continuously to a homeomorphism 
$f: \mathbb{T} \rightarrow \mathbb{T}$.

For all $\eta>0$, Corollary \ref{Pomm_cor} implies that there
	exists a $\delta>0$ so that if $h$ is a $\delta$-homeomorphism 
	of $\gamma$, then
\begin{equation}\label{unfm} 
\sup_{z\in\overline{\mathbb{D}}}|\phi(z)-\phi_h(z)|<\eta.
\end{equation}
Thus, for all $\eta>0$, there exists $\delta>0$ so that if $h$
	is a $\delta$-homeomorphism of $\gamma$, then the homeomorphism
\[ f:=\phi_h^{-1}\circ h\circ\phi: \mathbb{T} \rightarrow \mathbb{T} \]
is an $\eta$-homeomorphism.

Let $r<1$ satisfy that $\phi(r\mathbb{T})\subset U$. We claim
that $f: \mathbb{T} \rightarrow \mathbb{T}$ extends to a 
homeomorphism 
$f: \overline{\mathbb{D}}\rightarrow\overline{\mathbb{D}}$
satisfying $f(z)=\phi_h^{-1}\circ\phi(z)$ for $z\in r\mathbb{D}$.
In order to verify this, we need to define $f$ on the annulus 
$\{z : r < |z| <1\}$. First note that  $\phi_h^{-1}\circ\phi$ is uniformly
close to the identity and is conformal in a neighborhood of $r\mathbb{T}$.
Thus  the Cauchy estimates imply 
that the derivative of $\phi_h^{-1}\circ \phi$
is as close to $1$ as we wish (by taking $\delta$ small enough)
on a smaller neighborhood of $r \mathbb{T}$. Hence the image 
of this circle is 
close to perpendicular to each radial segment it touches, and 
hence the Jordan curve 
$\phi_h^{-1}\circ\phi(r\mathbb{T})$ intersects any radial 
line $\{ z : \arg(z)=\theta\}$ in at most one point. Thus, we
may define a homeomorphism $\mathcal{R}$ of the annulus with
boundary components $\mathbb{T}$, $\phi_h^{-1}\circ\phi(r\mathbb{T})$
onto the annulus $\{z : r < |r| < 1\}$ by specifying $\mathcal{R}$
preserves (set-wise) radial lines, $\mathcal{R}$ is the identity 
on $\mathbb{T}$, and $\mathcal{R}$ maps each point 
$\zeta\in\phi_h^{-1}\circ\phi(r\mathbb{T})$ to the point on 
$r\mathbb{T}$ having the same argument as $\zeta$. We can 
interpolate between $f: \mathbb{T} \rightarrow \mathbb{T}$ 
and $\mathcal{R}\circ\phi_h^{-1}\circ\phi: r\mathbb{T}\rightarrow 
r\mathbb{T}$ to give a homeomorphism we call 
$g: \{z : r < |z| < 1\} \rightarrow \{z : r < |z| < 1\} $
by using the interpolation which is linear in logarithmic 
coordinates (i.e. when $g$ is lifted by the exponential to 
give a self-map of a horizontal strip, this lift maps straight
line segments to straight line segments). The homeomorphism 
$g$ is close to the identity as long as $h$ is close to the 
identity. Then $f:=\mathcal{R}^{-1}\circ g$ gives the desired 
homeomorphic interpolation between $f: \mathbb{T} \rightarrow 
\mathbb{T}$ and $\phi_h^{-1}\circ\phi: r\mathbb{T} \rightarrow
\phi_h^{-1}\circ\phi(r\mathbb{T})$, and $f$ is close to the 
identity as long as $\delta$ is small.

Thus, for $\delta$ small enough (depending on $\varepsilon$, 
$\phi$), if $h$ is any $\delta$-homeomorphism of $\gamma$, the map
\[ \phi_h \circ f\circ \phi^{-1} :
\overline{\Omega} \rightarrow \overline{\Omega_h} \]
is an $\varepsilon$-homeomorphic extension of 
$h: \gamma \rightarrow h(\gamma)$, which is the 
identity on $\phi(r\mathbb{D})$. By our choice of $r$,
this means $h$ is the identity outside of $U$.
\end{proof}

\begin{proof} [Proof of Lemma \ref{small move graphs}]
For each component $E$ of $G$, take a neighborhood $U_E$ of 
$E$ which is disjoint from all other components of $G$. Then
by Lemma \ref{sc_case}, there exists a 
$\delta >0$ (depending on $E$, $U_E$, and $ \varepsilon$) 
so that if $h|_E$ is a 
$\delta$-homeomorphism, then $h|_E$ 
extends to an $\varepsilon$-homeomorphism of $\Chat$ which 
is the identity outside of $U$. Taking $\delta$ to be the 
smallest value that works for all components 
$G$, we see that the desired extension of any
$\delta$-homeomorphism $h$ may be defined piecewise in a 
neighborhood of each component of $G$ and the identity outside
of these neighborhoods; this gives the desired 
$\varepsilon$-homeomorphic extension of $h$.
\end{proof}

\section{Approximating Graphs Without Vertices}\label{Alexs_proof}

We are now ready to start the proof of our main result, Theorem 
\ref{main_thm_2}.  The results in Section \ref{smoothing sec}
show that it suffices 
to consider lemniscate graphs that have analytic edges, and form equal 
angles at each vertex. 
In this section,  we prove Theorem \ref{main_thm_2} for lemniscate graphs
with no vertices, so it suffices to assume $H$ is a  
union of pairwise disjoint analytic  Jordan curves.
This special  case 
is sufficient for the proofs of Theorems \ref{power_series_cor} 
and \ref{julia_set_cor} (and its proof only uses Theorem \ref{eps homeo curves}, 
not Theorem  \ref{eps homeo graphs}). 

We recall the convention
that general lemniscate graphs (possibly with vertices) 
will be denoted by $G$, and lemniscate graphs without vertices will be 
denoted by $H$.

For any lemniscate graph $G$, the boundary  of each face is a finite union of 
non-trivial continua, so each face is regular for the Dirichlet problem. 
Thus for each face $B$ and each point $p \in B$,  we can define
the harmonic measure  $ \omega_p = \omega(\cdot , p,B)$  with base point $p$. 
This is a  probability measure  on $\partial B$ that  satisfies 
$$
u(p) = \int_{\partial B} f( \zeta) d \omega_p(\zeta),
$$
where $u$ is the harmonic extension of $f \in C(\partial B)$ to $B$
(e.g., on the unit disk harmonic measure is given by the Poisson 
kernel).
For a bounded face $B$  of $G$, 
the Green's function for $B$ with pole at $p$ is defined as 
\begin{eqnarray} \label{Green defn}
	G_B(z,p) = \int \log|\zeta-p| d\omega_z(\zeta) - \log|z-p|, 
	\text{ for } z \in B, 
\end{eqnarray} 
and we set it to zero outside $B$, by convention.
This is the unique harmonic function 
on $B \setminus \{p\}$ that vanishes on $\partial B$ and has a  
logarithmic pole at $p$. 
For the basic properties of harmonic measure and Green's 
function see, e.g., Chapters II and III of \cite{MR2450237}.

\begin{notation}\label{mu_B_notation} 
Suppose that $H$ is a lemniscate graph without vertices, 
and that $P \subset \Chat$ is a finite set that contains 
at least one point in each grey face of $H$. If 
$B$ is a grey face of $H$, we set $P_B := P \cap B$.
For a grey face $B$ of $H$, we define the signed measure
$\mu_B:=\sum_{p\in P_B}\left(\omega_p -\delta_{p}\right)$
where (as above) $\omega_p$ is harmonic measure for $B$ with 
base point $p$ and and $\delta_{p}$ is a unit mass at $p$. 
Note that $\mu_B$ has total mass zero.
\end{notation}

For lemniscate graphs $H$  without vertices,
the proof of Theorem \ref{main_thm_2} 
will only require considering sets of poles
that have one  element in each grey face of $H$.
However, we will need to consider  multiple poles per face when 
proving Theorem \ref{main_thm_2} for graphs with 
vertices (see Section \ref{graphs_with_verts_sec}), so 
we allow this possibility here.
Throughout this section, we fix a lemniscate graph
$H$ without vertices, a $2$-coloring of the faces of $H$, and a (non-empty)
finite set of points $P_B\subset B$ for each grey face $B$ of $H$. We will
assume $\infty$ is contained in a white face of $H$. Set $P:=\cup_BP_B$.
For $p\in P$, let $B(p)$ denote the face of $H$ containing $p$. 

\begin{notation}\label{u_without_verts}
Given $H$ and $P$, define a $[0,\infty]$-valued function 
$$u(z) = u_{H,P}(z) 
:=\sum_{p\in P}G_{B(p)}(z,p)$$
for $z\in\Chat$, where $ G_{B(p)}(z,p)$
denotes the Green's function for $B(p)$ with pole at $p$. As above, 
we  set $G_{B(p)}(z,p)=0$ for $z\not\in B(p)$. 
\end{notation}

The following two formulas are immediate from the
definitions above.

\begin{prop}\label{log_pot_id} 
For each grey face $B$ of $H$, 
and all $z \in \Chat$, 
\begin{equation}  
\sum_{p\in P_B}G_{B}(z,p) = \int_{\Chat} \log|z-\zeta|d\mu_{B}(\zeta).
\end{equation} 
%
	For all $z \in \Chat$, 
\begin{equation} \label{eqtn_spelled_out}
u_{H,P}(z) = \sum_{B}\int_{\Chat} \log|z-\zeta|d\mu_{B}(\zeta).
\end{equation} 
\end{prop} 

Next,  we approximate  harmonic measure 
by $\delta$-masses at points $(\zeta_j^B)_{j=1}^m \in \partial B$.

\begin{definition}\label{zetapointsdefn}
First consider the case that $\partial B$ consists of a 
single Jordan curve. We fix  some large 
	$m \in \naturals$ and a point $\zeta_1^B \in  \partial B$,
	and for $j>1$ we define  points $\{\zeta_j^B\}_2^m$ so that 
the segment $I_j \subset \partial B$ from 
$\zeta_{j-1}^B$ to $\zeta_j^B$  oriented positively with respect to $B$, 
satisfies
\[\sum_{p\in P_B}\omega(I_j, p, B) =\frac{|P_B|}{m},\]
where $|P_B|$ denotes the number of elements in the set $P_B$. 

When $\partial B$ has more than one component, we follow a similar procedure,
now placing on each component $\gamma$ of $\partial B$ either 
$\lfloor m\cdot\sum_{p\in P_B}\omega(\gamma, B, p) \rfloor$ points or
$\lfloor m\cdot\sum_{p\in P_B}\omega(\gamma, B, p) \rfloor+1$ points, 
so that each edge connecting two adjacent points on $\gamma$ has measure
$=|P_B|/m$, except, possibly,  for one edge which has measure $<  |P_B|/m$.
Define  $\{ \omega^B_m\}$ by placing mass $|P_N|/m$ 
at each point constructed above, and 
	define $\mu_m^B = \omega_m^B -\sum_{p\in P_B} \delta_p$.
\end{definition}

We claim that the measures  $\{ \omega^B_m\}$ converge weak-$\ast$ to 
$\sum_{P \in P_B} \omega_p$.
This follows because harmonic measure is non-atomic, i.e.,  
single points always have harmonic measure zero (this is 
true for general domains in $\reals^n$, $n\geq 2$, but we only need 
it for finitely connected domains in the plane). 
Because of this, the maximum size $\delta_m$  of the arcs $I$
connecting adjacent points 
in Definition \ref{zetapointsdefn} tends to zero as 
$m$ tends to infinity. Any continuous function $g$ on the 
graph $H$ is uniformly continuous and hence 
$|x-y|\leq \delta_m$  implies 
$|g(x)-g(y)|\leq \epsilon_m$ for some sequence $\epsilon_n$ tending 
to zero. Thus 
$$ 
\left| \int g  \sum_{P \in P_B} d  \omega_p
-\int g  d  \omega^B_m \right| \leq \epsilon_m $$
for any continuous function $g$, which is the definition 
of weak convergence of measures. 

For future reference, we note that this convergence also 
holds for  $g(z) = \log |z- \zeta|$ when $\zeta \in H$ but 
$z \not \in H$, since 
$g$ is uniformly continuous outside any neighborhood of the pole. 
Moreover, if $z \in K$ for some compact set $K$ disjoint 
from $H$, then we have a uniform modulus  bound  for the
family $\{ \log |z-\zeta|\}_{z \in K, \zeta \in H}$ depending only 
on $\dist(K, H)$. Thus  
\begin{eqnarray*} 
\int  \log |z-\zeta|   d  \omega^B_m(\zeta)  \to 
\int \log|z-\zeta|   \sum_{P \in P_B} d  \omega_p(\zeta) , 
\end{eqnarray*} 
uniformly on $K$.

\begin{definition}\label{defnsofr_mu_m}
Given the points $\{\zeta^B_j\}_{j=1}^m$
from Definition \ref{zetapointsdefn},
define the rational function 
\begin{equation}\label{defnsofur} r_m(z):=\frac{\prod_{j, B} (z-\zeta_j^B)^{|P_B|}}{\prod_{p\in P}(z-p)^m},
\end{equation}
\noindent 
where the product in the numerator is over all grey faces $B$ and 
$1\leq j \leq m$, and the product in the 
	denominator is over $P=\cup_BP_B$. Set
\begin{equation} u_m(z):=\frac{1}{m}\log|r_m(z)|.
\end{equation} 
\end{definition}

\begin{prop}\label{triv_id}
For all $m\in\mathbb{N}$ and all $c \in \reals$,
we have  $r_m^{-1}(e^{mc}\mathbb{T})=u_m^{-1}(c)$. 
\end{prop}

\begin{proof} This is easy since $|r_m| = e^{cm}$ implies 
$u_m = \frac 1m  \log |r_m| = \frac 1m \log e^{cm}  = c$. 
\end{proof}

\begin{thm}\label{conv_thm}
The sequence $(u_m)_{m=1}^\infty$ converges uniformly to $u_{H,P}$ 
on compact subsets
of $\Chat\setminus (H\cup P)$.
\end{thm}

\begin{proof} 
Recall from Notation \ref{u_without_verts} that $u_{H,P}$
denotes the sum over $P$ of the Green's functions $G_{B(p)}(x,p)$. 
Also recall that we defined 
$\mu_m^B = \omega^B_m - \sum_{p\in P_B} \delta_p$; 
this is the discrete measure with mass
$-1$ at each $p\in P_B$, 
and mass $|P_B|/m$ at each point $\zeta^B_j$ which lies on $\partial B$.
A computation using this definition shows that 
\begin{eqnarray*} 
u_m(z)
& =&
\sum_B\left(  \sum_{p\in P_B}\log\left|\frac{1}{z-p}\right| + 
\sum_{j=1}^m \frac{|P_B|}{m} \log\left|z-\zeta^B_j\right|  \right) \\
&=& 
\sum_B\left(  \int_{\Chat} \log|z-\zeta|d\mu_{m}^B(\zeta) \right).
\end{eqnarray*} 
Our remarks following   Definition \ref{zetapointsdefn} imply 
the measures $(\omega_m^B)_{m=1}^\infty$ converge weak-$\ast$ 
to the measure $\mu_B$ for each $B$. This fact 
	and  Equation (\ref{eqtn_spelled_out}) imply the theorem. 
\end{proof}

\begin{proof}[Proof of Theorem \ref{main_thm_2} for lemniscate 
	graphs with no vertices]
Theorem \ref{eps homeo curves} says that for any $\varepsilon >0$,  $H$ 
is $\varepsilon$-homeomorphic to a union of analytic closed 
curves, so without loss of generality, we may assume $H$ has this form. 
Under this assumption, the function  $u= u_{H,p}$ 
has an analytic extension across the boundary of each 
face, and these extensions have non-zero gradients 
on the boundary, since  $u$ is a sum of Green's functions,
each of which have positive inward pointing normal 
derivative. Hence the gradient of $u$  is non-zero 
in a neighborhood of the boundary of each face.
(But note that the analytic extension of $u$ across 
the boundary of a face does not equal $u$ on the 
adjacent face; $u$ is always 
non-negative, but the analytic extension  
becomes negative when we cross the boundary.) 

Thus for $c>0$ small enough, 
the level set $H_c= \{ z: u(z)   =c\}$ has $n$ components that 
are Jordan curves approximating the $n$ components of $H$.
For such $c$, there is a homeomorphism from  $H_c$ to $H$
given by the steepest descent curves 
of $u$ (i.e., following the vector field $ -\nabla u$),
and we may assume  maximum  diameter of these
connecting curves is as small as we wish, by taking
$c$ small enough. Thus 
by Corollary \ref{small move graphs}, the level set $H_c$  and $H$ are 
$\varepsilon$-homeomorphic if $c>0$ is small enough.  

For $ 0 < s< t$ let 
	$H_{s,t}= \{ z: s \leq u_{H,P}(z)  \leq t \}$.
As $m \nearrow  \infty$, the functions $u_m$ uniformly 
approximate $u$ on  the compact set $H_{c/4,4c}$. 
Hence for $m$ large enough, the level set
$H^m_c = \{z: u_m (z) =c\}$ lies inside $H_{c/2,2c}$. 
Suppose $\delta$ is the distance from $H_{c/2,2c}$ to  
the complement of $H_{c/4,4c}$. 
Since the functions $u_m$  are harmonic,  the uniform 
convergence of $u_m \to u$ on a $\delta$-neighborhood 
of $H_{c/2,2c}$ implies 
the gradients of $u_m$ converge uniformly to the gradient of 
$u$ on $H_{c/2,2c}$. Hence for large enough $m$, we have  
$$  \sup |\nabla u_m  - \nabla u| \leq 
 \frac 12 \inf |\nabla u|,$$
where both the supremum and infimum are take over 
the set $H_{c/2,2c}$.   This
inequality implies $\nabla u_m$ is never zero
and is never perpendicular to $\nabla u$ anywhere on $H_{c/2,2c}$. 
Thus following the gradient line of 
$u$   through a point $z \in H_c$ will reach a unique 
point of $H^m_{c} $ before leaving $H_{c/2,2c}$. 
This defines a homeomorphism $H_c \to H_c^m$.
Hence by Corollary \ref{small move graphs},
$H_c$ and $H_c^m$ are $\varepsilon$-homeomorphic 
if $m$ is large enough, and hence $H^m_c$ is 
$2\varepsilon$-homeomorphic to $H$. 
\end{proof}

\section{Graphs With Vertices: Approximate Pole Placement}
\label{graphs_with_verts_sec}

In this section, we will 
show that any lemniscate graph is $\varepsilon$-homeomorphic 
to a rational lemniscate whose poles can be prescribed with 
error at most $\varepsilon$. 
The fixed point argument that proves we can place the 
poles exactly is given in the next section.

The proof for lemniscate graphs with no vertices (given in 
the previous section) 
is easier than the general case because we 
were free to choose any small enough value $c>0$ to define 
a level line. All small enough  choices 
gave level sets whose components  were Jordan curves, and thus they
automatically have the same topology as the components of $H$.

In the general case,  the union of disjoint curves $H$ is 
replaced by a lemniscate graph $G$ that can have vertices,
and to mimic this graph with level 
sets of  a harmonic function $u$ requires the level sets to run through 
critical values of $u$, of which there are only finitely many. Moreover, 
we need to use the same critical value in every  face. 
Thus only very particular values of $c$ will work, and only after 
we have made delicate adjustments to the shape of the lemniscate $G$.
The description of these adjustments is delayed to 
Section \ref{harmonic_app_appendix}, 
 but we now state the needed result and finish the 
proof of Theorem \ref{main_thm_2} using it. 
Recall   from Notation \ref{u_without_verts} that 
the function  $u_{H, P}$ is a sum of Green's functions 
with  poles in the set $P$ inside the
grey faces (with possibly more than one pole per face).

\begin{thm}\label{whatweneedfromappendix} 
Let $G$ be a lemniscate graph, and fix $\varepsilon>0$.
Then there exists a lemniscate graph $H$ without vertices and 
$\delta>0$ so that each grey face of $G$ is contained 
in a grey face of $H$, and so that $H_\delta = u_{H, P}^{-1}(\delta)$ 
and $G$ are $\varepsilon$-homeomorphic to each other. 
\end{thm}

See Figure \ref{from_G_to_H}. 
In Section \ref{harmonic_app_appendix}, $H$ is constructed 
by modifying $G$ in a small neighborhood of each vertex, and 
then taking a quasiconformal image of the result.
Assuming Theorem \ref{whatweneedfromappendix} for now, 
we continue with  the proof of Theorem \ref{main_thm_2}.

\begin{figure}[htb]
\centerline{
\includegraphics[height=1.1in]{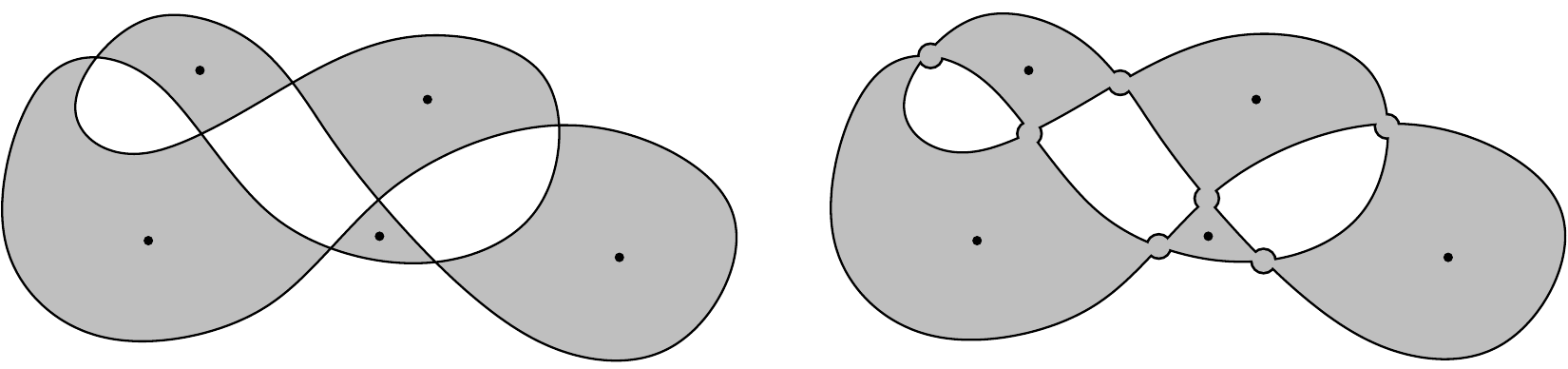}
}
\caption{On the left is a 2-colored lemniscate  graph $G$
and on the right is a vertex-free graph $H$ that approximates $G$. 
Note that the grey face of $H$ contains every grey face of $G$. Here
the grey face of $H$  equals the grey faces of $G$, together with 
disks centered at the vertices of $G$. The actual construction 
involves other steps, and is explained in 
Section \ref{harmonic_app_appendix}. 
}
\label{from_G_to_H}
\end{figure}

\begin{notation}\label{fixingugp}
Throughout the remainder of this section, we will fix $\varepsilon>0$ 
and a lemniscate graph $G$ (perhaps with vertices).
We also fix a $2$-coloring 
of the faces of $G$, and a point $p_B$ in each grey face $B$ of $\Chat\setminus G$. 
Set $P:=\cup_B p_B$. 
Apply  Theorem \ref{whatweneedfromappendix}
to $G$ and  $\varepsilon$ to obtain $\delta>0$, $H$ and $H_\delta$. 
Set $u:=u_{H, P}$. Since $H$ has no vertices, 
the results and definitions of Section \ref{Alexs_proof} apply 
to $u$.
Note  that several distinct grey faces of $G$ may be contained in a
single grey of $H$, so there may be several points in $P$ contained in a
single grey face of $H$.
Let $X$ denote the set of vertices of $H_\delta$. For each $x\in X$ we
denote by $D_x:=D(x,r_x)$ a Euclidean disc centered at $x$ of sufficiently
small radius $r_x>0$ so that $D(x,2r_x)\cap (H\cup P)=\emptyset$, and 
the collection $\{D(x,2r_x)\}_{x\in X}$ are pairwise disjoint. 
\end{notation}

Let $\{r_m\}_{m=1}^\infty$ be as in 
Definition \ref{defnsofr_mu_m}, and let 
$(u_m)_{m=1}^\infty = (\frac 1m \log |r_m|)_{m=1}^\infty$. 
By Theorem \ref{conv_thm}, $u_m \to u$
uniformly on compact subsets of $\Chat\setminus (H\cup P)$. 

\begin{prop}\label{exist_of_limiting_map}
For each $m\in\mathbb{N}$ and $x\in X$, there exists a branch of the
logarithm $\log(r_m)$ in $D_x$ so that $ h_m :=(1/m)\log(r_m)$ converges 
uniformly to an analytic  function $h$ in $D_x$, as $m\rightarrow\infty$. 
\end{prop}

\begin{proof} 
Given a sequence of holomorphic functions 
$ f_n= u_n +i v_n$, it is a general fact that 
that if the real parts $u_n$ converge uniformly on a 
closed disk, then so do the imaginary parts $v_n$, assuming they 
converge at  the center of the disk. This holds because the 
partial derivatives  of $u_n$ converge by the Cauchy estimates, 
and hence so do the partials of $v_n$ by the Cauchy-Riemann 
equations. Thus to prove the proposition, 
it is enough to verify that $\{v_n\}$ converges at $x$.

For large $m$, the map $r_m$ does not have any zeros in $D_x$ since 
for $z \in D_x$, 
\begin{equation*}
	(1/m)\log|r_m(z)|=u_m(z)\rightarrow u(z)>0.
\end{equation*}
Thus for $m$ large enough, a branch of $\log(r_m)$ exists in $D_x$.
Next, for $z\in D_x$  we have 
\begin{align}\label{arg_comp} 
\textrm{Im}\left(\frac{1}{m}\log(r_m(z))\right)=\frac{1}{m}\arg(r_m(z))
 = \sum_{B}\sum_{p\in P_B}\sum_{j=1}^m\frac{1}{m}\arg\left(\frac{z-\zeta_j^B}{z-p}\right),
\end{align}
\noindent
where the first sum $\sum_B$ is over all grey components 
$B$ of $\Chat\setminus H$.
As in the proof of Theorem \ref{conv_thm}, consider the discrete measure
$\mu_m^B$ defined by having mass $-1$ at each $p\in P_B$, and mass $|P_B|/m$ 
at each of the points $(\zeta^B_j)_{j=1}^m$. Thus, the right-hand side 
of (\ref{arg_comp}) can be rewritten as 
\begin{equation} \sum_B\int_{\Chat} \arg\left(z-\zeta\right)d\mu_m^B(\zeta).
\end{equation}
Let  $\mu_B= \sum_{p\in P_B}\left(\omega_p -\delta_{p}\right)$ 
be the measure  from Notation \ref{mu_B_notation}, and recall 
the points $\zeta_j^B$ were chosen in Definition \ref{zetapointsdefn}
so that $\mu_m^B \to \mu_B$
weak-$\ast$. Thus  
\begin{equation}\label{arg_comp3}
\sum_B\int_{\Chat} \arg\left(z-\zeta\right)d\mu_m^B(\zeta) 
\to \sum_B\int_{\Chat} \arg(z-\zeta)d\mu_B(\zeta).
\end{equation}
Combining (\ref{arg_comp})-(\ref{arg_comp3}), we see that the imaginary
part of $(1/m)\log(r_m)$ converges uniformly in $D_x$, 
as desired. 
\end{proof}

For each $x\in X$ and large $m\in\mathbb{N}$, we define a set $X_m(x)$ 
as follows. Note that since $x\in X$ is a vertex of $H_{\delta}$, it 
follows that $x$ is a critical point of $u$ of degree $\textrm{deg}(x)/2-1$;
here we are using $\textrm{deg}(x)$ to denote the degree of $x$ as a 
vertex of $H_{\delta}$ (see Definition \ref{embeddedgraph}). So, for 
instance, degree $4$ vertices are simple ($=$ degree $1$) critical points
of $u$. Since $(u_m)_{m=1}^\infty$ converges to $u$ uniformly on compact
subsets of $\Chat\setminus (H\cup P)$, we have that for each $x\in X$, 
there is a set $X_m(x)$ consisting of $\textrm{deg}(x)/2-1$
many critical points (counted with  multiplicity) of $u_m$, so that each
$x_m\in X_m(x)$ converges to $x$ as $m\rightarrow\infty$. 

	\begin{definition}
We set $h_m:=(1/m)\log(r_m)$ and $h:=\lim_{m\rightarrow\infty}h_m$ in 
$D_x$ (see Proposition \ref{exist_of_limiting_map}). 
	As usual, let $D(h(x),s) $
denote the Euclidean disc centered at $h(x)$ of radius $s>0$.
Fix $s$ sufficiently small so that for all sufficiently large $m$, 
\begin{enumerate}
\item $D(h(x),4s) \subset h(D_x)$, 
\item $h_m^{-1}(D(h(x),2s)) \cap D_x$ 
	is a Jordan domain containing $X_m(x)$,
	and 
\item $h_m: h_m^{-1}(D(h(x),2s)\cap D_x) 
	\rightarrow D(h(x),2s)$ is a proper map.
\end{enumerate}
By taking $s$ smaller, if necessary, we can ensure
(1)-(3) also hold  when $h$ replaces $h_m$ (and $X$ replacing $X_m(x)$
in (1)). For large $m$, let $A^x_m$ be the topological annulus with outer
boundary $h_m^{-1}(\partial D(h(x), 2s))$ 
and inner boundary $h^{-1}(\partial D(h(x),s))$. 
Note $A^x_m \subset D_x$, and 
that there will be corresponding annuli around each point of $X$.
\end{definition} 

We have used $h$ to define the inner boundary of $A^x_m$ and
have used $h_m$ to define its outer boundary.
Our next proposition will show that  $A^x_m$  is
``very close''  to the annulus obtained when we use  $h$ to define 
both boundaries. The measure of closeness is given by quasiconformal maps. 
For the definitions and basic
results about quasiconformal mappings see.e.g., \cite{Ahlfors-QCbook} or 
\cite{MR344463}. 
The dilatation of a quasiregular 
map $g$ is $\mu_g = g_{\overline{z}}/g_z$ and 
we set  $k = \| \mu_f\|_\infty <1$. 
The quasiconformal constant of $g$ is  $K= K(g)  = (k+1)/(k-1) \geq 1$.
Geometrically, $|\mu_g(z)|$ bounds the eccentricity of the ellipse that 
is the image of a circle centered at zero 
under the tangent map of $g$, at almost every point $z$.
When $K(g)$ is close to $1$, then $g$ is close to holomorphic. 

\begin{prop}\label{defnofpsi_m} 
For all large $m$, there exists a quasiconformal mapping 
\[ \psi_m: A^x_m \rightarrow h^{-1}(\{z : s<|z-h(x)|<2s\})\]
so that $\psi_m(z)=z$ on the inner boundary of
$A^x_m$, and $h\circ\psi_m(z)=h_m(z)$ 
on the outer boundary of $A^x_m$.
Moreover, $K(\psi_m) \to 1$ as $m \nearrow \infty$. 
\end{prop}

\begin{proof} One would like to define $\psi_m:=h^{-1}\circ h_m$ on the outer boundary of $A^x_m$, but since $h$ is not injective near $x$, the expression $h^{-1}\circ h_m$ is not a well-defined mapping. We will address this issue by lifting the relevant maps to vertical strips in what follows.

Consider the sets $A^x_m$, 
\[ U:=h^{-1}(\{z : s<|z-h(x)|<2s\}), \,\]
and 
\[ A(s,2s):=\{z : s<|z-h(x)|<2s\}.\]
The sets $U$ and $A(s,2s)$ are topological annuli, and $A^x_m$ 
is also a topological annulus for large enough $m$. 
Without loss of generality  we may assume that $x=h(x)=0$. 
Then the exponential map is a covering from the vertical strip 
 \[  H:=\{z :  \log(s) < \textrm{Re}(z) < \log(2s)\}\]
onto the round annulus $A(s, 2s)$.
Similarly, there are ``curved'', but $2\pi i$-periodic vertical 
strips  $H_A$ and $H_U$ so that $\exp$ is a covering map  
	from these domains to  $A^x_m$ and  $U$, respectively.
Although the boundary  components of $H_A$ and  $H_U$ need not be straight
lines, the left boundary component of $H_A$ coincides with 
the left boundary component of $H_U$ (since the inner boundary of $A^x_m$
coincides with the inner boundary of $U$). 

Let $R_A$ denote the right boundary component of $H_A$.
	By the lifting property, 
	\[ h_m: h_m^{-1}(\partial D(h(x),2s)) 
	\rightarrow  \partial D(h(x),2s)  \] 
	lifts to a periodic homeomorphism 
\[ \widehat{h}_m: R_A \rightarrow \{z : \textrm{Re}(z)=\log(2s)\},\] and 
$h: U \rightarrow A(s,2s)$ lifts to a  periodic homeomorphism 
$\widehat{h}: H_U \rightarrow H$. 
	Since $h_m\rightarrow h$ as $m\rightarrow\infty$,
the lifts may be chosen so that the map $\widehat{h}^{-1}\circ\widehat{h}_m$ 
(defined on $R_A$) converges to the identity as $m\rightarrow\infty$. 
Thus there exists a periodic quasiregular interpolation 
$f_m: H_A\rightarrow H_U$ so that $f_m$ interpolates between 
$\widehat{h}^{-1}\circ\widehat{h}_m$ on $R_A$ and the identity on the left
boundary component of $H_A$, and satisfies $K(f_m)\rightarrow1$ as
$m\rightarrow\infty$. The map
\begin{equation} 
\psi_m:=\exp \circ f_m\circ\log: A^x_m \rightarrow U
\end{equation}
is well-defined and satisfies the conclusions of the proposition.
\end{proof}

\begin{rem} \label{cts dependence} 
In Section \ref{proof_of_main_thm}, we will need a little extra information 
about $\psi_m$. In the construction above, we can take the 
quasiregular interpolating map $f_m$ to be smooth and to equal 
the identity in a neighborhood of the left boundary of $H_A$, 
and to equal the analytic map 
$\widehat{h}^{-1}\circ\widehat{h}_m$ on the right boundary,  $R_A$.
This implies that the dilatation of $\psi_m$ is continuous in $A^x_m$ and 
vanishes near the boundary of $A^x_m$, and thus the dilatation extends to 
be uniformly continuous on the whole plane.
In Section \ref{proof_of_main_thm}, we will 
replace the holomorphic function $h$
by a parameterized family of holomorphic  functions $h^{\myvec{q}}$ 
that depend analytically on a vector $\myvec{q}$ inside 
some open set in $\complex^n$ ($n$ is the number of poles in $P$). The annulus
$A^x_m$ is replaced by a family of annuli $A^{x,\myvec{q}}_m$ that 
move analytically with $q$, and the map $\psi_m$ becomes 
a parameterized family $\psi^{\myvec{q}}_m$. 
In the proof of Theorem \ref{exist_of_fixedpoint}, we will use the fact
that the uniform continuity of the dilatation of 
$\psi_m^{\myvec{q}}$ implies 
that the dilatations of these maps move continuously in   
the supremum norm metric as functions of $\myvec{q}$, i.e., 
that $\myvec{q} \to \mu_{\psi^{\myvec{q}}_m}$ is a 
continuous map from a neighborhood in $\complex^n$ into 
the unit ball of $L^\infty(\complex)$. 
\end{rem}

\begin{definition}\label{g_m_defn}
We define a map $g_m: \Chat\rightarrow\Chat$ as follows.
\begin{equation}\label{g_m_formula} 
g_m(z):= \begin{cases} 
r_m(z) & \textrm{for } z
\textrm{ in the unbounded component of } \Chat\setminus\cup_x A^x_m,\\
	\exp( m\cdot h(z)) & 
\textrm{for } z\textrm{ in a bounded component of } \Chat\setminus\cup_x A^x_m, \\	 
	\exp ( m \cdot h(\psi_m(z))) & \textrm{for } z \in A^x_m. \\	 
  \end{cases} 
\end{equation}
\end{definition}

\begin{prop}\label{props_of_gm}
The map $g_m: \Chat\rightarrow\Chat$ is quasiregular, satisfies 
$K(g_m) \to 1$ as $m \nearrow \infty$, and has a degree
$\textrm{deg}(x)/2-1$ critical point at each $x \in X$ satisfying 
$|g_m(x)|=e^{m\delta}$. 
\end{prop}

\begin{proof} 
We claim that the definitions of $g_m$ agree on both 
components of $\partial A^x_m$. For  $z$ on
the  inner boundary component of $\partial A^x_m$, this 
follows from Proposition 
\ref{defnofpsi_m}  (since $\psi_m(z)=z$  on 
the inner boundary). 
For $z$  on the outer boundary of $\partial A^x_m$, Proposition
\ref{defnofpsi_m} says  $h\circ \psi_m(z)=h_m(z)$, and so 
for such $z$ we have 
\begin{equation*} 
	\exp ( m \cdot h( \psi_m(z)))= \exp ( m \cdot h_m(z)) 
:=\exp(m\frac{1}{m}\log(r_m(z)))=r_m(z),
\end{equation*}
which says that the definitions of $g_m$  agree on the outer
boundary  of $\partial A^x_m$.

The components of $\partial A^x_m$ are analytic curves,
and hence they  removable for quasiregular mappings. 
(e.g.,  Theorem I.8.3 of \cite{MR344463}).
Since $g_m$ is holomorphic 
in $\Chat\setminus\cup_x A^x_m$,  quasiregular in each $A^x_m$, and 
	continuous across each $\partial A^x_m$, we can deduce 
that $g_m$ is quasiregular on $\Chat$
with $K(g_m)=K(\psi_m)$, and $K(\psi_m)\to 1$ as $m \nearrow \infty$ 
by Proposition \ref{defnofpsi_m}. 

Recall that $\textrm{Re}(h)=u$ by  definition in 
Proposition \ref{exist_of_limiting_map}.
Since $u$ has a
degree $\textrm{deg}(x)/2-1$ critical point at each $x\in X$, and since
$g_m(z)=\exp(mh(z))$ in a neighborhood of each $x\in X$, the map $h$ also has a degree $\textrm{deg}(x)/2-1$ critical point at each $x\in X$. 
Moreover, by the definition of $g_m$, we have that for any $x\in X$,
\begin{equation*} 
|g_m(x)|=|\exp(mh(x))|
=\exp\left( \textrm{Re}\left[mh(x)\right] \right) 
= \exp\left( mu(x)\right) = \exp(m\delta).  \qedhere
\end{equation*}
\end{proof}

\begin{thm}\label{gmhomapp}
Let $\varepsilon>0$. Then, for all sufficiently large $m$, we have that
$g_m^{-1}(e^{m\delta}\mathbb{T})$ is $\varepsilon$-homeomorphic to $H_\delta$.
\end{thm}

\begin{proof} 
The homeomorphism $f$ of $g_m^{-1}(e^{m\delta}\mathbb{T})$ onto $H_\delta$
may be described as follows. Since $g_m(z)=\exp(mh(z))$ in a neighborhood 
of each $x\in X$ and $\textrm{Re}(h)=u$, 
	we have $|g_m(z)|=\exp\left( mu(z)\right)$
in a neighborhood of each $x\in X$.
	Thus the sets  $g_m^{-1}(e^{m\delta}\mathbb{T})$
and $H_\delta$ in fact coincide in a neighborhood of each $x\in X$. Thus
we may set $f(z)=z$ for $z$ in such a neighborhood. 

Moreover, $x$, $\tilde{x}\in X$ are connected by an edge in $H_\delta$ if 
and only if $x$, $\tilde{x}$ are connected by an edge in 
$g_m^{-1}(e^{m\delta}\mathbb{T})$.  We claim  that  for 
	any $\eta>0$,  if 
$m$ is large enough, then 
 $f$ extends to a homeomorphism
$f: g_m^{-1}(e^{m\delta}\mathbb{T}) \rightarrow H_\delta$
satisfying $d(f(z),z)<\eta $   for 
all $z\in g_m^{-1}(e^{m\delta}\mathbb{T})$.
This map can be constructed by following the steepest gradient 
descent lines of $u_{H,P}$, just 
as in the proof at the end of Section \ref{Alexs_proof}; 
the details are the same, since $H$ 
consists of disjoint Jordan curves. 
Using Theorem \ref{eps homeo graphs},  this proves that for any 
$\varepsilon >0$, 
$H_\delta$  and $g_m^{-1}(e^{m \delta} \circle)$ are 
$\varepsilon$-homeomorphic if $m$ is large enough.
\end{proof}

\begin{thm}\label{MRMTapp} 
For every $m\in\mathbb{N}$ there exists a quasiconformal mapping $\phi_m$ 
so that $g_m\circ\phi_m^{-1}: \Chat\rightarrow\Chat$ is holomorphic 
(and hence a rational mapping), and moreover
as $m \nearrow \infty$, 
\begin{equation} \label{phi_m_close_to_id}
\sup_{z\in\Chat} d(\phi_m(z),z)\to 0,
\end{equation}
where, as usual,  $d(\cdot, \cdot)$ denotes the spherical metric on $\Chat$.
\end{thm}

\begin{proof}
The first statement follows from the Measurable Riemann Mapping theorem.
The map $\phi_m$ is uniquely determined once we choose a normalization; 
we may choose to normalize $\phi_m$ so that $\phi_m(z)=z+O(1/|z|)$ as 
$z\rightarrow\infty$, or else to normalize $\phi_m$ so as to fix $3$ 
distinct points in $X$ (either normalization will work in what follows). 
With either choice, the relation (\ref{phi_m_close_to_id}) follows since
$K(g_m) \to 1$ by Proposition \ref{props_of_gm}. 
\end{proof}

\begin{definition}\label{defn_of_rm}
We set $R_m:=g_m\circ\phi_m^{-1}$ where $\phi_m$ is as in
Theorem \ref{MRMTapp}, so that $R_m$ is a rational mapping.
\end{definition}

\begin{prop}\label{final_sec4_result}
Let $\varepsilon>0$. For all large $m$, we have that the rational lemniscate
$R_m^{-1}(e^{mc}\mathbb{T})$ is $\varepsilon$-homeomorphic to $H_\delta$.
\end{prop}

\begin{proof}
We already know that $g_m^{-1}(e^{mc} \circle)$ is 
$(\varepsilon/2)$-homeomorphic to $H_\delta$ if $m$ is 
sufficiently large, and we know that the quasiconformal map 
$\phi_m$ is an $(\varepsilon/2)$-homeomorphism if $m$ is large enough. 
Thus 
$ R_m^{-1}(e^{mc}\mathbb{T})=\phi_m(g_m^{-1}(e^{mc}\mathbb{T}))$ 
is $\varepsilon$-homeomorphic to $H_\delta$.
\end{proof}

We have now proved (assuming Theorem \ref{whatweneedfromappendix})
that any lemniscate graph is $\varepsilon$-homeomorphic to 
a rational lemniscate with one pole in each grey component, and that these
poles may be specified with error at most $\varepsilon$. Indeed, three 
of the poles of $R_m$ can be placed exactly by  applying 
a linear fractional transformation. 
In the next section, we show that all the poles can be prescribed exactly. 

\section{Exact Placement of Poles}\label{proof_of_main_thm}

To place the poles exactly, we will apply a fixed point argument.
More precisely,
given a desired set of $n$  poles $P$, and fixing a
large integer $m$
(that determines a discrete approximation to harmonic measure),
we introduce  a
parameterized family $R_m^{\myvec{q}}$ of rational
functions, where $\myvec{q}$ ranges over
a neighborhood of  $P$, considered as a point in $\complex^n$.
We will show that  if $m$ is sufficiently large, then
there exists a   value of $\myvec{q}$
so that the poles of $R_m^{\myvec{q}}$ are exactly $P$.

Throughout this section we will fix $G$, $\varepsilon$ and $P$
as in Theorem \ref{main_thm_2}. By applying a M\"obius transformation,
we may assume without loss of generality that the union of the
grey components of $\Chat\setminus G$ is contained in $D(0,1/2)$.
It will be convenient to list the points in
$P$ as $\myvec{p}:=(p_1, ..., p_n)\in\mathbb{C}^n$, so that
$\cup_j\{p_j\}=P$. Let $H$, $\delta$, $u:=u_{H, P}$ be as in
Notation \ref{fixingugp}; namely $H$ is a lemniscate graph without
vertices whose grey faces properly contain the grey faces of
$G$, and $H_{\delta}:=u^{-1}(\delta)$ is $\varepsilon$-homeomorphic to $G$.

The index $m$ will denote the  parameter of $r_m$, the rational function
defined in  Equation (\ref{defnsofur}).
As the reader may recall, $r_m$ was
defined by cutting the boundary of each grey face of $H$ into
approximately $m$ arcs of approximately equal  harmonic measure.
In this section, $n$ (the number of points in $P$)
will remain fixed, but we shall take $m$ as
large as is needed to make our arguments work.

\begin{lem}\label{Royden_lemma}
Given a set  of $n$ distinct points $P=\{ p_1, \dots, p_n\}
\subset D(0, 1/2)$, there exists $\rho>0$ so that for each
$\myvec{q}=(q_1, ..., q_n)\in\prod_{j=1}^n D(p_j,\rho)$,
there exists an injective, holomorphic mapping
$\psi_{\myvec{q}}: \mathbb{D} \rightarrow \psi_{\myvec{q}}(\mathbb{D})$ satisfying:
\begin{enumerate}
\item $\psi_{\myvec{q}}(q_j)=p_j$ for $1\leq j \leq n$.
\item $\psi_{\myvec{p}}(z)=z$ for all $z\in\mathbb{D}$, and
\item $\psi_{\myvec{q}}$ depends continuously on $\myvec{q}$.
\end{enumerate}
\end{lem}

\begin{proof}
This is stated without proof as Lemma 1 in \cite{MR64147}, but
we will give a proof for the sake of completeness.

Set
$\psi_{\myvec{q}}(z):=z+f_{\myvec{q}}(z)$ where $f_{\myvec{q}}$
is the Lagrange interpolating polynomial for data
$f_{\myvec{q}}(q_j)=p_j-q_j$ for $1\leq j \leq n$. More
explicitly,
 \begin{eqnarray} \label{defn f_q}
         f_{\myvec{q}} (z)
   =\sum_{j=1}^n (p_j -q_j) \prod_{k\ne j} \frac {z-q_k}{q_j-q_k},
 \end{eqnarray}
is a polynomial of degree at most $n-1$ that takes the $n$ given
values  $\{p_j-q_j\}_1^n$ at the $n$ specified distinct  points
$\{q_j\}_1^n$.
It is easy to check from this definition
that (1) and  (2) hold, and that $\psi_{\myvec{q}}$
depends continuously on $\myvec{q}$ as long as the
components of $\myvec{q}$ are distinct.
Let  $\delta$ be half the minimal distance between
distinct points of $P$.
If $\rho < \delta/2$  and $|q_j - p_j| < \rho$ for
all $j=1, \dots, n$,  then the components of $\myvec{q}$
are distinct and  we even have
$|q_j-q_k| \geq\delta$ for $j \ne k$.

It only remains to verify that $\psi_q$ is injective.
Our assumption that $\rho < \delta/2$ implies
that for $ z \in D(0,3/2)$
each of the $n$ products  on the right side of (\ref{defn f_q})
is bounded by $(2/\delta)^n$. By the Cauchy estimates, this implies
$|f_{\myvec{q}}'| = O(\rho \cdot  n (2/\delta)^n)$ on $ \disk$.
Let $\rho$ be sufficiently small (depending on $\delta$ and $n$)
so that $|f_{\myvec{q}}'(z)|<1$ for
        $z\in \disk$. Then integration gives the inequality:
\begin{equation}
|\psi_{\myvec{q}}(z_1)-\psi_{\myvec{q}}(z_2)|
=\left|\int_{z_1}^{z_2}(1+f_{\myvec{q}}')
\right|\geq|z_1-z_2|-\int_{z_1}^{z_2}|f_{\myvec{q}}'|
> 0,
\end{equation}
for $ z_1, z_2\in\disk$.
This proves the injectivity of $\psi_{\myvec{q}}$.
\end{proof}

\noindent We henceforth fix $\rho>0$ as in the conclusion of Lemma \ref{Royden_lemma}.

\begin{definition}\label{parametrized_u}
For each $\myvec{q}\in\prod_{j=1}^n D(p_j,\rho)$, we introduce
parameterized versions  $G_{\myvec{q}}$, $H_{\myvec{q}}$,
$u_{\myvec{q}}$ of the objects $G$, $H$, $u$ as follows.
Set $G_{\myvec{q}}:=\psi_{\myvec{q}}^{-1}(G)$,
$H_{\myvec{q}} := \psi_{\myvec{q}}^{-1}(H)$,
$P_{\myvec{q}} := \psi_{\myvec{q}}^{-1}(P)$
and define $u_{\myvec{q}}$
by setting $u_{\myvec{q}}(z):=u_{H, P} \circ\psi_{\myvec{q}}(z)$
for $z$ in a grey face of $H_{\myvec{q}}$, and $u_{\myvec{q}}(z):=0$ otherwise.
As before, we set $u:= u_{H,P} $ to simplify notation.
\end{definition}

\noindent It will be useful to note that $G_{\myvec{p}}=G$,
$H_{\myvec{p}}=H$, and $u_{\myvec{p}}=u$ by Lemma \ref{Royden_lemma}(2).
Moreover, $u_{\myvec{q}}^{-1}(\delta)=\psi_{\myvec{q}}^{-1}(u^{-1}(\delta))=
\psi_{\myvec{q}}(G)$ is $\varepsilon$-homeomorphic to $G$ for all
$\myvec{q}\in\prod_{j=1}^n D(p_j,\rho)$ after taking
$\rho$ smaller if need be. On the other hand,
Notation \ref{fixingugp} defines the function
$u_{H_{\myvec{q}}, \psi_{\myvec{q}}(P)}$ for each
$\myvec{q}\in\prod_{j=1}^n D(p_j,\rho)$. In fact, we have the following.

\begin{prop}\label{equivdefnsofu}
For each $\myvec{q}\in\prod_{j=1}^n D(p_j,\rho)$,
we have $u_{\myvec{q}}=u_{H_{\myvec{q}}, P_{\myvec{q}}}$.
\end{prop}

\begin{proof}
Both functions vanish except in the grey faces of $H_{\myvec{q}}$,
where they are both harmonic except for
logarithmic poles at $P_{\myvec{q}}$. The result follows
from the maximum principle.
\end{proof}

Proposition \ref{equivdefnsofu} implies that the
definitions of Section \ref{graphs_with_verts_sec} apply to
each $u_{\myvec{q}}$ to produce parameterized families
\begin{equation}\label{firstparfamily}
\psi_m^{\myvec{q}}, g_m^{\myvec{q}}, \phi_m^{\myvec{q}},
\textrm{ and }R_m^{\myvec{q}}:=g_m^{\myvec{q}}\circ(\phi_m^{\myvec{q}})^{-1}
\end{equation}
so that
\begin{equation}
\psi_m^{\myvec{p}}=\psi_m, g_m^{\myvec{p}}=g_m, \phi_m^{\myvec{p}}
=\phi_m, \textrm{ and }R_m^{\myvec{p}}=R_m.
\end{equation}
Moreover, the results of Section \ref{graphs_with_verts_sec} apply
to the maps listed in  (\ref{firstparfamily}) for each $\myvec{q}\in\prod_{j=1}^n
D(p_j,\rho)$.
The important point is that all these functions depend continuously
on the parameter $\myvec{q}$.

\begin{definition}
For each $m\in\mathbb{N}$, define a mapping
$\Psi_m: \prod_{j=1}^n D(p_j,\rho) \rightarrow \Chat^{n}$ by setting
\[\Psi_m(\myvec{q}):=((\phi_m^{\myvec{q}})^{-1}(p_1), ...,
(\phi_m^{\myvec{q}})^{-1}(p_n)). \]
\end{definition}

Note that
\begin{equation}\label{poles_of_g_m_q}
(g_m^{\myvec{q}})^{-1}(\infty)=\myvec{q},
\end{equation}
where we are abusing notation slightly in
(\ref{poles_of_g_m_q}) by setting $(q_1, ..., q_n)=\{q_1, ..., q_n\}$.
So if $\myvec{q}$ is a fixed point of $\Psi_m$ then $\Psi_m(\myvec{q}) =
\myvec{q}$, or equivalently
$$ (\phi_m^{\myvec{q}})^{-1} (p_j) = q_j, \quad j=1, \dots, n,$$
or
$$  p_j = \phi_m^{\myvec{q}}(q_j), \quad j=1, \dots, n,$$
or
$$  \myvec{p} =  \phi_m^{\myvec{q}} ( (g_m^{\myvec{q}})^{-1} (\infty) ).$$
Therefore, if  we  can prove that $\Psi_m$ has a 
fixed point $\myvec{q}$, and we set
$R_m = g_m^{\myvec{q}} \circ (\phi_m^{\myvec{q}})^{-1}$,
then $R_m$ is a rational function with poles equal to
$$R_m^{-1}(\infty) = \phi_m^{\myvec{q}}((g_m^{\myvec{q}})^{-1} (\infty) )
= \myvec{p}.$$
Assuming $\Psi_m$ has a fixed point $\myvec{q}$ if $m$ is large enough
(Theorem \ref{exist_of_fixedpoint} below),
and assuming Theorem \ref{whatweneedfromappendix} (proven in Section
\ref{harmonic_app_appendix}) 
 we can complete the proof of our main result.

\begin{proof} [Proof of Theorem \ref{main_thm_2}]
The map $R_m^{\myvec{q}}$ with  $m$ and $\myvec{q}$ chosen as above
satisfies the conclusions of Theorem
\ref{main_thm_2}. Indeed,  we just showed that
        $(R_m^{\myvec{q}})^{-1}(\infty)=P$
and the remaining conclusions of
Theorem \ref{main_thm_2} were already verified in Section
\ref{graphs_with_verts_sec} (see Proposition \ref{final_sec4_result}).
\end{proof}

Let us now prove that $\Psi_m$ does, indeed, have a fixed point.

\begin{thm}\label{exist_of_fixedpoint}
Let $\rho>0$ be as given in Lemma \ref{Royden_lemma}.
If $m\in\mathbb{N}$ is sufficiently large, then
$\Psi_m: \prod_{j=1}^n D(p_j,\rho) \rightarrow \Chat^{n}$ has a fixed point.
\end{thm}

\begin{proof} By Brouwer's fixed point theorem \cite{MR1511644},
it suffices to verify that:
\begin{enumerate}
\item $\Psi_m(\prod_{j=1}^n D(p_j,\rho)) \subset \prod_{j=1}^n D(p_j,\rho)$
for sufficiently large $m$,
\item $\Psi_m$ is continuous.
\end{enumerate}
By (\ref{phi_m_close_to_id}),
        the image set  $\Psi_m(\prod_{j=1}^n D(p_j,\rho)) $ converges to
the point $\myvec{p}$ as $m \nearrow \infty$, so it is certainly
contained inside $\prod_{j=1}^n D(p_j,\rho)$ for sufficiently large $m$.

To prove the  continuity of $\Psi_m$,  we write it
as the composition
$\Psi_m=\Psi_m^{1}\circ\Psi_m^{2}$, where
\[ \Psi_m^{1}: \{ \mu \in L^{\infty}(\Chat) :
||\mu||<1\} \rightarrow  \Chat^{n} \textrm{ is defined by }
\Psi_m^{1}(\mu):=(\phi_\mu^{-1}(p_1), ..., \phi_\mu^{-1}(p_n)).\]
and
\begin{equation*}
\Psi_m^{2}: \prod_{j=1}^n D(p_j,\rho) \rightarrow
\{ \mu \in L^{\infty}(\Chat) : ||\mu||<1\} \textrm{ is defined by }\\
\Psi_m^2(\myvec{q}):=(g_m^{\myvec{q}})_{\overline{z}}/(g_m^{\myvec{q}})_{z}
\end{equation*}
Here  $\phi_\mu$ is the unique quasiconformal mapping
(normalized so $\phi(z)=z+O(1/|z|)$ as $z\rightarrow\infty$)
so that $(\phi_\mu)_{\overline{z}}/(\phi_\mu)_{z}=\mu$.
The map $\Psi_m^{1}$ is continuous by a standard result
saying that pointwise evaluation of
quasiconformal maps depends continuously
on the dilatation, e.g.,   Theorem I.7.5 of \cite{MR1230383}.

The continuity of $\Psi_m^2$ is the claim that the
 complex dilatation  of $g_m^{\myvec{q}}$  moves
continuously  in the supremum norm as a function
of $\myvec{q}$. By  Definition \ref{g_m_defn},
$g_m^{\myvec{q}}$ is holomorphic except where  it equals
$g_m^{\myvec{q}} = \exp \circ( m\cdot h^{\myvec{q}}) \circ \psi_m^{\myvec{q}}$.
Since post-composing
by holomorphic maps does not change the dilatation of
a quasiregular map, the dilatation  of
$g_m^{\myvec{q}}$ is the same as  the dilatation of
$\psi_m^{\myvec{q}}$, and the latter  dilatation
moves continuously in $L^\infty$ as a function
of $\myvec{q}$; this was explained in   Remark \ref{cts dependence}.
\end{proof}

\section{Proof of Theorem \ref{power_series_cor}: 
quantitative Runge's theorem}\label{sec_cor_proof}    

The proof of Theorem \ref{power_series_cor}
roughly follows the proof of Theorem I in \cite{MR1501732} 
(which assumes the set  $K$ is connected,
and $P=\{\infty\}\subset\Chat\setminus K$), 
except that we replace their application of a weaker polynomial 
lemniscate result with our Theorem \ref{main_thm_2}. There are also
some other non-trivial adjustments to handle the case of disconnected $K$. 

Let $K$, $P$ and $f$ be as in the statement of 
Theorem \ref{power_series_cor}.
Let $U$ be the neighborhood of $K$ in which $f$ is holomorphic. 

\begin{lem} 
Given $K$, $U$ and $P$ as above, we 
can find  $K' \supset K$, 
$ U' \subset U$, and $P' \subset P$, 
so that  both $K'$ and  $U'$ are bounded by finitely many pairwise disjoint
smooth Jordan curves, and that $P'$ contains exactly one 
point in each connected component of $\Chat \setminus K'$. 
\end{lem}

\begin{proof} 
First replace $U$ by a subset $U'$ that 
still covers $K$ and that is bounded 
by finitely many disjoint smooth curves.
For example, take a union of sufficiently small 
grid squares and round the corners of the resulting 
polygon, as in Figure \ref{Rounding}. 
Then $U'$ and  the connected components of 
$\Chat \setminus K$ cover $\Chat$.  This is an open 
cover of a compact space,  
so  $U'$ and a finite subcollection $\{\Omega_k\}_1^n$ of 
the components of $\Chat \setminus K$ also cover $\Chat$. 
Let $P' \subset P $ be the $n$ points contained in these 
finitely many open sets. Since $K$ has positive distance 
from $ \Chat \setminus U'$, we see that $\Omega_j \setminus U'$
is a compact subset of $\Omega_j$. Replace each $\Omega_j$ by a 
connected subset $\Omega_j'$ that is bounded by 
finitely many smooth curves, and that contains the 
compact set $P \cup (\Omega_j \setminus U')$ (again 
we may take a union of small squares and round 
the remaining corners).  
Let $K' = \Chat \setminus  \cup_{j=1}^n \Omega_j'$.
Then $K \subset K' \subset U' \subset U$, and $P'$ 
contains exactly one point in each complementary component 
of $K'$.  Thus the claim is  verified.
\end{proof}

\begin{figure}[htb]
\centerline{
\includegraphics[height=1.75in]{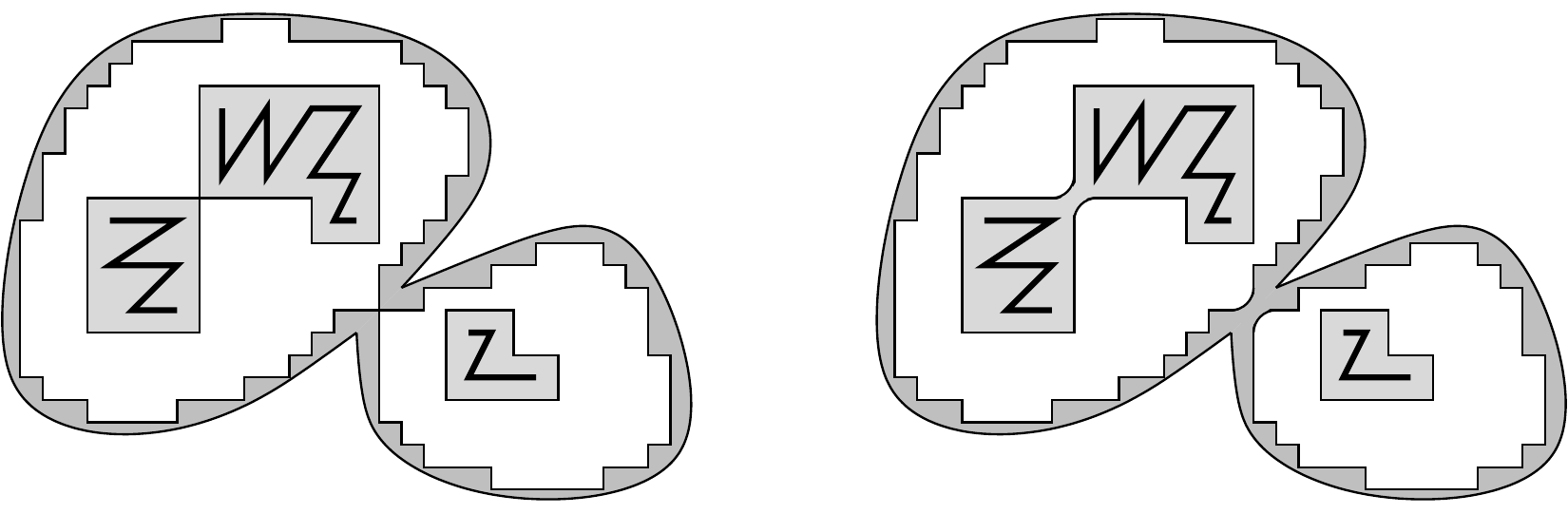}
}
\caption{
Replacing $K$ and $U$ by regions with smooth  Jordan curve boundaries.
If $\delta <\ll \dist(K, \partial U)$ we can take all squares
from a $\delta$-grid that lie inside $U$, and then round 
the corners of the resulting polygon to get  a 
an open set $U' \subset U$ that covers $K$ and has 
smooth boundaries. A similar construction gives 
a smooth set $K'$ containing $K$. 
}
\label{Rounding}
\end{figure}

It suffices to prove Theorem \ref{power_series_cor} for these new 
sets $K',U', P'$; this immediately implies the same result for 
the original sets $K,U,P$.
To simplify notation, we henceforth 
refer to the new sets simply as $K$,  $U$ and $P$ (dropping the 
prime notation).

Considering $\partial K$ 
as a lemniscate graph, each of its faces  either contains points of
$K$ or a single point of $P$, and we color these faces  white and  grey
respectively. 
Consider the function 
$u_{\partial K, P}$ as in  Notation \ref{u_without_verts}. 
Recall that inside each grey face of $\partial K$, this
is the Green's function for that  face  with pole at the 
corresponding point of $P$ (there is one such point per grey face), 
and $u_{\partial K,P}$  is zero outside all the grey faces.
If $S >1$,  set 
$C_S:=\{z : u_{\partial K, P}(z)=\log S\}$. Then fix some $R>1$
so that $C_R$ separates $K$ from $\partial U$.
When $R$ is close to one, $C_R$ is a union of  level lines
of Green's functions for all the grey faces, and this  union
is as close to  $\partial K$ as we wish, so such a choice is possible.

\begin{lem}\label{first_app_th_lem} 
For every $\rho\in(1,R)$, there exists a rational mapping $r$ 
so that the lemniscates $L_r(1)=\{z: |r(z)|=1\}$ and
$L_r((R/\rho)^d) = \{z: |r(z)|=(R/\rho)^d\}$ 
both separate $K$ from $C_R$, where $d=\textrm{deg}(r)$.
\end{lem}

\begin{proof}
Fix $ \rho \in (1, R)$ and 
consider $C_{\rho} =\{ z: u_{\partial K,P}(z) = \log \rho\}$.
By Theorem \ref{main_thm_2}, there exists 
a rational mapping $r$ with $r^{-1}(\infty)=P$ so that the 
lemniscate $L_r(1) $ separates $K$ from $C_{\rho}$. Note that 
\begin{equation}\label{form_of_Greens_func} 
u_{L_r, P}(z) = \frac{1}{d}\log|r(z)|
\end{equation}
and 
\begin{equation} \label{form_of_Greens_func2} 
u_{C_{\rho}, P}(z)=u_{\partial K,P}(z)-\log(\rho) 
\end{equation}
for all $z$. 
Since $L_r$ separates $K$ from $C_{\rho}$, we have that 
$u_{C_{\rho}, P}-u_{L_r, P}<0$ on $C_\rho$ and hence by the maximum principle
\begin{equation}\label{needed_ext_ineq}
u_{C_{\rho}, P}-u_{L_r, P}<0 \textrm{ on } \{z : 0 < u_{C_{\rho}, P}(z) < \infty\}.
\end{equation}
Now assume  $z \in  L_r((R/\rho)^d)$, i.e., 
assume that  $|r(z)|=(R/\rho)^d$. Then by (\ref{form_of_Greens_func})
we have that $u_{L_r, P}(z)=\log(R/\rho)$, and so by (\ref{needed_ext_ineq}) 
we conclude $u_{C_{\rho}, P}(z) < \log(R/\rho)$. 
Thus (\ref{form_of_Greens_func2}) implies that 
for $z \in L_r((R/\rho)^d)$ 
$$u_{\partial K,P}(z) = u_{C_{\rho}, P}(z)+ \log(\rho)<
\log(R/\rho) + \log \rho  =\log(R). $$
Hence the lemniscate 
$L_r((R/\rho)^d)$ separates $K$ from $C_R$. Since we
already arranged for the lemniscate $ L_r(1)$ to separate
$K$ from $C_{\rho}$ (and it separates $K$ from $C_R$ since 
$1 < \rho < R$), the proof of the lemma is finished.
\end{proof}

Recall that we have fixed $R>1$ so that $C_R$ separates $K$ 
from $\partial U$. 
For the remainder of this section, we fix some $\rho\in(1,R)$
and fix  a rational map $r$ by applying Lemma \ref{first_app_th_lem} 
using this $\rho$. 
After applying a M\"obius
transformation, if necessary, we may assume 
that $\infty\not\in K$ and $\infty\in P$. 
Thus if  $r(z)=p(z)/q(z)$, where $p$ and $q$ are polynomials, 
then  $r$ has a pole at $\infty$ and hence 
$ \deg(r) = \deg(p) > \deg(q)$.

We will refer to the \emph{interior} 
of $C_R$ as the union of the components of $\Chat\setminus C_R$
which do not contain a point of $P$.  
Equivalently, this is the union of components where $u_{\partial K,P} < R$, 
and hence these are the components that  contain points of $K$.
If we orient the individual curves in 
$C_R$ so the interior regions  (shown as white in Figure \ref{Winding})
are on the left, and note that 
the unbounded component is not in the interior of $C_R$, then any white point 
is surrounded by an odd number of nested  boundary curves, with alternating 
orientations. 
The outermost curve is oriented in the counterclockwise 
direction, so the total winding number of $C_R$ around any white point 
must equal one. Similarly, the winding number around any grey point is zero.

\begin{figure}[htb]
\centerline{
\includegraphics[height=2.0in]{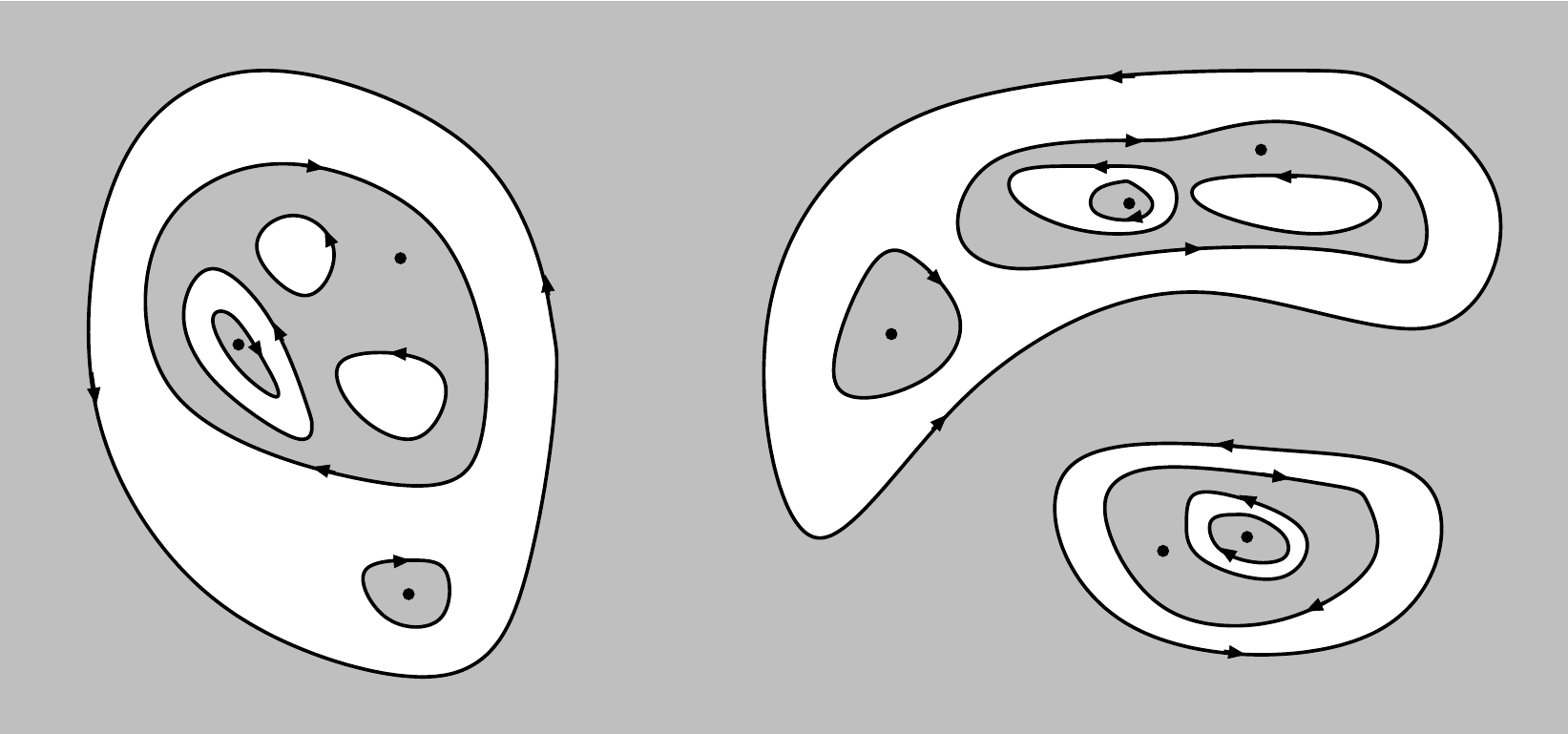}
}
\caption{
The white regions are the interior of $C_R$, and the gray regions each 
contain a point of $P$ (black dots, plus $\infty$). Each component 
of $C_R$ is oriented with the interior on the left, and the 
total winding number of $C_R$ around any white point is one.
}
\label{Winding}
\end{figure}

We can now define the rational maps that will prove 
Theorem \ref{power_series_cor}. 

\begin{definition} 
	For $n\geq1$,   $d = \deg(r)$, and $z$ in the interior 
	of $C_R$, define
\begin{equation}\label{Rdn-1defn}
R_{dn-1}(z):=f(z)-\frac{1}{2\pi i}\int_{C_{R}} \frac{f(\zeta)}{\zeta-z}
\left[ \frac{r(z)}{r(\zeta)}\right]^nd\zeta .
\end{equation}
\end{definition}

\begin{lem}\label{isarationalmap} 
For each $n$, $R_{dn-1}$ is a rational map of degree $\leq dn-1$
with $R_{dn-1}^{-1}(\infty)\subseteq P$.
\end{lem}

\begin{proof} 
Assume $z$ is interior to $C_R$. By the remarks above, 
$C_R$ has winding number 
one around $z$.  By the homology version of 
Cauchy's integral theorem (e.g.,  Theorem 2.41
of \cite{MR4422101}), we have that
\begin{equation}\label{int_rep} 
f(z)=\frac{1}{2\pi i}\int_{C_{R}} \frac{f(\zeta)}{\zeta-z}d\zeta.
\end{equation}
Thus, (\ref{Rdn-1defn}) and (\ref{int_rep}) combine to give
\begin{equation*}
R_{dn-1}(z)=\frac{1}{2\pi i}\int_{C_{R}}f(\zeta)\left[\frac{1}{\zeta-z}-
	           \frac{r^n(z)}{(\zeta-z)r^n(\zeta)}\right]d\zeta.
\end{equation*}
Recall that $r(z)=p(z)/q(z)$, where $p$ and $q$ are polynomials 
with $\deg(p) > \deg(q)$.   Observe
\begin{equation*} 
\frac{1}{\zeta-z}-\frac{r^n(z)}{(\zeta-z)r^n(\zeta)} =
\frac{1}{q^n(z)}\frac{p^n(\zeta)q^n(z)-p^n(z)q^n(\zeta)}{(\zeta-z)p^n(\zeta)},
\end{equation*}
and hence 
\begin{equation}\label{RNlastexp}
R_{dn-1}(z)=\frac{1}{2\pi i}\frac{1}{q^n(z)}\int_{C_{R}}f(\zeta)
\left[\frac{p^n(\zeta)q^n(z)-p^n(z)q^n(\zeta)}
	{(\zeta-z)p^n(\zeta)}\right]d\zeta.
\end{equation}
Since $p^n(\zeta)q^n(z)-p^n(z)q^n(\zeta)$ 
	vanishes when $\zeta=z$, we see that 
	for fixed $\zeta$, 
\[ \frac{p^n(\zeta)q^n(z)-p^n(z)q^n(\zeta)}{(\zeta-z)p^n(\zeta)}\]
is a polynomial in $z$ of degree  at most  $dn-1$.
Therefore 
\begin{equation}\label{isapol}
	\int_{C_{R}}f(\zeta)\left[\frac{p^n(\zeta)q^n(z)-p^n(z)q^n(\zeta)}
	{(\zeta-z)p^n(\zeta)}\right]d\zeta
\end{equation} 
is a linear combination of polynomials with degrees at most $dn-1$, 
and hence is 
a polynomial of degree  at most $ dn-1$.
From (\ref{RNlastexp}) and (\ref{isapol}) we conclude that $R_{dn-1}$ 
is a rational map, whose numerator has degree at most $dn-1$, and 
whose denominator has degree $n \deg(q)  < n \deg(p) \leq dn-1$. 
Thus the degree of $R_{dn-1}$ is $\leq dn-1$. Moreover,
from (\ref{RNlastexp}) and (\ref{isapol}) we also see that the 
poles of $R_{dn-1}$ are a subset of the poles of $r$ (the zeros of $q$). 
\end{proof}

\noindent \emph{Proof of Theorem \ref{power_series_cor}}.
Let $z\in K$. From (\ref{Rdn-1defn}), we see that
\begin{equation}\label{orig_est}
|R_{dn-1}(z)-f(z)| \leq \frac{1}{2\pi}\int_{C_R} \frac{|f(\zeta)|}
	{|\zeta-z|}\left| \frac{r(z)}{r(\zeta)}\right|^n|d\zeta|
\end{equation}
Note that 
\begin{equation}\label{next_est}
\int_{C_R} \frac{|d\zeta|}{|\zeta-z|} < \frac{\textrm{length}(C_R)}{\textrm{dist}(K, C_R)}.
\end{equation}
Furthermore, Lemma \ref{first_app_th_lem} implies that
for all $z \in K$ and $ \zeta \in C_R$ we have 
\begin{equation}\label{next_est2}
\left| \frac{r(z)}{r(\zeta)}\right|^n < \frac{1}{(R/\rho)^{dn}} .
\end{equation}
Combining (\ref{orig_est}) with (\ref{next_est}) and (\ref{next_est2}), we see that 
\begin{equation*}   
|R_{dn-1}(z)-f(z)| \leq \frac{\sup_{w\in K}|f(w)|\cdot\textrm{length}(C_R)}
	{2\pi\cdot\textrm{dist}(K, C_R)}\cdot\frac{1}{(R/\rho)^{dn}}.
\end{equation*}
Thus, setting
\begin{equation*}
A:=\frac{\sup_{w\in K}|f(w)|\cdot\textrm{length}(C_R)}
{2\pi\cdot\textrm{dist}(K, C_R)\cdot(R/\rho)} 
	\quad \textrm{ and }  \quad 
	B:=(R/\rho),
\end{equation*}
we see that 
\begin{equation}\label{final_est_need} 
|R_{dn-1}(z)-f(z)| \leq \frac{A}{B^{dn-1}}
\end{equation}
as desired. Since $R_{dn-1}$ is a rational map of degree $\leq dn-1$
with poles at $P$ by Lemma \ref{isarationalmap}, 
the inequality (\ref{final_est_need}) proves Theorem \ref{power_series_cor} 
for the subsequence $(R_{dn-1})_{n=1}^\infty$. Setting 
\[ R_m:=R_{dn-1} \textrm{ for } dn-1\leq m < d(n+1)-1 \]
and increasing $A$ by a factor of $B^d$ in (\ref{final_est_need}), 
we readily deduce the desired estimate for all $m\in\mathbb{N}$. 
\qed

\begin{rem}
It is worth noting that in the special  case when $f$ is analytic on a disk 
$U$ centered at $p$, we may take   $r(z) = (z-p)$,   and take
$C_R \subset U$ to be a circle centered at $p$. 
Then the residue theorem shows that 
(\ref{RNlastexp}) computes the degree $n$  truncation of the power series 
for $f$ around $p$. Thus our calculation mimics 
the well known fact that these truncations converge 
geometrically fast to $f$ on compact subsets of the disk of convergence.
\end{rem}

The following fact is not needed for the proof of Theorem \ref{power_series_cor},
but it illustrates the structure of the approximants $R_{dn-1}$ quite 
clearly, and explains how they generalize Taylor series approximations.  
We let $R^{(k)}$ denote the $k$th derivative of $R$.

\begin{prop} 
The map $R_{dn-1}$ is a rational map of degree $\leq dn-1$ 
satisfying $R_{dn-1}^{-1}(\infty)\subseteq P$ and 
\begin{equation}\label{reltn1} 
\textrm{ if } r(a)=0 \textrm{ then } R_{dn-1}^{(k)}(a)=f^{(k)}(a)
\textrm{ for each } 0\leq k \leq n-1. 
\end{equation}
Moreover, if $R$ is another rational map of degree $\leq dn-1$ with 
the same poles of the same order as $R_{dn-1}$ so that $R$ 
also satisfies (\ref{reltn1}), then $R\equiv R_{dn-1}$. 
\end{prop}

\begin{proof} 
The fact that $R_{dn-1}$ satisfies (\ref{reltn1}) follows from
(\ref{Rdn-1defn}) and the quotient rule (it shows the 
first $n-1$ derivatives of the integral in (\ref{Rdn-1defn}) are 
all zero). The other conclusions 
about $R_{dn-1}$ were proven in Lemma \ref{isarationalmap}. 
If $R(z)$ is another rational map satisfying 
(\ref{reltn1}), then the difference $R_{dn-1}(z)-R(z)$ 
is a rational function of  degree $\leq dn-1$ that has  $dn$ zeros
counted with multiplicity 
(since (\ref{reltn1}) has $dn$ equations).
Hence $R_{dn-1}(z)-R(z)$ must be identically $0$. \end{proof}

\section{Geometric decay in Theorem \ref{power_series_cor} is sharp}
\label{sharp_approx_sec}

For finite sets $K$, we can interpolate any
function on $K$ exactly by a finite degree rational function, 
but for larger sets the geometric decay rate  in 
Theorem \ref{power_series_cor} is sharp. 

\begin{thm}\label{rate_of_conv} 
Suppose $K \subset \complex$ is a compact set with positive 
logarithmic capacity,  and that  $P$ and  $f$  
are as in Theorem \ref{power_series_cor}.
Assume $f$ does not admit a holomorphic extension to 
	$\Chat\setminus P$. Then there exists a constant $D \in (1, \infty)$
with the following property. For any ${C}\in(0,\infty)$ 
and any sequence of rational maps ${R}_n$ of degree 
$\leq n$ satisfying ${R}_n^{-1}(\infty)\subseteq P$,
we have that
\begin{equation*} 
\sup_{z\in K}|f(z)-{R}_n(z)|>\frac{{C}}{D^n}
\end{equation*}
for all sufficiently large $n$.
\end{thm}

Consider the function $u_{\partial K, P}$, defined 
as  follows. In any face $F$ of $K$ (i.e., a  complementary 
component of $K$)
that contains a point in $P$   we set $u_{\partial K, P}$ 
to be Green's function for $F$  with pole at that point in $P$, 
and we  set $u_{\partial K, P}=0$ otherwise.  
Set $C_S:=\{z : u_{\partial K, P}(z)=\log S\}$ for $S>1$.
As in the previous section, the \emph{interior} of $C_S$,
denoted $\textrm{int}(C_S)$,  refers to the  union of the 
components of $\Chat\setminus C_S$ which do not contain a point of $P$.
Although $u_{\partial K,P}$ need not be continuous at 
irregular points of $K$, it only blows up at points of $P$, 
and so for $S$ sufficiently large, $C_S$ consists of finitely 
many closed loops, all disjoint from $K$.
After applying a M\"obius transformation,
if need be, we may assume 
that $\infty\not\in K$ and $\infty\in P$.


\begin{lem}\label{opt_conv_lemma}
	Suppose $Q$ is a rational map satisfying
$Q^{-1}(\infty)\subseteq P$
and $\sup_{z\in K}|Q(z)| \leq M$ for some constant $M<\infty$.
Let $p\in P$, let $d$ denote the local degree of $Q$ at the
pole $p$, and let $F$ denote the face of $\partial K$ 
containing $p$. Then for any $S>1$ we have 
$|Q(z)| \leq M\cdot S^d$ for $z\in\emph{int}(C_S)\cap F$.
\end{lem}

\begin{proof} Set $u:=u_{\partial K, P}$. We have that
	$ (1/d)  {\log|Q|}-u  $  
is subharmonic on the face $F$ (including at the point $p$), 
and so by the maximum principle, 
	$$\max_{z \in F}  \left( \frac {\log |Q(z)|}{d} - u(z) \right)  \leq 
	\limsup_{z \in F, z \to K}  \left(\frac {\log |Q(z)|}{d} - u(z)  \right).$$
Recalling that $u \geq 0 $ and that $|Q(z)|\leq M$ for $z\in K$ by assumption,
we conclude that
\begin{equation*} \frac{\log|Q(z)|}{d}-u(z) 
	\leq \frac{\log(M)}{d} \textrm{ for } z \in F.
\end{equation*}
In particular, we have that
\begin{equation*} 
	\frac{\log|Q(z)|}{d}-\log(S) \leq \frac{\log(M)}{d}
\end{equation*}
for $ z \in C_S \cap F$,
and hence by the maximum principle, for $S >  1$  we have that
\begin{equation}\label{int_to} 
	\frac{\log|Q(z)|}{d}
	\leq \log(S) + \frac{\log(M)}{d} 
\end{equation}
	 for $ z\in \textrm{int}(C_S) \cap F$.
Inequality (\ref{int_to})  implies the lemma.
\end{proof}

\begin{proof}[Proof of Theorem \ref{rate_of_conv}]
Since we have assumed that $f$ does not admit a holomorphic 
extension to $\Chat\setminus P$, there exists some $S<\infty$
so that $f$ does not admit a holomorphic extension to the interior
of $C_S$. We set $D:=S$. Now suppose there were 
a ${C}\in(0,\infty)$ and a sequence of 
rational maps ${R}_n$ of degree $\leq n$ 
with poles only in $P$, so that
\begin{equation}\label{assume_cont}
\sup_{z\in K}|f(z)-{R}_{n}(z)|
	\leq\frac{{C}}{D^{n}}
\end{equation}
for some subsequence in $n$. To simplify notation,
assume  that (\ref{assume_cont}) 
holds for all $n$, and not just along a subsequence 
$\{n_k\}_1^\infty$ (in general, just replace
$n$ by $n_k$ in what follows).

We will show that the sequence ${R}_n(z)$ 
converges uniformly on compact subsets of the interior of 
$C_{D}$; this will be a contradiction since such
a limit would necessarily constitute a holomorphic 
extension of $f$ to the interior of $C_{D}$.
Note that
\begin{equation}\nonumber
	|{R}_{n+1}(z)-{R}_n(z)|
	\leq |{R}_{n+1}(z)-f(z)|
	+ |f(z)-{R}_{n}(z)| \
	\leq \frac{{C}}{D^{n+1}}+
	 \frac{{C}}{D^{n}}
\end{equation}
for all $z \in K$ and $n\in\mathbb{N}$.
Let $F$ be a face of $\partial K$ containing a point $p\in P$.
The local degree of ${R}_{n+1}-{R}_{n}$ at $p$ 
is $\leq n+1$. 
Thus Lemma \ref{opt_conv_lemma} implies that for 
any $D_1<D$, we have
\begin{equation}\nonumber 
|{R}_{n+1}(z)-{R}_n(z)|
\leq D_1^{n+1}\left( \frac{{C}}{D^n}+
\frac{{C}}{D^{n+1}} \right) 
	= \left(\frac{D_1}{D}\right)^n{C}
D_1\left(1+\frac{1}{D}\right) 
\end{equation} 
 for $z\in\textrm{int}(C_{D_1})\cap F$.
As $F$ was arbitrary, it follows that
\begin{equation}\nonumber |
{R}_{n+1}(z)-{R}_n(z)|
\leq  \left(\frac{D_1}{D}\right)^n{C}
D_1\left(1+\frac{1}{D }\right) 
\end{equation} 
for $ z \in\textrm{int}(C_{D_1})$.
The right-hand side is summable in $n$, 
and so the sequence $({R}_{n})_{n=1}^\infty$ 
is uniformly Cauchy and hence uniformly convergent 
on $\textrm{int}(C_{D_1})$, as desired.
\end{proof}

\section{Proof of Theorem \ref{julia_set_cor}: approximation by 
Julia sets}\label{first_cor_proof} 

Our proof of Theorem \ref{julia_set_cor} follows the  second of
the two proofs  that Lindsey and Younsi give for their
Theorem 1.2 in  \cite{MR3955554}, 
 except that we replace their application of Hilbert's
lemniscate theorem with an application of our Theorem \ref{main_thm_2}.
We provide the details here for the convenience of the reader. 
As with Theorem \ref{power_series_cor}, we only need to 
use  Theorem \ref{main_thm_2} in the easier case of lemniscate 
graphs without vertices.

\begin{proof} [Proof of Theorem \ref{julia_set_cor}]
For $E\subset\Chat$, denote the $\varepsilon$-neighborhood of $E$ by 
\[N_{\varepsilon}(E):=\{z \in \Chat : d(z, E)<\varepsilon\}.\]
As usual, $d$ denotes the spherical metric. 
Note that for any $\varepsilon>0$,  the open set $N_{\varepsilon}(E)$
has finitely many connected  components, even if $E$ has infinitely many.
This holds since each such component contains a distinct point of any 
$(\varepsilon/2)$-dense set in $\Chat$ and such a set can be finite. 

After applying a M\"obius transformation, we may assume without 
loss of generality that $p=\infty\in A_1$ and $0\in A_2$. By Theorem 
\ref{main_thm_2}, there exists a lemniscate $L_r\subset N_{\varepsilon}(A_2)$
so that $L_r$ consists of a union of pairwise disjoint Jordan curves, and
$L_r$ separates $\partial N_{\varepsilon}(A_2)$ from $A_2$. 
Since each component is a Jordan curve, we know that $r$ has no critical 
points on $L_r$. 

We may color white exactly those faces of $L_r$ which contain at least one
component of $A_2$, and arrange for $r(\infty)=\infty$. Furthermore, if 
$A_1$ has finitely many components, and $P$ contains $\infty$ as well as
one point from each bounded component of $A_1$, we may  arrange 
(by taking $\varepsilon$ smaller) for $r$ to satisfy  $r^{-1}(\infty)=P$. 

Consider the map $r^n$, where $r^n$ denotes the 
product $r$ with itself $n$ times, 
not the $n$th iteration of $r$. 
Since $r$ has a pole at $\infty$,  we know $r^n$ has a super-attracting fixed 
point at $\infty$. Since  $0\in A_2$, we must have $|r|< 1$
on some disk $D(0,\delta)$, 
and so $|r|^n <\delta/2$ on $D(0, \delta)$ 
if  $n$ is large enough. This implies  $r_n$ maps $D(0, \delta)$ 
into $D(0,\delta/2)$ and hence $r^n$ has 
an attracting fixed point in $ D(0, \delta) \subset A_2$. 

Denote the corresponding basins of attraction
by $\mathcal{A}_n(\infty)$ and  $\mathcal{A}_n(0)$, where we emphasize the 
dependence on $n$. Since $A_2$ is compactly contained in $r^{-1}(\mathbb{D})$
(the union of the white faces), we have that $A_2\subset \mathcal{A}_n(0)$
for large $n$. Similarly, any compact subset of
$r^{-1}(\Chat\setminus\overline{\disk})$
is contained in $\mathcal{A}_n(\infty)$ for large enough $n$. 
As a consequence we have $\mathcal{A}_n(0)\subset N_{\varepsilon}(A_2)$ 
for large $n$. This proves 
\[ d_H(A_2, \mathcal{A}_n(0))<\varepsilon, \] 
and the other two inequalities in Theorem \ref{julia_set_cor}
follow by similar
considerations. Thus, $r^n$ and $\mathcal{A}_1:=\mathcal{A}_n(\infty)$ and
$\mathcal{A}_2:=\mathcal{A}_n(0)$ satisfy the conclusions of the theorem
for all large enough $n$. 

We note the fact that $\mathcal{F}(r^n)=
\mathcal{A}_1\sqcup\mathcal{A}_2$ for large $n$ follows since all critical
values of $r^n$ lie in $\mathcal{A}_1\sqcup\mathcal{A}_2$ for large enough $n$
(since $r$ has no critical points on $L_r := \{|r|=1\}$, all the critical 
values of $r^n$ have very large or very small absolute value if $n$ is large).
This observation uses some standard but deep facts:  by Sullivan's 
``no wandering domains'' theorem 
(e.g., \cite{MR819553}, Theorem IV.1.3 of 
\cite{MR1230383} or Theorem F.1 of \cite{MilnorCDBook}), 
every Fatou component 
is eventually periodic, and it is known that 
every periodic Fatou component has 
an associated critical orbit  in the domain or that  accumulates
on the  boundary of the component 
(e.g. Theorems 8.6 and 11.17, Corollary 10.11 and Lemma 15.7 of 
\cite{MilnorCDBook}). In our case, every critical orbit is attracted to 
fixed points in ${\mathcal A}_1$ or ${\mathcal A}_2$, so there 
can be no other Fatou components except those that eventually land
on one of these two. 
\end{proof}


\section{A Topological Lemma}\label{topological_section_app}

Recall that the proof of our main result (Theorem \ref{main_thm_2}) 
relied on Theorem \ref{whatweneedfromappendix}, but that
Theorem \ref{whatweneedfromappendix} has not yet been proven 
in general (only in the case when there are no vertices).
Section \ref{harmonic_app_appendix}
is devoted to establishing this result. In this section, we  give 
a topological result that will be used in 
Section \ref{harmonic_app_appendix}.

The $n$-dimensional simplex $\Delta_n$  is the convex hull
of the standard unit vectors $\{ e_k\}_1^{n+1} \subset  \reals^{n+1}$;
$e_k$ is $1$ in the $k$th coordinate and zero elsewhere. 
These  points are called the vertices of the simplex.
A face of $\Delta_n$ is a convex hull of some
non-empty subset of its vertices. A facet is a face 
of dimension $n-1$ (the convex combination of all but 
one vertex).  Every face is the intersection of facets
that contain it, so a continuous  map $f: \Delta_n \to \Delta_n$ 
that maps each facet into itself must also map each face 
into itself. Every point $x \in \Delta_n$ is in the interior 
of some face, where interior means that $x$ does not lie in 
any strictly lower dimensional face (the vertices are interior points 
of themselves).

The following is Lemma 2.1 of 
\cite{MR435791}, but was probably known much earlier.

\begin{thm}\label{simplex map} 
Suppose $f: \Delta_n \to \Delta_n$ is a continuous map 
that sends each facet  of $\Delta_n$ into itself. 
Then $f$ is surjective.
\end{thm}

\begin{proof}
For completeness, we recall the proof from 
\cite{MR435791}. Since $n$ is fixed, we write 
$\Delta := \Delta_n$, to simplify notation.
Since $f(\Delta)$ is compact and the 
interior of $\Delta$ is dense in $\Delta$, it 
suffices to show that every interior point is in  $f(\Delta)$.
By way of contradiction, suppose $p$ is an interior point that 
is not in the image, and let $g$  be the 
radial projection of $\Delta\setminus \{p\}$ onto
$\partial \Delta$. This  map is continuous and 
fixes each point of the boundary. Let $h:
\Delta \to \Delta$ be the linear map 
extending the cyclic permutation 
$$ (x_1, x_2, \dots, x_n ) \to 
(x_2, \dots, x_n, x_1 ) $$
of the vertices. 
Then  $\phi :=h \circ g \circ f: \Delta  \to \partial \Delta$ 
is clearly continuous.  

We will show  $\phi$  has no fixed point, contradicting 
Brouwer's theorem  \cite{MR1511644}
(that a continuous self-map 
of a compact, convex set in $\reals^n$ must have a 
fixed point). This proves  
that $ p \in f(\Delta)$,  and  hence that  $f$ is surjective. 
Since $\phi(\Delta) \subset \partial \Delta$, any fixed 
point $x$ must be in $\partial \Delta$, and hence in the interior of
some face $F_1$.  By assumption $f(x)$ is also in $F_1$, 
and so $g(f(x)) = f(x) \in F_1$, since $g$ is 
the identity on $\partial Q$.
But $h$ maps each face of $\Delta$ to a distinct face  
(every subset of vertices is sent to a 
different subset by the cyclic permutation).
Hence  $h$ maps the interior of
$F_1$ to the interior of some 
different face $F_2$,  and hence 
$\phi(x) = h(f(x)) \ne x $. 
Thus $\phi $ has no fixed points. 
\end{proof} 

By projecting $e_{n+1}$ to zero, the simplex $\Delta_n \subset \reals^{n+1}$  
can identified with $X_n \subset \reals^n$, 
the  convex hull of $\{0\}$ and the 
$n$ unit vectors in $\reals^n$. 
Note that $\Delta_{n-1} \subset X_n$ is the unique facet of 
$X_n$ that does not contain the origin.
Let us denote projection  from $\reals^n$ onto the
$j$th coordinate by $\pi_j(x_1, ..., x_n):=x_j$ for $1 \leq j \leq n$,
and denote the vector $(x_1..., x_n)\in\mathbb{R}^n$ by $\myvec{x}$.

\begin{thm}\label{surjective}
Let  $F: X_n \rightarrow [0,\infty)^n$ be a continuous mapping satisfying
\begin{equation}\label{surj 1} F_j(\myvec{x})=0 
\textrm{ iff } x_j=0. \end{equation}
Then  $ \varepsilon X_n \subset F(X_n) $ for some $\varepsilon>0$.
\end{thm}

\begin{proof} 
The condition  in the theorem 
says that each facet of $X_n$, except for $\Delta_{n-1}$,
is mapped into the hyperplane containing that facet.
Because our assumptions imply  $F^{-1}(0)=0$, 
$F(\Delta_{n-1})$ is  a compact set  that does  not contain 
zero, and so it has a positive distance $\varepsilon$ 
from the origin. Let $R$ denote the radial retraction 
of $[0,\infty)^n$ onto $\epsilon  \cdot X_n$, i.e., 
$$ R(\myvec{x}) :=  \varepsilon \cdot 
 \frac {(x_1, \dots, x_n)}{ x_1 + \dots+ x_n}.$$
if  $x_1+ \dots + x_n \geq \varepsilon$, and $R(\myvec{x}) 
= \myvec{x}$   if $x_1+ \dots + x_n  \leq \varepsilon$. 
Then $  \myvec{x} \to (R\circ F(\myvec{x}))/\varepsilon$ 
is a continuous map of the simplex $X_n$ into itself, and 
every facet of $X_n$ maps into itself.
Hence $(R \circ F)/\varepsilon$ is surjective 
by Theorem \ref{simplex map}, or equivalently, 
$\varepsilon X_n   =  R(F (X_n)) $.
Since $R$ maps
$[0,\infty)^n \setminus \varepsilon  X_n$ onto 
$\varepsilon \Delta_{n-1} \subset \partial (\varepsilon X_n)$, 
an interior point of $\varepsilon X_n$ can't be in
$R(F(X_n))$ unless it is in $F(X_n)$.
Thus $F(X_n)$ contains the interior of 
$\varepsilon X_n$. Since  $F(X_n)$  is compact, it 
must contain all of $\varepsilon X_n$, as desired. 
\end{proof} 

\begin{cor}\label{main_top_thm}
Let  $F: [0, 1]^n \rightarrow [0,\infty)^n$ 
	be a continuous mapping satisfying
\begin{equation}\label{iff_cond} 
F_j(\myvec{x})=0 \textrm{ iff } x_j=0. \end{equation}
Then there exists $\delta>0$ so that
$\{\myvec{y}\in\mathbb{R}^n :
y_1 = ... = y_n \textrm{ and }
0\leq y_1<\delta\}\subset F([0,1]^n)$.
\end{cor}

\begin{proof} 
Restrict $F$ to $X_n \subset [0,1]^n$ and 
apply Theorem \ref{surjective}. 
If $\delta < \varepsilon/n$, then the diagonal segment 
is inside $\varepsilon X_n \subset F(X_n) \subset F([0,1]^n)$,
and so the corollary follows. 
\end{proof} 

\section{Proof of Theorem \ref{whatweneedfromappendix}: harmonic level sets}
\label{harmonic_app_appendix}

In this section, we prove Theorem \ref{whatweneedfromappendix}.
Recall that if $H$ is a  colored lemniscate graph with 
no vertices, and if $P$ is a finite set of points 
from the grey faces, then the function $u_{H, P}$ was defined 
in Notation \ref{u_without_verts}  as 
the  sum of Green's functions for faces of $H$ 
with the poles in the set $P$, i.e.,  
$$ u_{H,P}(z) := \sum_{p \in P} G_{B(p)} (z, p) ,$$
where  $B(p)$ denotes the face of $H$ containing $p$
(we define the Green's function of a domain to be zero off that 
domain).
In this section, we will say that two  positive functions $f,g$ are 
relatively close if $|f-g|/g $ is close to zero. In other words, 
$f/g$ is uniformly close to $1$. 

For the convenience of the reader, we re-state the desired result.

\begin{thm}\label{harmonic_app} 
Let $G$ be a 2-colored lemniscate graph,
$\varepsilon>0$, and $P$ a set of points
consisting of one point  in each grey face of $G$.
Then, for all $\delta>0$ sufficiently small,
there exists a lemniscate graph $H$ without vertices so that each grey 
face of $G$ is contained in a grey face of $H$, and
so that $u_{H, P}^{-1}(\delta)$ and $G$ are $\varepsilon$-homeomorphic. 
\end{thm}

The proof breaks into several steps: 
\begin{enumerate}
\item Modify $G$ near  
each vertex $v$ to consist of  straight segments meeting at $v$ 
and making equals angles there.
\item We further modify $G$  so that 
the Green's function for each grey face of $G$ 
takes  the same value at every point 
of a certain  finite set ($d$ points are 
chosen  a fixed distance from every vertex of degree $2d$).
\item Define  a vertex-free lemniscate graph $H$ 
by adding disks around each vertex $v$ of $G$.
		See Figure \ref{Pinching2}.
\item Define a harmonic function $w_v$  on the faces of $H$
near each vertex $v$, so  that $w_v$ has a single 
critical point of order $d-1$ at $v$, where $2d$ is the 
degree of  $v$ as a vertex of $G$.
\item Use a partition of unity to combine $w_v$ and $u_{G,P}$ 
in an annulus around each vertex $v$.  We make this
new function harmonic using  the measurable Riemann mapping
theorem, but this  correction causes the poles to move slightly.
\item Use a fixed point argument to show that the poles 
can be placed precisely on $P$.
This gives the desired  harmonic  function with 
 poles in $P$, and with a critical level set 
that is $\varepsilon$-homeomorphic to $G$.
\end{enumerate} 

\vskip.2in
\noindent
{\bf Step 1: modifying $G$  to be straight near its vertices.}
By Theorem \ref{eps homeo graphs} we can assume that $G$ has 
analytic edges that form equal angles  at each vertex. 
It is easy to modify such a graph  $G$  so that its edges are smooth arcs
and are  straight line segments near each vertex.
In other words, we may assume there is a 
$\rho_0>0$ so that  
for every vertex $v$ of $G$,  $G \cap D(v, 2\rho_0)$ consists
of line segments making equal angles at $v$. 
Thus $\rho_0$ represents a  scale below which $G$ looks 
like line segments around each vertex. Note that $\rho_0$ 
is the same value at every vertex.
We also assume that $\rho_0$ is so small  that no 
point of $P$ lies inside any of the disks $D(v, 2\rho_0)$.

A sector $S$ with vertex $v\in \complex$, 
radius $r \in (0, \infty]$ and angle $\theta \in (0, 2 \pi]$
is a set of the form 
$$ S := v+ \{z: 0< |z| <r, |\arg z - \theta_0| \leq  \theta/2\}$$
for some $\theta_0\in [0, 2 \pi)$ and a truncated sector 
is a set of the form 
$$  v+ \{z: s< |z| <r, |\arg z - \theta_0| \leq  \theta/2\}$$
for some $0< s< r< \infty$.

By our choice of $G$ and $\rho_0$, for each 
grey face $B$ of $G$ and vertex $v \in \partial B$, 
 $B \cap D(v, \rho_0)$ 
is a union of sectors of radius $\rho_0$. If $v$ has 
degree $2d$ in $G$, then each sector at $v$  has  angle $\pi/d$.
Each such  sector can be conformally  mapped to the half-disk
$W := \disk \cap  \uhp  =\{ z: |z|<1,\im(z) >0\}$  
by a power map $  \tau(z) := a (z-v)^d$, for some
$a \in \complex \setminus\{ 0\} $.
	Then $ U(z)  := u_{G,P} \circ  \tau^{-1}(z) $ is harmonic 
on the half-disk $W$ and it vanishes on  $I :=[-1,1]$, 
so   $U$  extends 
harmonically across $I$ by the Schwarz reflection principle. Thus 
$$ U(z) = U(x+iy)= by + O(y^2) + O(xy) 
	= b \cdot \re(z) +O(|z|\cdot \re(z)),$$
as $z$ approaches zero, where $b>0$ is the normal derivative of $U$ 
at zero. The normal derivative is positive since $U$
positive on $W$, for if  
$\frac{ \partial U }{\partial n} (0) \leq 0$, 
then $U$ would take negative values somewhere in $W$.
This implies that $u_{G,P}$ is asymptotically equal to 
$c \cdot \re((\lambda (z-v) )^d)$ as $z \to v$ through the sector, 
for some $|\lambda|=1$ and $c>0$.
More precisely, we have 
\begin{eqnarray} \label{sector limit 1}
c_S := \lim_{z \to v, z \in S} \frac {u_{G,P}(z)}{\re(\lambda(z-v)^d)} 
	\in (0, \infty).
\end{eqnarray}

\vskip.2in 
\noindent
{\bf Step 2: modifying $G$ to equalize the Green's function.}
In this step we use the topological fact recorded in Corollary 
\ref{main_top_thm}.
For the remainder of this section 
we use a plain $G$ to denote a lemniscate graph and a 
$G$ with a subscript, e.g., $G_W$, to denote Green's function 
for a domain $W$. 

We will choose $\eta,\rho>0$ 
so that $\eta \ll \rho \ll \rho_0$.  Given a grey 
face $B$ of $G$,  let $v$ be  a vertex on the 
boundary of $B$.  Let $2d$ be the degree of $v$. 
Then $\partial D(v, \eta) \cap B$ consists
of $d$ arcs  each of angle measure $\pi/d$. We call the center
of each such arc a ``base point'' $b$.  
See Figure \ref{Pinching}.
We claim that we can modify $G$ to obtain a new graph $G'$ 
so that $u_{G',P}(b)$ 
has the same  positive value at every base point $b$. 

Figure \ref{Pinching}
shows the case where the vertex $v$ has degree eight, so $d=4$ and 
$ B \cap D(v, \rho_0)$ consists of four sectors. 
Let $S$ be one of these sectors. We modify its boundary 
	inside the annulus $\{\rho/2 \leq |z| \leq  \rho\}$.
Let $\gamma$ be one of the radial sides of the  sector $S$.
We modify $\gamma$ as illustrated in Figure \ref{Pinching}.
The part of $\gamma$  between the circles $\{|z|=\rho\}$ and 
$\{|z|= \rho/2\}$ is replaced by a new arc consisting of three 
parts: one subarc on each of these two circles, and a radial segment 
joining them. The circular arcs lie inside $S$ and 
both have angle measure $\theta \in (0, \pi/2d$). 
The  second radial  segment in $ \partial S$  is modified
symmetrically. Doing this for every sector associated to every 
vertex of $G$ gives a new lemniscate graph $G'$ that 
has the same vertices as $G$. 

\begin{figure}[htb]
\centerline{
\includegraphics[height=2.25in]{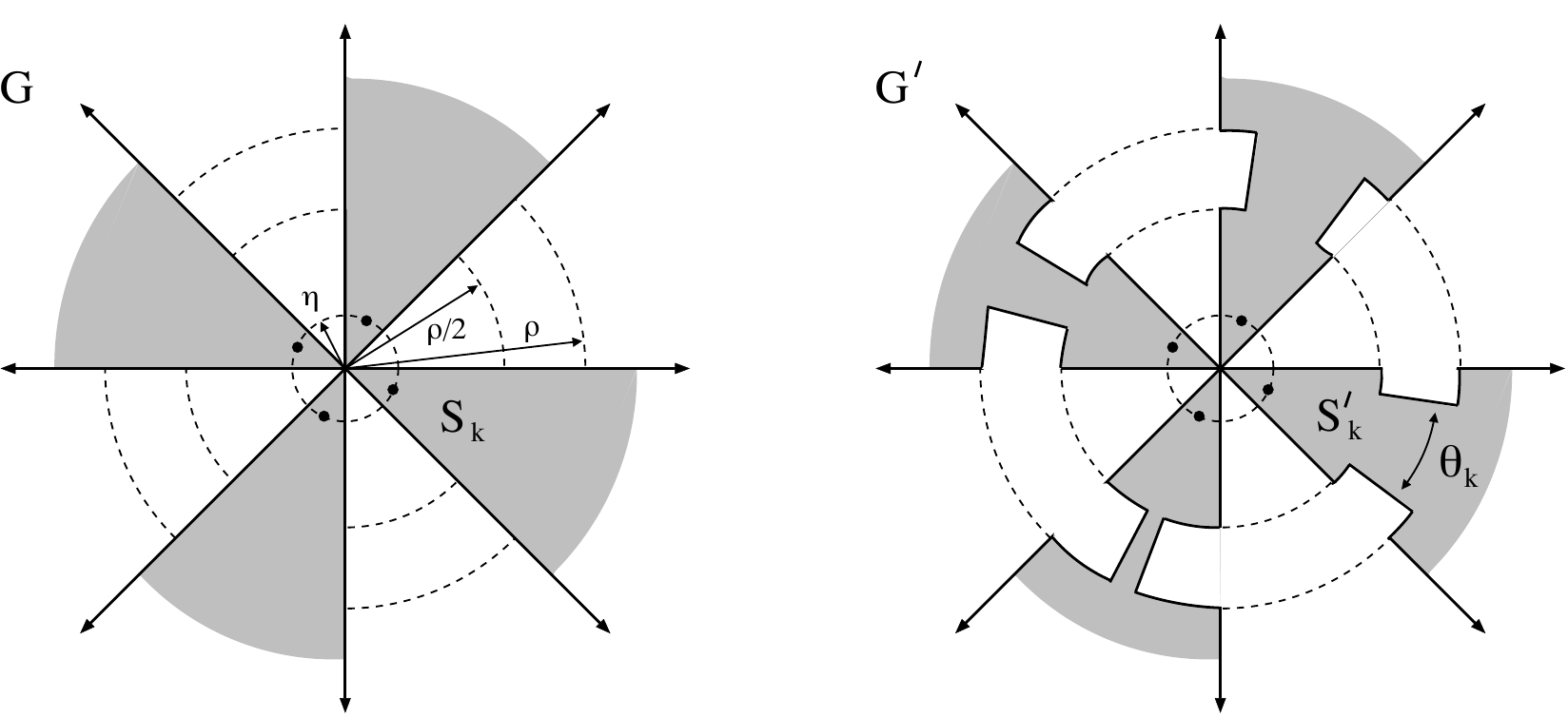}
}
\caption{
By ``pinching'' the sectors that lead to a vertex $v$, we 
can decrease the harmonic measure near $v$. 
Hence we can make Green's function 
as small as we wish at the base points (black dots), and  
make the values at all the base points as close to 
each other as we wish. 
}
\label{Pinching}
\end{figure}

Let $\{S_k\}_1^m$ be a listing of all the sectors 
of radius $\rho_0$ associated to $G$,
and let $S_k' \subset S_k$ be
the subregion obtained after the modification. 
This list is over  all sectors of all grey faces of $G$; 
since half the sectors at any vertex $v$ are grey, 
the number of sectors is $\frac 12\sum_v \deg(v)$. 
Let $\theta_k$ be the angular width of the channel 
in $S_k'$.  We let $v_k$ denote the vertex of $S_k$; 
observe that we can have $v_k =v_j$ even if $j \ne k$ 
(different sectors can share a vertex).
Also, $S_k' \cap D(v, \rho/4)$ is still a sector 
and it contains exactly one base point, which we denote  $b_k$.

As we continuously change the value of $\theta_k$, 
the value of $u_{G',P}(b_k)$ changes continuously, and 
it tends to zero as $\theta_k$ tends to zero 
(this is intuitively clear, and it is an
immediate consequence of  the Ahlfors distortion 
theorem, e.g., Theorem IV.6.2 of  \cite{MR2450237}).
By Corollary \ref{main_top_thm} there is a choice of 
$\myvec{\theta}=(\theta_1, \dots,\theta_m)$ so that 
$u_{G',P}$ takes the same value at every base point.



\vskip.2in
\noindent
{\bf Step 3: constructing the vertex-free approximation.}
To simplify notation we refer to $G'$ as just $G$;
we will have no further  need to refer to the previous 
versions of the graph.

We want to modify $G$ inside even smaller neighborhoods 
of each vertex $v$  to obtain a vertex free graph $H$. 
For each $v$ we will choose a value  $\delta_v>0$,
remove  the segments  $G \cap D(v, \delta_v)$, 
and add the $d$  arcs of  $\partial D(v, \delta_v)$
that are outside the grey faces of $G$. 
This construction is illustrated in 
Figure \ref{Pinching2}.  The new lemniscate graph $H$ has 
no vertices (we just removed all of them) and $H$ equals $G$ outside 
the $\delta_v$-disks around each vertex of $G$. 
We shall see later that $\delta_v$ is the same for 
vertices with the same degree, although this is not 
crucial to the argument. More important is that 
we will be able to choose every $\delta_v$ to be 
as small as we wish, say all 
smaller than some $\delta_0 \ll \eta$. 

\begin{figure}[htb]
\centerline{
\includegraphics[height=2.3in]{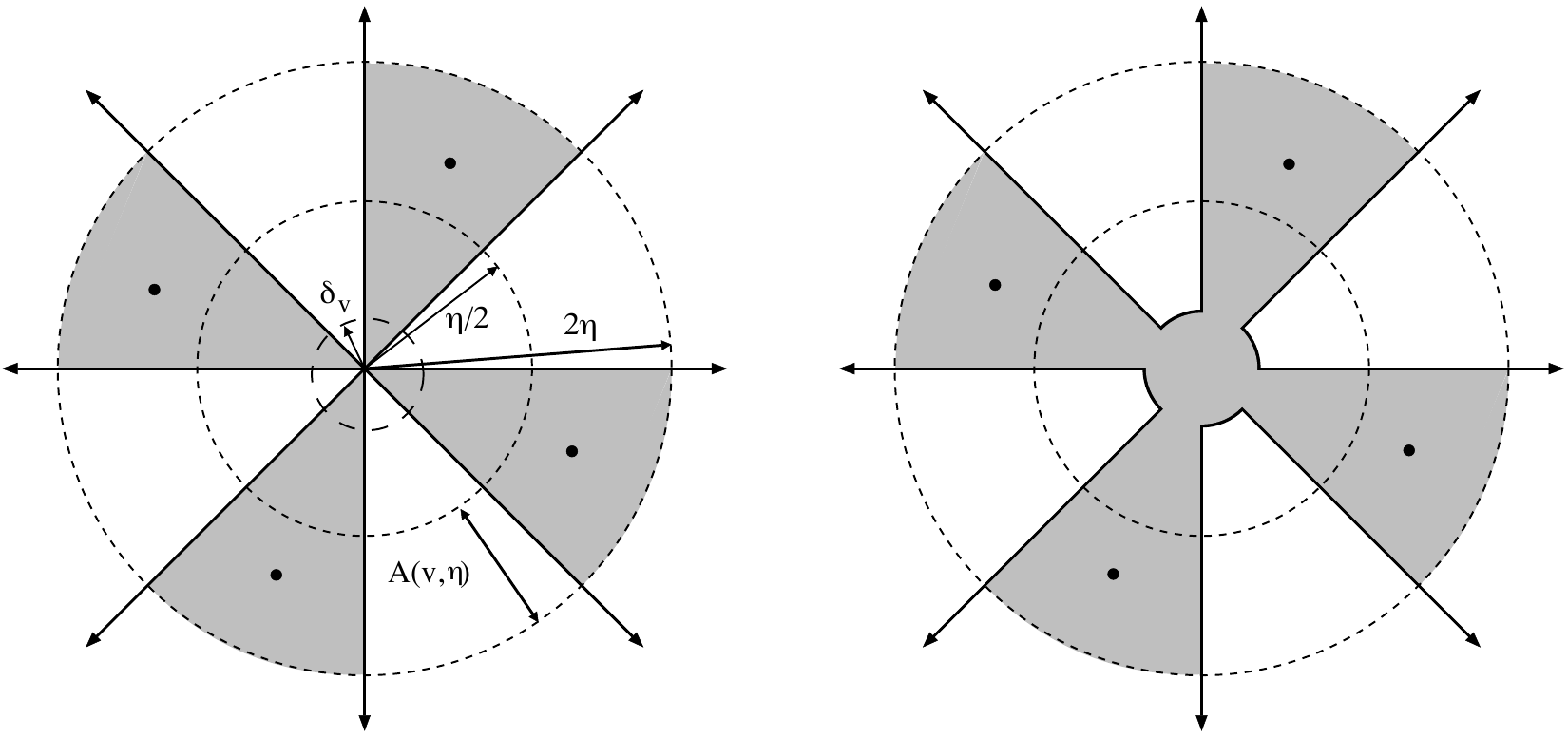}
}
\caption{
We modify $G$ near each vertex $v$ by removing radial 
segments and adding circular arcs, so that $d$ different 
sectors are now joined by a disk centered at $v$.
We let $A(v, \eta):=\{z: \eta/2 < |z-v|< 2 \eta\}$.
}
\label{Pinching2}
\end{figure} 

\begin{rem}
Figure \ref{Pinching2} also shows that each grey face 
of $H$ is  the union of some  grey faces of $G$, together 
with disks centered at all the vertices on the boundaries 
of these faces. 
If $G$ is connected, then so is the closure of it grey faces, 
and hence $H$ has only one grey face (although $H$ itself
may be disconnected).
Similarly, if $G$ is disconnected, but has no multiply connected 
white faces, then the closure of its grey faces is still connected, 
and hence $H$  has a single grey face in this case too. 
\end{rem}

\begin{rem}
The construction in the following Steps 4-6 is not needed
if all the vertices of $G$ have degree four. In that 
case, one can show that if the $\delta_v$'s are 
all small enough, then $u_{H,P}$ must have a simple 
critical point near each of the vertices $v$ of $G$. 
Then  Corollary \ref{main_top_thm}
lets us choose the $\delta_v$'s  so that these critical values 
are all the same.  One can then prove that the level set 
of $u_{H,P}$ passing through these critical points
is $\varepsilon$-homeomorphic to $G$.

However, a more complicated argument seems 
necessary for higher degree vertices.
Although it is not too hard to show that $u_{H,P}$ has $d-1$ critical 
points (counted with multiplicities) near each vertex $v$ of  
degree $2d$, 
it is not obvious whether these points 
form a single critical point of order $d-1$; 
probably they do  not. 
Instead, we will construct a function that has most of 
the necessary properties: it is zero on $H$, has logarithmic poles 
at $P$, has critical points of the correct orders 
at the vertices of $G$, and it is ``nearly harmonic'', i.e.,  
it fails to be harmonic only in the annuli $A(v,  \eta)
:=\{z: \eta/2 <|z-v|< 2\eta\}$. Using 
the measurable Riemann mapping theorem, we will
obtain  a quasiconformal map $\psi:\complex \to \complex$ 
that is close to the identity, and so that pre-composing
our nearly harmonic function with $\psi$
gives a function  $V$  that  is harmonic 
on the faces of $H' :=\psi^{-1}(H)$ and has its poles at 
$P' := \psi^{-1}(P)$.  The final step will be to show the poles can 
be placed exactly at $P$.
\end{rem}

\vskip.2in 
\noindent
{\bf Step 4: adding high degree critical points.}
To start the construction described in the previous paragraph, 
let $\Omega_d := D(0, 1) \cup \{ \im(z^d) < 0\}$. 
This is 
an infinite domain that is symmetric under rotation by $2\pi/d$. 
(It looks like the right side of Figure  \ref{Pinching2}: a 
union of the unit disk and $d$ evenly spaced  infinite sectors
of angle $\pi/d$.)
Let $f_d$ be  the holomorphic map from $\Omega_d$
to the right half-plane that is given by  
the following  composition  of four maps
(see Figure \ref{ComposeMaps}): 
\begin{enumerate}
\item  the power map $z^d$ sending $\Omega_d$ to  $W_1:=\{z: \im(z) <0\} 
	\cup \{z: |z|< 1\}$,
\item the linear  fractional map $(1+z)/(1-z)$ 
	sending $W_1$ to $ W_2 :=\{ -\pi < \arg z < \pi/2\}$, 
\item the power map $ e^{i\pi/6} z^{2/3}$ sending  $W_2$
	to the right half-plane, 
\item  the linear fractional map  $(4/3)/(z+i)$ (this preserves the right
	half-plane).
\end{enumerate} 
The pole of the final map is chosen so that $f_d(\infty) = \infty$.
The last three maps  define 
a conformal map from a perturbed version of the lower 
half-plane to the right half-plane that fixes $\infty$, 
and it is easy to check that this map  is asymptotically linear 
near infinity, and the ``4/3'' is chosen so that this map (the composition of the last three maps in Figure \ref{ComposeMaps}) is asymptotic 
to $iz$. Thus  $f_d$ is asymptotic to $iz^d$ 
near infinity (for $ z \in \Omega_d$).

\begin{figure}[htb]
\centerline{
\includegraphics[height=1.3in]{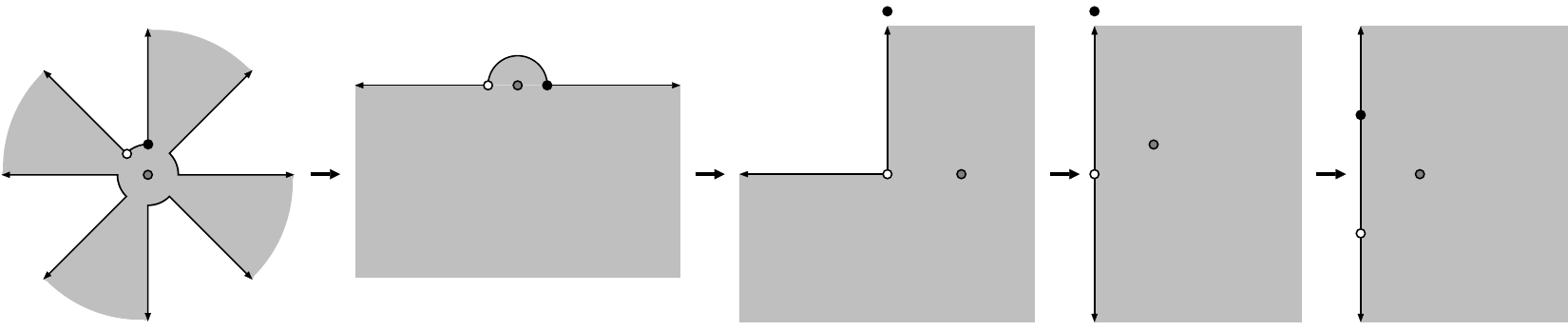}
}
\caption{
The domain $\Omega_d$ can be mapped to a half-plane by 
a composition  $f_d$ of explicit maps: the power map $z^d$, 
the linear fractional transformation $(1+z)/(1-z)$,
	the power map $e^{i\pi/6} z^{2/3}$, 
	and another linear fractional transformation  $(4/3)/(1+z)$
(chosen so that $f_d(\infty)   =\infty$).
}
\label{ComposeMaps}
\end{figure}

Set  $w_d (z) := \re(f_d(z))$.
This is a  positive harmonic
function  on $\Omega_d$ that is zero on $\partial \Omega_d$, 
and   is asymptotic to  $ \re (i{z^d})$ as $z$ tends to infinity
in $\Omega_d$.  More precisely, 
\begin{eqnarray} \label{sector limit 2}
\lim_{z \to \infty, z \in \Omega_d}
	\frac {w_d(z)}{\re(i(z-v)^d)}  =1.
\end{eqnarray} 
It is easy to check that 
 $w_d$  has a critical 
point of order $d-1$ at the origin and no other critical points. 
Figure \ref{plot_v} shows contour plots of $w_d$ for 
$d=2,3,4,5, 6, 10$.

\begin{figure}[htb]
\centerline{
\includegraphics[height=4.3in]{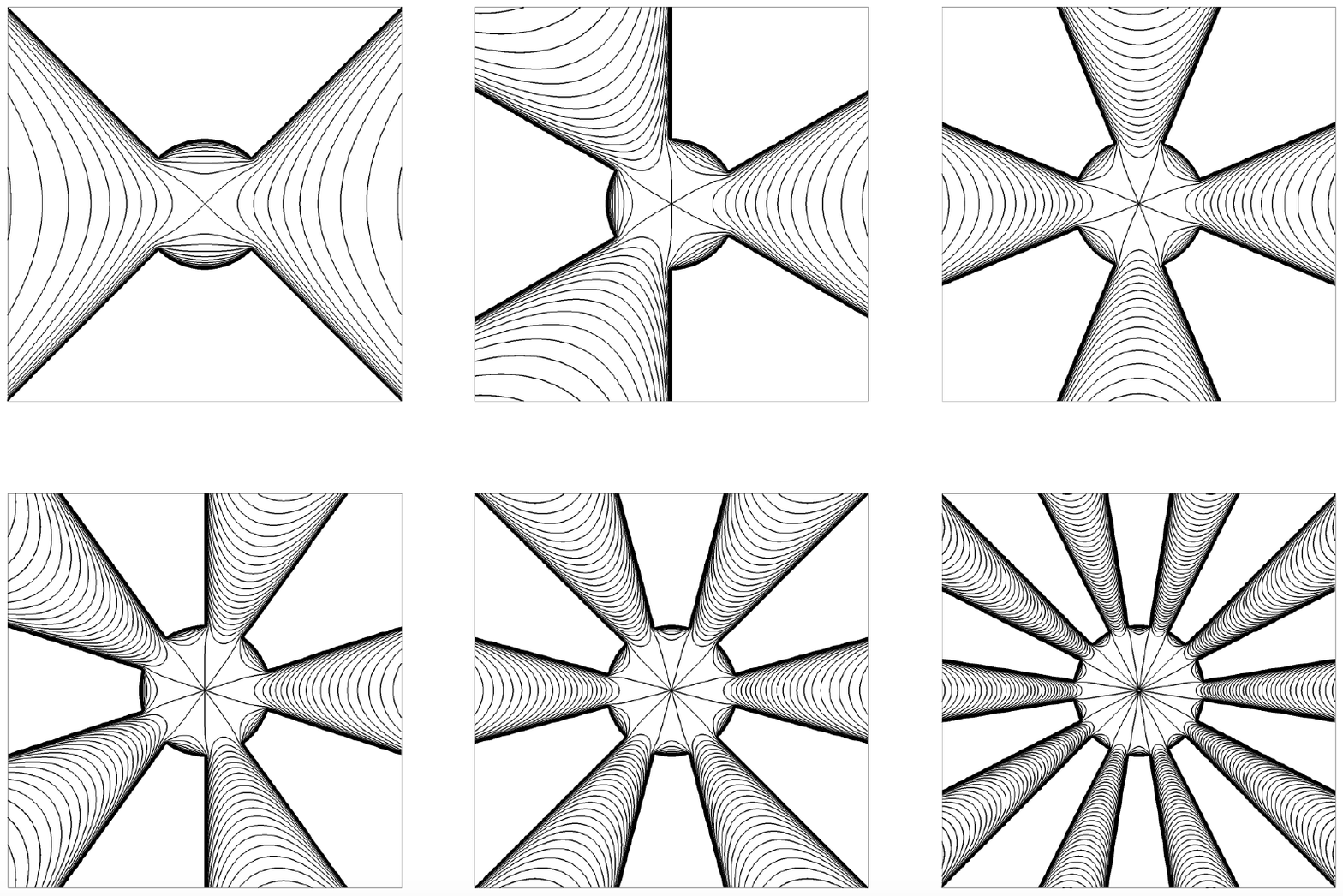}}
\vspace{-4mm}\caption{ Contour plots of the function $w_d$ when $d=2,3,4,5,6$ and $10$. 
These were computed as $ w_d=\re(f_d(z))$, where $f_d$ is 
the holomorphic map  shown in 
Figure \ref{ComposeMaps} sending $\Omega_d$ to the 
right half-plane. Since $f_d'$ only has a zero at $0$, 
the only critical point of $w_d$ is at the origin.
}
\label{plot_v}
\end{figure}


Suppose $B$ is a grey face of $H$. 
For each vertex $v$ of $G$ that is in $B$, we define a function 
by translating, rotating and rescaling $w_d$, i.e., 
on $B \cap D(v, 2 \eta)$  
we set  
\begin{eqnarray} \label {w_v eqn}
	w_v(z) := c_v \cdot w_d(\lambda_v(z-v)/\delta_v)
\end{eqnarray} 
where $|\lambda_v|=1$ is chosen to rotate the 
sectors of $B$ at $v$ to match the ``arms'' of $\Omega_d$,
and  the constant $c_v$ is chosen so that $w_v(b) = 
u_{G, P}(b) $ at each of the base points $b$ surrounding $v$
(recall $u_{G,P}$ has the same value at each of these points). 

Then $w_v$ is a harmonic function that has a critical point
of the correct degree located at the correct point. However, 
this function is  only defined near each vertex, not on 
entire faces of $H$.
On the other hand, the function $u_{G,P}$ is not harmonic 
on $B \cap D(v, \delta_v) = D(v,\delta_v)$ for each vertex $v$
(inside this disk, $u_{G,P}$ is positive and harmonic on half 
of the sectors touching $v$ and vanishes on the other sectors). 
Thus neither $w_v$ nor $u_{G,P}$  can be  $u_{H,P}$, 
but we will construct $u_{H',P}$  for some 
$H' \approx H$ by combining these two functions 
with a partition of unity,  and then  applying the measurable Riemann 
mapping theorem. 

\vskip.2in 
\noindent
{\bf Step 5: merging two harmonic functions.}
By our earlier remarks,  
(\ref{sector limit 1})  holds if $\eta \ll \rho $
and (\ref{sector limit 2}) holds if $\delta_v \ll \eta$. 
If both these conditions hold then 
$ {u_{G,P}(z)}$ and ${w_v(z) } $ are both relatively close 
to  multiples of the same function in $B \cap A(v, \eta)$,
and hence are relatively close to each other. 
In other words, $ {u_{G,P}(z)}/{w_v(z) } $
is uniformly close to a constant  on $B \cap A(v,\eta)$.
Previously, we had multiplied  $w_v$ by a constant  to make
it agree with $u_{G,P}$ at the basepoints around each vertex, 
so we must have $ {u_{G,P}(z)}/{w_v(z) } \approx 1 $.
That is, for any $\varepsilon >0$ we can  ensure
$$   1- \varepsilon \leq \frac 
{u_{G,P}(z) }{  w_v(z) }  \leq 1+ \varepsilon $$
on $B \cap A(v, \eta)$
by  choosing $\rho$, $\eta$ and $\delta_0$ all 
sufficiently small (and each sufficiently smaller than the previous one).

Note that when we decrease $\delta_v$ by a factor of $t$,
the constant $c_v$ in 
(\ref{w_v eqn}) has to decrease by a factor of approximately $t^d$, 
in  order to maintain the equality $w_v(b) = u_{G,P}(b)$. 
Thus as $\delta_v$ tends to zero, so does $c_v$, and to make
the values of $w_v(v)$ the same for every $v$,
we must take $c_v$ to have the same value  for all vertices 
with the same degree. Since each point $v$ is a critical 
point of $w_v$, this implies that all the critical values 
are the same. 

Suppose $S = S_k$ is one of the sectors we are considering, and 
suppose $f_1$ is a holomorphic function in $ W :=S \cap \{ \eta/3
< |z-v| < 3 \eta\}$ that has real  part $w_v$; this  truncated
sector is simply connected, so $w_v$ has a harmonic conjugate 
on $W$. Similarly suppose  $f_2$ is holomorphic on $W$
with real part $u_{G,P}$.
By Schwarz reflection, both these functions extend analytically 
across both radial sides of $W$
 and define a holomorphic function  
on the union $W'$ of $W$ and its reflections.  Now 
let $\Omega :=  S \cap \{ \eta/2 < |z-v| < 2 \eta\} \subset W$. 
This is compactly contained in $W'$ and 
$\dist(\partial W', \Omega) \simeq \eta$. See 
Figure \ref{CauchyEst}.

\begin{figure}[htb]
\centerline{
\includegraphics[height=2.5in]{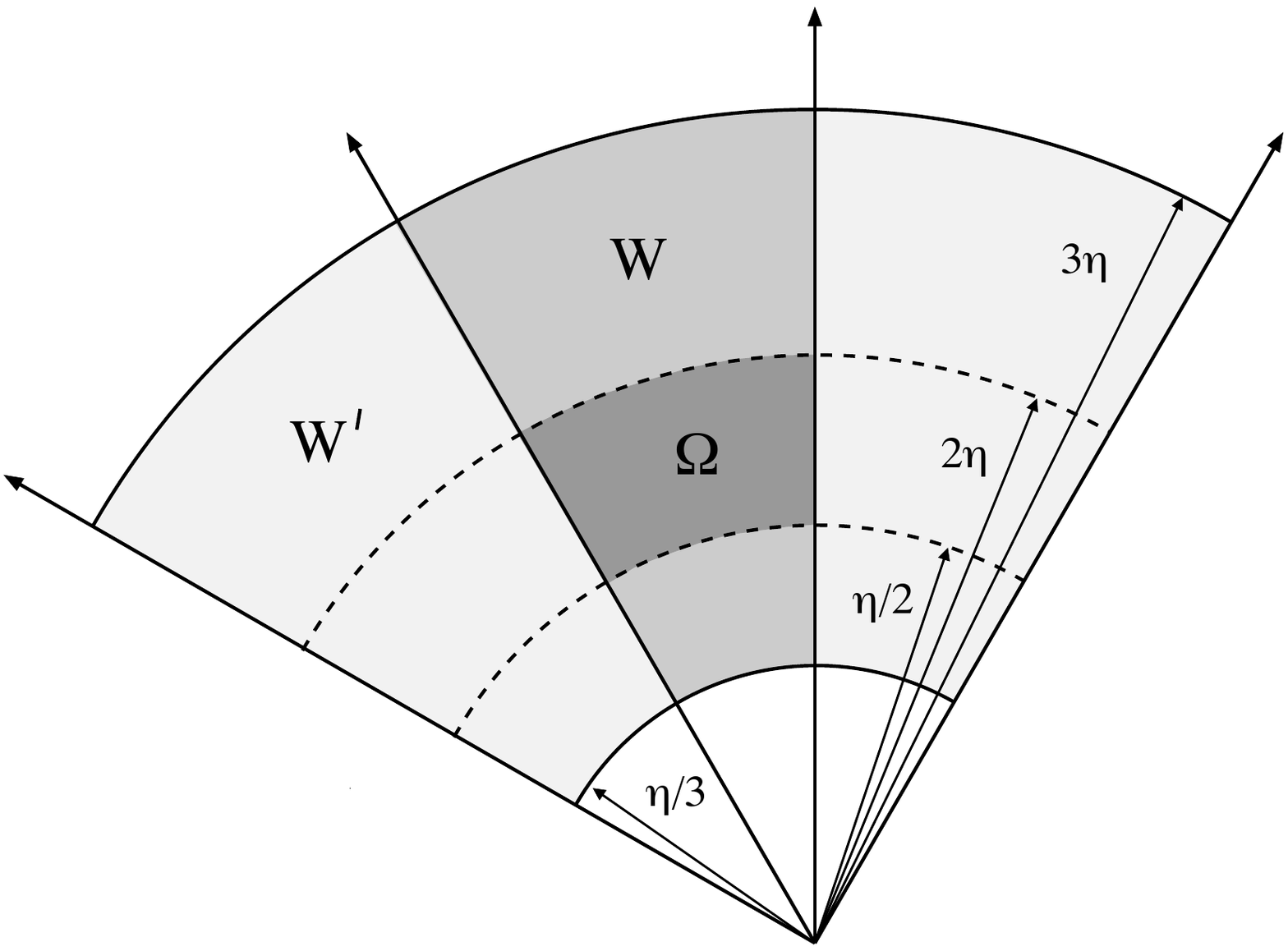}
}
\caption{
The functions $w_v$ and $u_{G,P}$ are harmonic 
on different parts of the  grey faces of $H$, 
but they are both defined on the intersection 
of these faces with the annuli $A(v, \eta)$ 
around each vertex.  On $\Omega$, a connected component
of this intersection,  we combine them
using a partition of unity.
Since both  functions extend analytically to a 
neighborhood $W'$ of $\overline{\Omega}$,
and  are very close there, 
the Cauchy estimates imply their   partial derivatives
are close in $\Omega$, and thus the merged function is 
``nearly'' harmonic.
}
\label{CauchyEst}
\end{figure}

To simplify notation, assume $v=0$.
By our remarks above, both the functions $f_1, f_2$
are very close to a multiple of 
$z^d$, and hence we can write 
$$ f_1(z) = cz^d(1+ \varepsilon_1(z)), 
 \quad f_2(z) = cz^d(1+ \varepsilon_2(z)),$$
where $\varepsilon_1$ and $\varepsilon_2$ are holomorphic functions 
on  $W'$ that are both as small 
as we wish,  say $|\varepsilon_1|, |\varepsilon_2| < \varepsilon$. 
By the Cauchy estimates, $|\varepsilon_1'|,|\varepsilon_2'|
=O(\varepsilon/\eta)$ 
on $\Omega$. 
Take a smooth, decreasing function $\varphi$ on $[0,\infty) \to [0,1]$
so that $\varphi =1$ on $[0, \eta/2]$ and $\varphi=0$ on $[2 \eta, \infty)$ 
and extend it to the plane by $\varphi(z) := \varphi(|z|)$.
We can choose $\varphi$ so that its partial derivatives are bounded 
by $O(\eta^{-1})$.
Define 
\begin{eqnarray*} 
F(z) 
&=&  \varphi(z) f_1(z)  + (1-\varphi(z)) f_2(z) \\
&=&  \varphi(z) c  z^d (1+\varepsilon_1(z))  
	+ (1-\varphi(z))  c z^d (1+\varepsilon_2(z)) \\
	&=&   c \cdot  z^d [ 1+\varepsilon_2(z)  +
	\varphi(z)(\varepsilon_1(z) -\varepsilon_2(z))].
\end{eqnarray*}
Since $|z| \simeq \eta$ on $A(0, \eta)$ we have 
\begin{eqnarray*} 
\partial_{\zbar} F(z) 
	&=&    c z^d \left(\partial_\zbar  \varphi(z)
		(\varepsilon_1(z) -\varepsilon_2(z)) \right)
	=    c z^d O(\eta^{-1} \varepsilon)
	=    c z^{d-1}  O( \varepsilon)
\end{eqnarray*}
and 
\begin{eqnarray*} 
\partial_{z} F(z) 
&=& c\cdot d \cdot z^{d-1} 
	[1+ \varepsilon_2(z) + 
	\varphi(z)(\varepsilon_1(z) -\varepsilon_2(z))] \\
&& \qquad +  c z^d  
 [\partial_z \varepsilon_2(z) +
	\partial_z \varphi(z)(\varepsilon_1(z) -\varepsilon_2(z))
      + \varphi(z) (\partial_z \varepsilon_1(z) -\partial_z \varepsilon_2(z)) \\
&=& c\cdot d \cdot z^{d-1}[1+   O(\varepsilon  ) 
	+ O( |z| ( \eta^{-1} \varepsilon)) ] \\
&=& c \cdot  d \cdot  z^{d-1}[1+   O(\varepsilon  ) ],
\end{eqnarray*}
so 
\begin{eqnarray*} 
	\frac {\partial_{\zbar} F} { \partial_{z} F }(z) 
&=&  \frac{  c \cdot  z^{d-1} O( \varepsilon)}
{ c \cdot d z^{d-1}[1+   O(\varepsilon  ) ]}  
=  O( \varepsilon).
\end{eqnarray*}
Next we use the measurable Riemann mapping theorem to find a 
quasiconformal map $\psi$ of the plane to itself whose dilatation 
is $F_{\zbar}/{F_z} $ on $\Omega$ and zero elsewhere. Then 
$\psi$ is conformal off $\Omega$ and   $F\circ \psi^{-1}$ is 
holomorphic on $\Omega$.
The dilatation of $\psi$ is
bounded  $O(\varepsilon)$, which is as small as we wish. 
Thus we can take  $\psi$ to be an $\varepsilon$-homeomorphism 
for any $\varepsilon >0$ that we want.

Now define 
$$
  V = \begin{cases}
  \re(F \circ \psi^{-1} ), & \text{ on } \Omega, \\
  w_v \circ \psi^{-1},  & \text{ on }  B \cap \cup_v D(v, \eta/2), \\
   u_{G,P} \circ \psi^{-1},  & \text{ on }  B \setminus \cup_v D(v, 3\eta).
  \end{cases} 
  $$
  Then $V$ is positive and harmonic on the
grey faces of  $H' :=\psi(H)$ with  logarithmic poles 
at $P' := \psi(P)$, and it vanishes on $H'$
and on the white faces of $H'$ 
($V$ is harmonic in the regions of the form  $\Omega$ because 
the quasiconformal map $\psi$  was chosen to make this happen; 
elsewhere $\psi$ is conformal so  $V$ harmonic since  $\re(F)$  
is harmonic.) Thus $V = u_{H',P'}$. 
Also, $V$ has critical points at the images under  $\psi$ 
of the vertices of $G$, and the critical 
values of $V$ at these points are not 
changed by pre-composition with $\psi^{-1}$.
Thus $V $ has the same critical values as $\re(F)$, and  
so the degree of the critical level set (as a graph) at
its vertex $\psi(v) $ equals the degree of $v$ in $G$. Using this and 
Lemma \ref{small move graphs}, it is
easy to prove that the level set is $\varepsilon$-homeomorphic to $G$.

\vskip.2in
\noindent
{\bf Step 6: placing the poles precisely.}
The one problem that remains is that the new poles $P'=\psi(P)$ are 
not the same as $P$.  However, this can be dealt with as follows. 
Suppose $P$ has $n$ points and we think of this as 
a vector in $\complex^n$. Let $B(P,\sigma)$ be
a small ball of radius $\sigma>0$
around $P$ in $\complex^n$,
and for $Q \in B(P, \sigma)$
we repeat the construction above, obtaining a harmonic 
function $V_Q$ that has logarithmic poles at $Q':=\psi_Q (Q)$. In the 
construction, when we move $P$ continuously, the value of $u_{G,P}$
at each base point changes continuously, so our choices of the 
channel angles $(\theta_1, \dots, \theta_m)$ also change continuously. 
Thus the values of $u_{G, P}$ and its partial derivatives on 
each $A(v, \eta)$ 
change continuously with $P$, which implies the same for  
$F$  ($w_d$ and $\varphi$ do not change), 
and hence so does our solution of the Beltrami 
equation $\psi$ (properly normalized). 
Thus the dependence of $Q'$ on $Q$ is continuous,  and 
we can make  $|Q -Q'|$ as small as we wish, say $ <  \sigma/2$, 
by choosing $\varepsilon$ to be small enough.
Then the following lemma shows that  there exists a $Q$ so that $Q' = P$.

\begin{lem}
Let $ {\mathbb B}^n$  denote the unit ball  
in $\complex^n$ and suppose that   
$F: {\mathbb B}^n  \to \complex^n$ is a 
continuous map with the property that $|F(\myvec{x}) -
\myvec{x}| < 1/2$. Then $0 \in F({\mathbb B}^n )$. 
\end{lem}

\begin{proof}
Suppose $ 0 \not \in F({\mathbb B}^n)$. 
Choose a continuous, increasing function 
	$\phi:[0,1]\to [0,1]$ so that $\phi(t)=0 $ 
if $t \in [0, 1/2]$ and $\phi(1) =1$. Then 
$G(\myvec{x}) := (1-\phi(|\myvec{x}|)) F(\myvec{x}) 
+\phi(|\myvec{x}|) \myvec{x}$
is continuous on $\overline{\mathbb B^n}$, 
equals $F$ on $ \frac 12 {\mathbb B}^n$,
and is the identity on $\partial {\mathbb B}^n $. 
	If $|\myvec{x}| \leq  1/2$, then  
$G(\myvec{x}) = F(\myvec{x})\ne 0$. 
	If $|\myvec{x}| > 1/2$, then the ball $B(\myvec{x} , 1/2)$
does not contain $0$, but it does contain 
both $\myvec{x}$ and $F(\myvec{x})$, and 
hence it contains $G(\myvec{x})$, which is 
on the line segment from  $\myvec{x}$ to $F(\myvec{x})$.
Therefore $G$ is never zero on ${\mathbb B}^n$. 
Taking  $R(\myvec{x}) := \myvec{x}/|\myvec{x}|$ to be the 
radial projection of $\complex^n\setminus \{0\}$ 
	onto $\partial {\mathbb B}^n$,
we see that  $R \circ G: \overline{\mathbb B^n}
\to \partial {\mathbb B}^n$
is continuous and equals the identity on $\partial 
{\mathbb B}^n$, i.e., it is a retraction of ${\mathbb B}^n$
onto $\partial {\mathbb B}^n$. 
But following such a retraction by the 
antipodal map  $\myvec{x} \to -\myvec{x}$ on
$\partial {\mathbb B}^n$ gives a continuous map 
of the closed ball into itself with no 
fixed point, contradicting Brouwer's 
theorem.  Therefore  we must have $0 \in F({\mathbb B}^n)$.
\end{proof} 

This completes the proof of Theorem \ref{harmonic_app}, and 
hence the proof of Theorem \ref{main_thm_2}.


\begin{thebibliography}{MRW22}

\bibitem[Ahl06]{Ahlfors-QCbook}
Lars~V. Ahlfors.
\newblock {\em Lectures on quasiconformal mappings}, volume~38 of {\em
  University Lecture Series}.
\newblock American Mathematical Society, Providence, RI, second edition, 2006.
\newblock With supplemental chapters by C. J. Earle, I. Kra, M. Shishikura and
  J. H. Hubbard.

\bibitem[And00]{MR1775150}
Vladimir Andrievskii.
\newblock On the approximation of a continuum by lemniscates.
\newblock {\em J. Approx. Theory}, 105(2):292--304, 2000.

\bibitem[And18]{MR3978145}
V.~V. Andrievskii.
\newblock On {H}ilbert lemniscate theorem for a system of quasidisks.
\newblock {\em Jaen J. Approx.}, 10(1-2):133--145, 2018.

\bibitem[ANV22]{MR4416767}
Diana Andrei, Olavi Nevanlinna, and Tiina Vesanen.
\newblock Rational functions as new variables.
\newblock {\em Banach J. Math. Anal.}, 16(3):Paper No. 37, 22, 2022.

\bibitem[BLL08]{MR2473634}
Thomas Bloom, Norman Levenberg, and Yu. Lyubarskii.
\newblock A {H}ilbert lemniscate theorem in {$\Bbb C^2$}.
\newblock {\em Ann. Inst. Fourier (Grenoble)}, 58(6):2191--2220, 2008.

\bibitem[Bor95]{MR1223265}
Peter Borwein.
\newblock The arc length of the lemniscate {$\{|p(z)|=1\}$}.
\newblock {\em Proc. Amer. Math. Soc.}, 123(3):797--799, 1995.

\bibitem[BP15]{MR3420484}
Christopher~J. Bishop and Kevin~M. Pilgrim.
\newblock Dynamical dessins are dense.
\newblock {\em Rev. Mat. Iberoam.}, 31(3):1033--1040, 2015.

\bibitem[Bro11]{MR1511644}
L.~E.~J. Brouwer.
\newblock \"{U}ber {A}bbildung von {M}annigfaltigkeiten.
\newblock {\em Math. Ann.}, 71(1):97--115, 1911.

\bibitem[BT21]{MR4375923}
Luka Boc~Thaler.
\newblock On the geometry of simply connected wandering domains.
\newblock {\em Bull. Lond. Math. Soc.}, 53(6):1663--1673, 2021.

\bibitem[CG93]{MR1230383}
Lennart Carleson and Theodore~W. Gamelin.
\newblock {\em Complex dynamics}.
\newblock Universitext: Tracts in Mathematics. Springer-Verlag, New York, 1993.

\bibitem[CP91]{MR1133876}
Fabrizio Catanese and Marco Paluszny.
\newblock Polynomial-lemniscates, trees and braids.
\newblock {\em Topology}, 30(4):623--640, 1991.

\bibitem[EG02]{MR1888795}
A.~Eremenko and A.~Gabrielov.
\newblock Rational functions with real critical points and the {B}. and {M}.
  {S}hapiro conjecture in real enumerative geometry.
\newblock {\em Ann. of Math. (2)}, 155(1):105--129, 2002.

\bibitem[EH99]{MR1704189}
Alexandre Eremenko and Walter Hayman.
\newblock On the length of lemniscates.
\newblock {\em Michigan Math. J.}, 46(2):409--415, 1999.

\bibitem[EHL20]{MR4245584}
Michael Epstein, Boris Hanin, and Erik Lundberg.
\newblock The lemniscate tree of a random polynomial.
\newblock {\em Ann. Inst. Fourier (Grenoble)}, 70(4):1663--1687, 2020.

\bibitem[EHP58]{MR101311}
P.~Erd\H{o}s, F.~Herzog, and G.~Piranian.
\newblock Metric properties of polynomials.
\newblock {\em J. Analyse Math.}, 6:125--148, 1958.

\bibitem[EKS11]{MR2868587}
P.~Ebenfelt, D.~Khavinson, and H.~S. Shapiro.
\newblock Two-dimensional shapes and lemniscates.
\newblock In {\em Complex analysis and dynamical systems {IV}. {P}art 1},
  volume 553 of {\em Contemp. Math.}, pages 45--59. Amer. Math. Soc.,
  Providence, RI, 2011.

\bibitem[FBY15]{MR3296178}
Maxime Fortier~Bourque and Malik Younsi.
\newblock Rational {A}hlfors functions.
\newblock {\em Constr. Approx.}, 41(1):157--183, 2015.

\bibitem[FKV18]{MR3784168}
Anastasia Frolova, Dmitry Khavinson, and Alexander Vasil'ev.
\newblock Polynomial lemniscates and their fingerprints: from geometry to
  topology.
\newblock In {\em Complex analysis and dynamical systems}, Trends Math., pages
  103--128. Birkh\"{a}user/Springer, Cham, 2018.

\bibitem[FN09]{MR2500509}
Alexander Fryntov and Fedor Nazarov.
\newblock New estimates for the length of the {E}rd{\H o}s-{H}erzog-{P}iranian
  lemniscate.
\newblock In {\em Linear and complex analysis}, volume 226 of {\em Amer. Math.
  Soc. Transl. Ser. 2}, pages 49--60. Amer. Math. Soc., Providence, RI, 2009.

\bibitem[GM08]{MR2450237}
John~B. Garnett and Donald~E. Marshall.
\newblock {\em Harmonic measure}, volume~2 of {\em New Mathematical
  Monographs}.
\newblock Cambridge University Press, Cambridge, 2008.
\newblock Reprint of the 2005 original.

\bibitem[Gut13]{MR3202645}
Larry Guth.
\newblock Unexpected applications of polynomials in combinatorics.
\newblock In {\em The mathematics of {P}aul {E}rd\H{o}s. {I}}, pages 493--522.
  Springer, New York, 2013.

\bibitem[Hil97]{Hilbert1897}
D.~Hilbert.
\newblock \"uber die {Entwicklung} einer beliebigen analytischen {Funktion}
  einer variablen in eine unendliche nach ganzen rationalen {Funktionen}
  fortschreitende {Reihe}.
\newblock {\em Gottinger Nachrichten}, pages 63--70, 1897.

\bibitem[JR76]{MR435791}
Robert~E. Jamison and William~H. Ruckle.
\newblock Factoring absolutely convergent series.
\newblock {\em Math. Ann.}, 224(2):143--148, 1976.

\bibitem[Kir87]{MR0902292}
A.~A. Kirillov.
\newblock K\"{a}hler structure on the {$K$}-orbits of a group of
  diffeomorphisms of the circle.
\newblock {\em Funktsional. Anal. i Prilozhen.}, 21(2):42--45, 1987.

\bibitem[KL20]{MR4205641}
Sarah Koch and Tan Lei.
\newblock On balanced planar graphs, following {W}. {T}hurston.
\newblock In {\em What's next?---the mathematical legacy of {W}illiam {P}.
  {T}hurston}, volume 205 of {\em Ann. of Math. Stud.}, pages 215--232.
  Princeton Univ. Press, Princeton, NJ, 2020.

\bibitem[KLL{\etalchar{+}}]{Korda2022}
Milan Korda, Jean-Bernard Lasserre, Alexey Lazarev, Victor Magron, and Simon
  Naldi.
\newblock Urysohn in action: separating semialgebraic sets by polynomials.
\newblock preprint, ArXiv 2207:00570v1.

\bibitem[Kos05]{MR2246962}
O.~N. Kosukhin.
\newblock On the rate of approximation of closed {J}ordan curves by
  lemniscates.
\newblock {\em Mat. Zametki}, 77(6):861--876, 2005.

\bibitem[Laz23]{lazebnik2023analytic}
Kirill Lazebnik.
\newblock Analytic and topological nets, 2023.
\newblock preprint, arXiv 2307.16108.

\bibitem[Lin15]{MR3377290}
Kathryn~A. Lindsey.
\newblock Shapes of polynomial {J}ulia sets.
\newblock {\em Ergodic Theory Dynam. Systems}, 35(6):1913--1924, 2015.

\bibitem[LL15]{MR3356754}
Antonio Lerario and Erik Lundberg.
\newblock Statistics on {H}ilbert's 16th problem.
\newblock {\em Int. Math. Res. Not. IMRN}, (12):4293--4321, 2015.

\bibitem[LL16]{MR3570241}
Antonio Lerario and Erik Lundberg.
\newblock On the geometry of random lemniscates.
\newblock {\em Proc. Lond. Math. Soc. (3)}, 113(5):649--673, 2016.

\bibitem[LR17]{MR3742436}
Erik Lundberg and Koushik Ramachandran.
\newblock The arc length and topology of a random lemniscate.
\newblock {\em J. Lond. Math. Soc. (2)}, 96(3):621--641, 2017.

\bibitem[LV73]{MR344463}
O.~Lehto and K.~I. Virtanen.
\newblock {\em Quasiconformal mappings in the plane}.
\newblock Die Grundlehren der mathematischen Wissenschaften, Band 126.
  Springer-Verlag, New York-Heidelberg, second edition, 1973.
\newblock Translated from the German by K. W. Lucas.

\bibitem[LY19]{MR3955554}
Kathryn~A. Lindsey and Malik Younsi.
\newblock Fekete polynomials and shapes of {J}ulia sets.
\newblock {\em Trans. Amer. Math. Soc.}, 371(12):8489--8511, 2019.

\bibitem[Mar19]{MR4321146}
Donald~E. Marshall.
\newblock {\em Complex analysis}.
\newblock Cambridge Mathematical Textbooks. Cambridge University Press,
  Cambridge, 2019.

\bibitem[Mil97]{MR1487640}
John~W. Milnor.
\newblock {\em Topology from the differentiable viewpoint}.
\newblock Princeton Landmarks in Mathematics. Princeton University Press,
  Princeton, NJ, 1997.
\newblock Based on notes by David W. Weaver, Revised reprint of the 1965
  original.

\bibitem[Mil06]{MilnorCDBook}
John Milnor.
\newblock {\em Dynamics in one complex variable}, volume 160 of {\em Annals of
  Mathematics Studies}.
\newblock Princeton University Press, Princeton, NJ, third edition, 2006.

\bibitem[MP19]{MR3964612}
V.~N. Malozemov and A.~V. Plotkin.
\newblock Strict polynomial separation of two sets.
\newblock {\em Vestn. St.-Peterbg. Univ. Mat. Mekh. Astron.},
  6(64)(2):232--240, 2019.

\bibitem[MRW22]{2022arXiv220411781M}
David {Mart{\'\i}-Pete}, Lasse {Rempe}, and James {Waterman}.
\newblock {Bounded Fatou and Julia components of meromorphic functions}.
\newblock {\em arXiv e-prints}, page arXiv:2204.11781, April 2022.

\bibitem[Niv09]{MR2518003}
St\'{e}phanie Nivoche.
\newblock Polynomial convexity, special polynomial polyhedra and the
  pluricomplex {G}reen function for a compact set in {$\Bbb C^n$}.
\newblock {\em J. Math. Pures Appl. (9)}, 91(4):364--383, 2009.

\bibitem[NT05]{MR2177185}
B\'{e}la Nagy and Vilmos Totik.
\newblock Sharpening of {H}ilbert's lemniscate theorem.
\newblock {\em J. Anal. Math.}, 96:191--223, 2005.

\bibitem[Pom92]{MR1217706}
Ch. Pommerenke.
\newblock {\em Boundary behaviour of conformal maps}, volume 299 of {\em
  Grundlehren der Mathematischen Wissenschaften [Fundamental Principles of
  Mathematical Sciences]}.
\newblock Springer-Verlag, Berlin, 1992.

\bibitem[Put05]{MR2129650}
Mihai Putinar.
\newblock Notes on generalized lemniscates.
\newblock In {\em Operator theory, systems theory and scattering theory:
  multidimensional generalizations}, volume 157 of {\em Oper. Theory Adv.
  Appl.}, pages 243--266. Birkh\"{a}user, Basel, 2005.

\bibitem[Ric16]{MR3558373}
Trevor Richards.
\newblock Conformal equivalence of analytic functions on compact sets.
\newblock {\em Comput. Methods Funct. Theory}, 16(4):585--608, 2016.

\bibitem[Ric21]{MR4261771}
Trevor~J. Richards.
\newblock Some recent results on the geometry of complex polynomials: the
  {G}auss-{L}ucas theorem, polynomial lemniscates, shape analysis, and
  conformal equivalence.
\newblock {\em Complex Anal. Synerg.}, 7(2):Paper No. 20, 9, 2021.

\bibitem[Roy54]{MR64147}
H.~L. Royden.
\newblock The interpolation problem for schlicht functions.
\newblock {\em Ann. of Math. (2)}, 60:326--344, 1954.

\bibitem[RY17]{MR3590700}
Trevor Richards and Malik Younsi.
\newblock Conformal models and fingerprints of pseudo-lemniscates.
\newblock {\em Constr. Approx.}, 45(1):129--141, 2017.

\bibitem[RY19]{MR3922299}
Trevor Richards and Malik Younsi.
\newblock Computing polynomial conformal models for low-degree {B}laschke
  products.
\newblock {\em Comput. Methods Funct. Theory}, 19(1):173--182, 2019.



\bibitem[RZ11]{MR2838245}
Alexander Rashkovskii and Vyacheslav Zakharyuta,
\newblock Special polyhedra for {R}einhardt domains
\newblock {\em Comptes Rendus Math\'{e}matique}, 349(1):965--968, 2011.


\bibitem[SM06]{Sharon-Mumford}
E.~Sharon and D.~Mumford.
\newblock {2D-Shape} analysis using conformal mapping.
\newblock {\em Int. J. Comput. Vis.}, 70:55--75, 2006.

\bibitem[ST42]{MR7036}
A.~H. Stone and J.~W. Tukey.
\newblock Generalized ``sandwich'' theorems.
\newblock {\em Duke Math. J.}, 9:356--359, 1942.

\bibitem[Sta75]{MR428074}
James~D. Stafney.
\newblock Set approximation by lemniscates and the spectrum of an operator on
  an interpolation space.
\newblock {\em Pacific J. Math.}, 60(2):253--265, 1975.

\bibitem[Ste86]{MR825840}
Kenneth Stephenson.
\newblock Analytic functions sharing level curves and tracts.
\newblock {\em Ann. of Math. (2)}, 123(1):107--144, 1986.

\bibitem[Sul85]{MR819553}
Dennis Sullivan.
\newblock Quasiconformal homeomorphisms and dynamics. {I}. {S}olution of the
  {F}atou-{J}ulia problem on wandering domains.
\newblock {\em Ann. of Math. (2)}, 122(3):401--418, 1985.

\bibitem[Thu10]{38274}
Bill Thurston.
\newblock What are the shapes of rational functions?
\newblock MathOverflow, 2010.
\newblock URL:https://mathoverflow.net/q/38274 (version: 2017-04-13).

\bibitem[Tom15]{MR3422731}
J.~Tomasini.
\newblock Realizations of branched self-coverings of the 2-sphere.
\newblock {\em Topology Appl.}, 196(part A):31--53, 2015.

\bibitem[WR34]{MR1501732}
J.~L. Walsh and Helen~G. Russell.
\newblock On the convergence and overconvergence of sequences of polynomials of
  best simultaneous approximation to several functions analytic in distinct
  regions.
\newblock {\em Trans. Amer. Math. Soc.}, 36(1):13--28, 1934.

\bibitem[You16]{MR3447662}
Malik Younsi.
\newblock Shapes, fingerprints and rational lemniscates.
\newblock {\em Proc. Amer. Math. Soc.}, 144(3):1087--1093, 2016.

\bibitem[Zak21]{MR4422101}
Saeed Zakeri.
\newblock {\em A course in complex analysis}.
\newblock Princeton University Press, Princeton, NJ, 2021.

\end{thebibliography}
\end{document}